%% file: main.tex
\documentclass[11pt,reqno]{amsart}
\usepackage{amsbsy,amsfonts,amsmath,amssymb,amscd,amsthm,mathrsfs}
\usepackage{mathtools}
\usepackage{xfrac,float,enumitem,graphicx,epstopdf,pxfonts,srcltx}
\IfFileExists{dsfont.sty}{\usepackage{dsfont}}{\providecommand{\mathds}{\mathbf}}
\usepackage[mathcal]{euscript}
\usepackage[a4paper,text={6.1in,8in},centering]{geometry}
\usepackage[colorlinks=true,linkcolor=blue,citecolor=green]{hyperref}
\usepackage[margin=20pt,font=small,labelfont=bf,textfont=it]{caption}
\usepackage{subcaption,inputenc,tikz,pgfplots}
\pgfplotsset{compat=1.17}
\hypersetup{linktocpage}
\usepackage{cancel}
\usepackage{needspace}

\small\normalsize
\theoremstyle{plain}
\newtheorem{mainthm}{Theorem}
\newtheorem{mainprop}{Proposition}
\newtheorem{maincor}[mainprop]{Corollary}
\newtheorem{thm}{Theorem}[section]
\newtheorem{cor}[thm]{Corollary}
\newtheorem{lem}[thm]{Lemma}
\newtheorem{claim}{Claim}
\numberwithin{claim}{thm}
\newtheorem{prop}[thm]{Proposition}

\newtheorem{defi}[thm]{Definition}

\theoremstyle{definition}

\newtheorem{rem}[thm]{Remark}
\newtheorem{question}{Question}

\numberwithin{equation}{section}

\newcommand{\rmref}[1]{\textup{\hyperref[#1]{\ref*{#1}}}}
 \DeclareSymbolFont{letters}{OML}{cmr}{m}{it}
\DeclareMathSymbol{\oldsigma}{\mathord}{letters}{"1B}
 \DeclareMathSymbol{\stdtau}{\mathord}{letters}{"1C}
 \renewcommand{\sigma}{\stdtau}

\renewcommand{\labelenumi}{\theenumi)}

\newcommand{\eqdef}{\stackrel{\scriptscriptstyle\rm def}{=}}
\newcommand{\aep}{$\mathrm{a.e.}$}

\DeclareMathOperator{\lbb}{Leb}

\newcommand{\R}{\mathbb{R}}
\newcommand{\Z}{\mathbb{Z}}
\newcommand{\N}{\mathbb{N}}
\newcommand{\fu}[3]{#1\colon#2\to#3}

\newcommand{\pa}[1]{\bigg(#1\bigg)}

\newcommand{\la}[1]{\big\{#1\big\}}
\newcommand{\Esp}{\Omega}
\newcommand{\C}{\Esp\times I}

\newcommand{\s}{\dfrac{1}{n}\sum\limits_{j=0}^{n-1}}

\newcommand{\g}{\gamma}

\newcommand{\f}{\varphi}

\newcommand{\PP}{\mathbb{P}}
\newcommand{\ind}{\mathds{1}}

\makeatletter
\def\l@part{\@tocline{0}{-2pt}{1pc}{}{}}
\def\l@section{\@tocline{1}{-2pt}{1pc}{4.6em}{}}
\renewcommand{\tocpart}[3]{%
  \indentlabel{\@ifnotempty{#2}{\makebox[2.3em][l]{%
    \ignorespaces#1 #2.\hfill}}}\bf{#3}}
\renewcommand{\tocsection}[3]{%
  \indentlabel{\@ifnotempty{#2}{\hspace*{2.3em}\makebox[2.3em][l]{%
    \ignorespaces#1 #2.\hfill}}}#3}
\makeatother

\let\oldtocsection=\tocsection
\let\oldtocsubsection=\tocsubsection
\renewcommand{\tocsection}[2]{\hspace{0em}\bf\oldtocsection{#1}{#2}}
\renewcommand{\tocsubsection}[2]{\hspace{4.8em}\oldtocsubsection{#1}{#2}}
\let\oldtocsubsubsection=\tocsubsubsection
\renewcommand{\tocsubsubsection}[2]{\hspace{4.2em}\oldtocsubsubsection{#1}{#2}}

 \makeatletter
\renewcommand\subsubsection{\@startsection{subsubsection}{3}{\z@}%
  {-3.25ex \@plus -1ex \@minus -.2ex}
  {1.5ex \@plus .2ex}
  {\normalfont\bfseries\itshape}}
\makeatother

\begin{document}
~\vspace{-3cm}
\title[Historical behavior \& Arcsine laws]{Historical behavior of skew products and arcsine laws}

\begin{abstract}
We study the occurrence of historical behavior for almost every point in the setting of skew products with one-dimensional fiber dynamics. Under suitable ergodic conditions, we establish that a weak form of the arcsine law leads to the non-convergence of Birkhoff averages along almost every orbit. As an application, we show that this phenomenon occurs for one-step skew product maps over a Bernoulli shift, where the stochastic process induced by the iterates of the fiber maps is conjugate to a random walk.

Furthermore, we revisit known examples of skew products that exhibit historical behavior almost everywhere, verifying that they fulfill the required ergodic and probabilistic conditions.
Consequently, our results provide a unified and generalized framework that connects such behaviors to the arcsine distribution of the orbits.

~\vspace{-1cm}

\end{abstract}

\author[Barrientos]{Pablo G.~Barrientos}
\address{\centerline{Instituto de Matem\'atica e Estat\'istica, UFF}
    \centerline{Rua M\'ario Santos Braga s/n - Campus Valonguinhos, Niter\'oi,  Brazil}}
\email{pgbarrientos@id.uff.br}
\author[Chavez]{Raul R.~Chavez}
\address{\centerline{Departamento de Matem\'atica PUC-Rio}
\centerline{Marqu\^es de S\~ao Vicente 225, G\'avea, Rio de Janeiro 22451-900, Brazil}}
\email{r.rodriguezchav@aluno.puc-rio.br
}

\maketitle
\thispagestyle{empty}

\input{Sec0}

\input{Sec1}

\input{Sec2}

\input{Sec4}

\input{Sec5}

\input{Sec6}

\input{Sec7}

\input{Sec8}

\input{Sec9}

\input{Sec10}

\bibliographystyle{alpha3}
\bibliography{reference}

\end{document}

%% file: Sec0.tex
\section{Introduction}
The dynamics of a function $f$ on a compact metric space $X$ is often studied through the long-term statistical behavior of its orbits. Given a continuous observable $\phi \colon X \to \mathbb{R}$, the orbit of $x$ exhibits \emph{historical behavior} if its \emph{time average} (or \emph{Birkhoff average})
\[
\frac{1}{n}\sum_{j=0}^{n-1}\phi(f^j(x))
\]
does not converge. Otherwise, the orbit exhibits \emph{predictable behavior}, a term introduced by Ruelle~\cite{Rue2001} and further developed by Takens~\cite{Takens2008}. An important case of predictable behavior occurs when an ergodic invariant probability measure~$\mu$ has a basin of positive reference measure, i.e., a positive-measure set of points whose time averages converge to $\int \phi \, d\mu$ for every continuous $\phi$. Such measures are called \emph{physical measures}. By Birkhoff's ergodic theorem, the set of \emph{irregular points} (i.e., orbits exhibiting historical behavior) has zero measure for every invariant probability measure. Nevertheless, this set may still be metrically large. In many known classes, the set of irregular points is residual~\cite{Takens2008,BKNRS20,CarVaran22}. There are also examples where it has full topological entropy and Hausdorff dimension~\cite{BS00,FFW01,BNRR22}, although in most of these examples it has zero Lebesgue measure. A paradigmatic configuration producing historical behavior for Lebesgue almost every point is the \emph{Bowen Eye} in~\cite{Takens-eyebowe94}.
Motivated by this model, Takens posed in~\cite{Takens2008} the following problem:
\begin{enumerate}[leftmargin=0.8cm,rightmargin=0.8cm]
	\item[] \textbf{\emph{Takens' last problem}:} Are there persistent classes of smooth dynamical systems that have a set of irregular points with positive Lebesgue measure?
\end{enumerate}

Kiriki--Soma~\cite{KS17} gave the first class of surface $C^2$ diffeomorphisms for which historical behavior cannot be eliminated by discarding negligible sets of systems and initial conditions. Their construction, based on wandering domains, was later extended to higher (analytic) regularity~\cite{BB23}, higher dimensions ($d\geq 3$) and lower ($C^1$) regularity~\cite{Bar22}, and flows~\cite{LR17}. More recently, historical behavior for Lebesgue almost every point has been obtained for families of rational maps of degree $d\geq 2$ on the Riemann sphere~\cite{talebi2022non}, reparameterizations of linear flows on the two torus~\cite{Andde22}, a class of transitive partially hyperbolic diffeomorphisms on $\mathbb{T}^3$~\cite{Crovisier2020}, and certain one-dimensional maps~\cite{AarThaZwei05,CoaLuza23,CoaMelTale24}. Such \emph{non-statistical} systems (i.e., with historical behavior almost everywhere) have no physical measure, so long-term statistical predictions fail on a significant set of initial conditions.
%\enlargethispage{1cm}

We approach this phenomenon differently, through an abstract framework linking historical behavior almost everywhere to {underlying ergodic and}  probabilistic properties. We relate it to probabilistic mechanisms such as the \emph{arcsine law}.
We place all the known skew-product constructions in a unified setting that clarifies the underlying mechanism and extends it to new examples. We view the arcsine law as complementary to classical mechanisms (based on wandering domains or Bowen-type dynamics) and suggest that probabilistic ideas may provide a new pathway toward a positive resolution of Takens’ last problem.

\subsection{Presentation of results and context} We consider the class of \emph{skew products}
\begin{equation} \label{skew product-principala}
	\fu{F}{\Omega\times M}{\Omega\times M}, \qquad  F(\omega, x) \eqdef (\sigma(\omega), f_{\omega}(x))
\end{equation}
where the base function $\fu{\sigma}{\Omega}{\Omega}$ is an ergodic measure-preserving map on a standard probability space $(\Omega, \mathscr{F},\PP)$, and the fiber maps $\fu{f_\omega}{M}{M}$ are measurable functions on a compact one-dimensional manifold $M$ endowed with the Lebesgue measure $\lbb$.  We denote the compositions of the fiber maps by
\begin{equation} \label{eq:fiber-dynamics}
f_\omega^0 \eqdef \mathrm{id} \quad \text{and} \quad f_\omega^n \eqdef f_{\sigma^{n-1}(\omega)} \circ \dots \circ f_{\sigma(\omega)} \circ f_\omega \quad n\geq1.
\end{equation}
Taking the base $\Omega$ to be a singleton,~\eqref{skew product-principala} includes one-dimensional dynamics.  We study conditions for historical behavior of $F$ at $(\PP\times\lbb)$-almost every point. We write \emph{$(\PP\times\lbb)$-a.e.}~(or~\emph{almost everywhere}). Such systems do not admit physical measures with respect to the reference measure $\mathbb{P}\times \lbb$. See Section~\ref{ss:def-hb} for the precise definition.

We relate historical behavior to the arcsine law. Recall that a random variable $Y$ on $[0,1]$ is~\mbox{\emph{arcsine~distributed}~if,}
$$
    \PP(Y \leq \alpha) = \frac{2}{\pi} \arcsin\sqrt{\alpha}, \quad \text{for every $\alpha \in [0,1]$.}
$$
L\'evy~\cite{Lev39} showed that the positive occupation time of one-dimensional Brownian motion has this distribution. Later, Erdős and Kac~\cite{ErdKac47} formalized an asymptotic version for
independent and identically distributed (i.i.d.)~random variables.
That is, $\{Y_n\}_{n\geq1}$ is \emph{asymptotically arcsine distributed} if
$$
\lim_{n\to \infty} \PP(Y_n \leq \alpha) = \frac{2}{\pi} \arcsin \sqrt{\alpha} \quad \text{for every $\alpha \in [0,1]$.}
$$
In particular, for $\alpha\in(0,1)$,
$$
{\liminf_{n \to \infty} \mathbb{P}(Y_n \leq \alpha) < 1 \quad \text{and} \quad \liminf_{n \to \infty} \mathbb{P}(Y_n \geq \alpha) < 1.}
$$
These conditions prevent concentration at either endpoint and define a \emph{fluctuation law}.
This weak arcsine law and an ergodic condition imply historical behavior almost everywhere for skew products as in~\eqref{skew product-principala}. These conditions differ slightly for \emph{one-step} and \emph{mild} skew products. For \emph{mild} skew products the fiber maps depend on all of $\omega\in\Omega$, whereas in the \emph{one-step} case the base is Bernoulli and the fiber map depends only on the zero-coordinate.

 For one-step systems, Molinek~\cite{molinek1994asymptotic} used the Erdős--Kac arcsine law to obtain historical behavior almost everywhere for generalized $(T,T^{-1})$-transformations with $T$ north-south. Bonifant and Milnor~\cite{Bonifant}  used the arcsine law to sketch the same conclusion for one-step maps with zero-Schwarzian fiber maps. Theorem~\ref{thm:B-random-walks} gives this conclusion for one-step skew products satisfying a weak arcsine law, where the ergodic condition is the triviality of the tail $\oldsigma$-algebra of the iterated fiber process. Namely, if this process is conjugate to a one-dimensional random walk with mean zero and positive finite variance, then $F$ exhibits historical behavior almost everywhere. We show that the examples of Bonifant--Milnor and Molinek are conjugate to random walks. New examples include the $T^\Psi$-transformations, which generalize the $(T,T^{-1})$-transformations, and a new non-trivial coupling of random~walks.

Theorem~\ref{thm:ergodic-assumption} treats mild skew products under ergodicity of the non-invariant reference measure $\mathbb P\times\lbb$ and a weak arcsine law. The earliest such examples are the \emph{skew-flows} of Ji and Molinek~\cite{JM00}. A \emph{skew-flow} is a skew product as in~\eqref{skew product-principala}, where~$f_\omega(x)=\varphi(\phi(\omega),x)$, with $\varphi:\mathbb{R}\times M\to M$ a flow and $\phi:\Omega\to\mathbb{R}$ a roof function. A partially hyperbolic diffeomorphism on $\mathbb{T}^3$ by Crovisier et al.~\cite{Crovisier2020} also belongs to this class. These examples are conjugate to skew-translations, through which we establish  vanishing interior occupation times and the arcsine~law. Our result also yields new non-statistical diffeomorphisms on $\mathbb{T}^{n+1}$ for $n\geq2$ and proves a conjecture of Bonifant--Milnor~\cite{Bonifant} for skew product endomorphisms of~$\mathbb{T}\times[0,1]$.

Theorem~\ref{thm:ergodic-assumption} covers Thaler functions~\cite{Thaler80,Thaler83,Thaler02}, which are full branch maps  conjugate to uniformly expanding doubling functions. Aaronson et al.~\cite{AarThaZwei05} first observed that these functions exhibit historical behavior almost everywhere. This was extended to generalized Thaler functions by Coates and Luzzatto~\cite{CoaLuza23} and Coates~et~al.~\cite{CoaMelTale24}.

To our knowledge, our results unify all skew-product examples with one-dimensional fibers exhibiting historical behavior almost everywhere. We also generalize these examples and obtain new classes of non-statistical systems.
 \vspace{-0.3cm}

\subsection{Historical behavior from random walks} \label{ss:skew product}
We study skew product maps as in~\eqref{skew product-principala}, where $M=I=[0,1]$, with the Bernoulli shift $\sigma\colon \Esp\to\Esp$ on $(\Esp,\mathscr{F},\mathbb{P})=(\mathcal{A}^\N,\mathcal{F}^\mathbb{N},p^\N)$ as base. Here, $\mathcal{A}$ is an at most countable alphabet, $\mathcal{F}$ its discrete $\oldsigma$-algebra, and $p$ a probability measure with full support, i.e., $p(a)>0$ for every $a\in\mathcal{A}$. We also assume that $F$ is \emph{one-step}~(or \emph{locally constant}), that is,
\begin{equation}\label{one-skew product-principal}
	 F(\omega, x) = (\sigma(\omega), f_{\omega_0}(x)), \quad \text{for every $\omega = (\omega_i)_{i \geq 0} \in \Omega$.}
\end{equation}
Here, the compositions of the fiber maps in~\eqref{eq:fiber-dynamics} can be written as follows:
\begin{equation} \label{eq:fiber-maps-one-step}
    f^0_\omega \eqdef \mathrm{id} \quad \text{and} \quad f^n_\omega \eqdef f_{\omega_{n-1}} \circ \dots \circ f_{\omega_0} \quad \text{for $\omega \in \Esp$ and $n \geq 1$}.
\end{equation}
We also consider the following conditions on the fiber maps:
\begin{enumerate}[leftmargin=1.25cm,label=(H\arabic*),start=0]
    \item\label{H0} For every $x \in (0,1)$, the fiber process $X^x_n(\omega)=f^n_\omega(x)$ has trivial tail $\oldsigma$-algebra
    $\mathcal T(\{X^x_n\})\eqdef\bigcap_{m\geq1}\oldsigma(X^x_m,X^x_{m+1},\ldots)$; that is, every $A\in\mathcal T(\{X^x_n\})$ has $\PP(A)\in\{0,1\}$.
    
    \item\label{H1} For every $\omega \in \Omega$, $f_\omega\colon (0,1) \to (0,1)$ is a monotonically increasing measurable map.
    \item\label{H2} For every $x \in (0,1)$, there exist $\alpha, \beta \in \Esp$  such that {$f_\alpha(x) < x < f_\beta(x)$.}
\end{enumerate}

Condition~\ref{H0} is a strong ergodic assumption, satisfied whenever the fiber process is conjugate to a random walk  (see Definition~\ref{def:conjugated-G-randomwalk} and Proposition~\ref{prop:conjugate-H0}). Condition~\ref{H1} requires order preservation, but neither continuity nor fixed endpoints. Condition~\ref{H2} forces interaction between the dynamics at the two ends of $I$, analogous to the heteroclinic connections between the two equilibrium points of the Bowen eye.

\begin{defi}\label{def:arcsine-law-one-s}
The system $F$ satisfies the
{\emph{pointwise-fiber fluctuation law}}  if there exist constants $\gamma_0,\gamma_1 \in (0,1)$ such that for every $\alpha \in (0,1)$ there are points $x_0, x_1\in (0,1)$ satisfying
\begin{align*}
    \liminf_{n \to \infty} \PP\bigg( \frac{1}{n}\sum_{j=0}^{n-1} \ind_{I_i(\gamma_i)}(f^j_\omega(x_i)) \leq \alpha \bigg) &< 1, \quad \text{for $i = 0, 1$}
\end{align*}
where $I_0(\gamma) \eqdef [0, \gamma]$ and $I_1(\gamma) \eqdef [\gamma, 1]$ for any $\gamma \in I$.
\end{defi}

The next theorem establishes historical behavior almost everywhere for one-step skew products satisfying the pointwise-fiber fluctuation law.

\begin{mainthm}\label{thm:B-random-walks}
    Let $F$ be a one-step skew product as in~\eqref{one-skew product-principal}, satisfying conditions~{\rmref{H0}--\rmref{H2}} and the {pointwise-fiber fluctuation law with constants $\gamma_0<\gamma_1$}. Then, for every $x\in (0,1)$,
    \begin{equation} \label{eq:liminf-limsup}
             \liminf_{n\to\infty} \frac{1}{n}\sum_{j=0}^{n-1} f^j_\omega(x)\leq \gamma_0 < \gamma_1 \leq \limsup_{n\to\infty} \frac{1}{n}\sum_{j=0}^{n-1} f^j_\omega(x)  \quad \text{for $\mathbb{P}$-a.e.~$\omega\in \Omega$}.
    \end{equation}
    In particular,
    $F$ exhibits historical behavior for $(\PP \times \lbb)$-almost every point.
\end{mainthm}

To connect~\ref{H0} and the {pointwise-fiber fluctuation law} with random walks, let $\{Y_n\}_{n\ge1}$ be~i.i.d.~with values in $G=\mathbb R$ or $\mathbb Z$. A \emph{random walk} starting at $t\in G$ is the process~$\{{S}_n\}_{n \geq 0}$ given by
\begin{equation}\label{def:random-walk}
{S}_0 = t \quad \text{and} \quad {S}_{n} = {S}_{n-1} + Y_n, \quad n \geq 1.
\end{equation}
The steps $Y_n$ have mean and variance
\[
\mu = \mathbb{E}[Y_1] \eqdef \int Y_1 \, d\PP \quad \text{and} \quad \oldsigma^2 = \mathbb{E}[(Y_1 - \mu)^2].
\]

For $x\in I$, set $X_n^x(\omega)=f_\omega^n(x)$. We denote by {
 $\mathcal{O}(x)= \{X^x_n(\omega): \omega\in \Omega, \, n\geq 0\}$ the orbit of $x$ by the semigroup  generated by
 $\{f_{\omega_0}\}_{\omega_0\in\mathcal{A}}$. }

\begin{defi}\label{def:conjugated-G-randomwalk} The sequence
 $\{X^x_n\}_{n \geq 0}$ is \emph{conjugate to a $G$-valued random walk} if there exists a {strictly monotonic injection}  $h\colon \mathcal{O}(x) \to G$ such that the step random variables
\[
Y_n^t(\omega) = {S^t_n(\omega) - S^t_{n-1}(\omega)}
\quad \text{for $n \geq 1$ and $\omega \in \Esp$},
\]
are independent and identically distributed, where $t = h(x)$ and {$S^t_n(\omega) = (h \circ f^n_\omega \circ h^{-1})(t)$.} {In other words, there exists a random walk $\{S^t_n\}_{n\geq 0}$ on $G$ starting at $t=h(x)$ such that $S^t_n=h(X_n^x)$ for  $n\geq 1$.}
\end{defi}

The Hewitt--Savage zero--one law~\cite{HeSavage55-01law} gives~\ref{H0} whenever $\{X_n^x\}$ is conjugate to a random walk. Erdős--Kac~\cite{ErdKac47} then gives the {pointwise-fiber fluctuation law} through the conjugacy. This gives the following consequence
of Theorem~\ref{thm:B-random-walks}, proved in Section~\ref{ss:Proof-cor-I-II-VI}.

\begin{mainprop}\label{maincor:conjugation-random-walk}
    Let $F$ be a one-step skew product as in~\eqref{one-skew product-principal}. Assume that for every $x \in (0,1)$, the sequence $\{X^x_n\}_{n \geq 0}$, where $X^x_n(\omega) = f^n_\omega(x)$, is {conjugate to a random walk on $\mathbb{Z}$ or $\mathbb{R}$ with mean zero and positive finite variance.}
    Then, $F$ satisfies~\rmref{H0}, and for every $ x \in (0,1)$ and   $\gamma\in (\inf\mathcal O(x),\sup\mathcal O(x))$, 
    \begin{equation} \label{eq:arcsine-law}
        \lim_{n \to \infty} \PP\bigg(\frac{1}{n}\sum_{j=0}^{n-1} \ind_{{I_i(\gamma)}}(f^j_\omega(x)) \leq \alpha\bigg) = \frac{2}{\pi} \arcsin \sqrt{\alpha}, \quad \text{for every $\alpha \in (0,1)$ and $i = 0, 1$}.
    \end{equation}
    {Moreover, if $F$ also satisfies  conditions~\rmref{H1} and~\rmref{H2}, then~\eqref{eq:liminf-limsup} holds and $F$ exhibits historical behavior for $(\PP \times \lbb)$-almost every point.}
\end{mainprop}

\subsection{Historical behavior from the ergodicity of the reference measure}

We study historical behavior for skew products
\begin{equation}\label{skew product-principal}
    \fu{F}{\C}{\C}, \quad 	 F(\omega, x) = (\sigma(\omega), f_{\omega}(x)),
\end{equation}
where $\sigma\colon \Esp\to\Esp$ is an ergodic, measure-preserving transformation of a standard probability space $(\Omega, \mathscr{F},\PP)$, and $\fu{f_\omega}{I}{I}$ are measurable functions on the interval $I = [0,1]$.
As in Definition~\ref{def:arcsine-law-one-s}, we introduce a corresponding {fluctuation law}.

\begin{defi}\label{def:arcsine-law-skew-p}
    The system $F$ satisfies the
    {\emph{skew-product fluctuation law}} if there exist constants  $\gamma_0, \gamma_1 \in (0,1)$
    such that
    \begin{equation} \label{eq:weak-arcsine}
        \liminf_{n\to\infty} \,( \PP \times \lbb)\bigg(\, {\frac{1}{n}}\sum_{j=0}^{n-1} \ind_{I_i(\gamma_i)}(f^j_\omega( x)) \leq \alpha \,\bigg) < 1, \quad \text{ for every $\alpha \in (0,1)$ and $i = 0,1$.}
    \end{equation}
\end{defi}

To relax~\ref{H0}, we use generalized ergodicity for the reference measure $\mu=\PP\times\lbb$. We say that $\mu$ is \emph{ergodic} with respect to $F$ if $\mu(A)\in\{0,1\}$ for every measurable set $A$ such that
 $F^{-1}(A)=A.$ Unlike classical ergodicity, this does not require $\mu$ to be $F$-invariant.  When $F$ is invertible and $\mu$ is quasi-invariant, this definition agrees with ergodicity of the $\mathbb Z$-action~\cite{greschonig2000ergodic,MR3204667}. For a nonsingular map, equality modulo $\mu$ gives the same ergodicity condition. 
%Conservativity is a separate requirement for infinite invariant measures.

%\enlargethispage{0.5cm}

If $\mu=\PP\times\lbb$ is $F$-invariant, Birkhoff's ergodic theorem gives convergence of time averages at $\mu$-a.e.~point. Thus,  $F$ cannot exhibit historical behavior almost everywhere. The next result gives historical behavior almost everywhere for mild skew products satisfying the {fluctuation law} and ergodicity of $\mu$ with respect to $F$. As a consequence, $\mu$ must be a \emph{non-invariant} measure of $F$.

\begin{mainthm}\label{thm:ergodic-assumption}
    Let $F$ be a skew product as in~\eqref{skew product-principal}. Assume that
    \begin{enumerate}[label=$\mathrm{(\roman*)}$]
        \item the measure $\PP \times \lbb$ is ergodic with respect to $F$, and
        \item $F$ satisfies the {skew-product fluctuation law} {with $\gamma_0<\gamma_1$}.
    \end{enumerate}
    Then,
    \begin{align*}
		 \liminf_{n\to\infty}\s f^j_{\omega}(x) {\leq \gamma_0 < \gamma_1 \leq}
        \limsup_{n\to\infty}\s f^j_{\omega}(x)  \quad \text{ for $(\PP\times \lbb)$-a.e.~$(\omega,x)\in\Esp \times I$.}
	\end{align*}
In particular, $F$ exhibits historical behavior for $(\PP \times \lbb)$-almost every point.
\end{mainthm}

\subsection{Limit points of the sequence of empirical measures}
We study the asymptotic behavior of the empirical measures 
\begin{equation*}
    e_n(\omega,x) \eqdef \frac{1}{n} \sum_{j=0}^{n-1} \delta_{f_\omega^j(x)}.    
\end{equation*}
We define the set
\begin{equation}\label{eq:empirical measure}
    \mathcal{L}(\omega,x) \eqdef \big\{ \, \nu \,\colon  \ \ \text{$\nu$ is an accumulation point in the weak$^*$ topology of $e_n(\omega,x)$} \big\}.
\end{equation}  
Theorems~\ref{thm:B-random-walks} and~\ref{thm:ergodic-assumption} show that the empirical measures
fail to converge in the weak$^*$ topology for $(\mathbb{P} \times \lbb)$-almost every point. Thus, $\mathcal{L}(\omega,x)$ is nontrivial.  Moreover, for one-step systems with continuous fiber maps, these accumulation points are \emph{stationary measures}, i.e., measures $\nu$ on $I$ that induce $F$-invariant measures of the form $\mathbb{P} \times \nu$.
{However, for mild skew products, the product measure $\mathbb{P} \times \nu$ is not generally $F$-invariant for a given $\nu \in \mathcal{L}(\omega,x)$.}

{For many systems exhibiting historical behavior, such as those considered here, orbits spend negligible time in the interval interior.
The next result describes the empirical-measure limit set under this vanishing occupation~time.
}

\begin{mainprop} \label{mainpropo-limitset}
Let $F$ be a skew product satisfying the assumptions of Theorem~\ref{thm:B-random-walks} or Theorem~\ref{thm:ergodic-assumption}. Assume that  for any $0<\epsilon <1/2$  it holds
    \begin{equation}  \label{OT}
            \lim_{n\to\infty}\frac{1}{n} \sum_{j=0}^{n-1} \ind_{[\epsilon,1-\epsilon]}(f^j_\omega(x))=0  \quad \text{$(\mathbb{P}\times \lbb)$-a.e.~$(\omega,x)\in \Omega\times I$} .
    \end{equation}   
 Then, $$\mathcal{L}(\omega,x) = \big\{ \lambda \delta_0 + (1-\lambda) \delta_1 \colon \lambda \in [0,1] \big\} \quad \text{for $(\mathbb{P} \times \lbb)$-a.e.~$(\omega,x) \in \Omega \times I$.}$$
%  
% Let $F$ be a skew product satisfying the assumptions of
% Theorem~\ref{thm:B-random-walks} or
% Theorem~\ref{thm:ergodic-assumption}.
% Assume that for every $0<\epsilon<1/2$,
% \begin{equation}\label{OT}
% \lim_{n\to\infty}\frac{1}{n}\sum_{j=0}^{n-1}
% \ind_{[\epsilon,1-\epsilon]}(f_\omega^j(x))=0
% \quad\text{for $(\PP\times\lbb)$-a.e.~$(\omega,x)\in \Omega \times I$.}
% \end{equation}
% Then
% \[
% \mathcal L(\omega,x)
% =\big\{\lambda\delta_0+(1-\lambda)\delta_1
% \colon\lambda\in[0,1]\big\}
% \quad\text{for $(\PP\times\lbb)$-a.e.~$(\omega,x)\in \Omega \times I$.}
% \]
%  
\end{mainprop}

\begin{rem} \label{rem-ppp}  Under the assumptions of Theorem~\ref{thm:B-random-walks}, the conclusion of the previous result holds for every $x\in (0,1)$ and $\mathbb{P}$-a.e.~$\omega\in \Omega$ {whenever, for every fixed $x\in (0,1)$ and  $0<\epsilon <1/2$, the limit of the occupational time in~\eqref{OT} holds $\mathbb{P}$-almost surely.}
\end{rem}

{

Classically, one-dimensional random walks with non-degenerate increments (i.e., non-frozen at a single value) spend asymptotically zero proportion of time in any fixed compact subset of the interior of their state space (see Proposition~\ref{cor:main-cases}). In Theorem~\ref{thm:A-unified}, we prove a vanishing interior occupation time law even with non-i.i.d.~steps.  Hence, in the setting of Proposition~\ref{maincor:conjugation-random-walk}, under~\rmref{H1} and~\rmref{H2}, we will see that the assumption~\eqref{OT} follows from the random-walk conjugacy. Thus, from Proposition~\ref{mainpropo-limitset}
and Remark~\ref{rem-ppp}, we have the following:

}

\begin{maincor} \label{cor:lim-set} 
Let $F$ be a one-step skew product satisfying~\rmref{H1} and~\rmref{H2}. Assume that, for every $x\in(0,1)$, the fiber process $\{X_n^x\}_{n\geq0}$ is conjugate to a random walk on $\mathbb Z$ or $\mathbb R$ with mean zero and positive finite variance. Then, $$\mathcal{L}(\omega,x)=\big\{\lambda \delta_0 + (1-\lambda) \delta_1\colon \lambda\in [0,1]\big\} \quad \text{for every $x\in (0,1)$ and  $\mathbb{P}$-a.e.~$\omega\in \Omega$.}$$
\end{maincor}

%% file: Sec1.tex
\subsection{Examples}\label{ss:examples-of-sp-hb}

This section shows how our results unify and extend known skew-product examples exhibiting historical behavior almost everywhere. These examples share structural features: the weak form of the arcsine law and a kind of ergodicity. We analyze them by applying Theorems~\ref{thm:B-random-walks} and~\ref{thm:ergodic-assumption}. {We also construct new non-trivial examples.}

{\vspace{-0.4cm}
\subsubsection{{$T^\Psi$-transformations}}
{
Consider the product space $(\Esp, \mathscr{F}, \PP) \eqdef (\mathcal{A}^\mathbb{N}, \mathcal{F}^\mathbb{N},p^\mathbb{N})$, where $\mathcal{A}$ is an at most countable alphabet and $p$ is a probability measure with full support.} Let $M$ be either $\mathbb{T} = \mathbb{R}/\mathbb{Z}$ or $I = [0,1]$, and let $T: M \to M$ be a \emph{Morse-Smale {homeomorphism}
of period one}. Thus, the nonempty set of periodic points of $T$ consists only of period-one sinks and sources. {Consider a one-step random variable $\Psi: \Omega \to \mathbb{Z}$, i.e., $\Psi(\omega)=\Psi(\omega_0)$ for $\omega=(\omega_i)_{i\geq 0}\in \Omega$.} The one-step skew product    $\fu{F^{}_{\smash{{T^\Psi}}}}{\Esp \times M}{\Esp \times M}$, is defined as
\begin{equation}\label{North-south}
 \quad F^{}_{\smash{{T^\Psi}}}(\omega, x) \eqdef (\sigma(\omega), f^{}_{\omega_0}(x))\quad \text{where} \quad f^{}_{\omega_0}(x) \eqdef T^{{\Psi}(\omega_0)}(x).
\end{equation}
{As we will see in Section~\ref{proof:Proposition-TPsi}, the sequence $\{X_n^x\}_{ n\geq 0}$, with $X^x_n(\omega)=f_\omega^n(x)$ is conjugated to the $\Z$-valued random walk $\{S_n\}_{n\geq 0}$ by $S_0(\omega)=0$ and $S_n(\omega)= \Psi(\omega_0)+\dots+\Psi(\omega_{n-1})$
for $n\geq 1$. Hence, }
 as a consequence of Proposition~\ref{maincor:conjugation-random-walk} and Corollary~\ref{cor:lim-set}, we obtain the following result:
\enlargethispage{1cm}

\begin{mainprop}\label{cor:north-south}
{If $\mathbb{E}[\Psi]=0$ and $0<\mathbb{E}[\Psi^2]<\infty$}, the skew product {$F^{}_{\smash{T^\Psi}}$} in~\eqref{North-south} satisfies~\eqref{eq:arcsine-law}~for $\lbb$-a.e.~$x \in M$ and exhibits historical behavior for $(\PP \times \lbb)$-a.e.~point. Moreover, for every pair of consecutive fixed points~$p$~and~$q$~of~$T$,
$$
\mathcal{L}(\omega,x)=\big\{\lambda \delta_p + (1-\lambda) \delta_q\colon \lambda\in [0,1]\big\} \quad \text{for $\mathbb{P}$-a.e.~$\omega\in \Omega$ \ and $x \in {(p,q)}$}.
$$
\end{mainprop}

{Taking in our general framework the alphabet $\mathcal{A}=\{-1,1\}$, the symmetric Bernoulli measure  $\PP(\omega_0 = -1) = \PP(\omega_0 = 1) = \frac{1}{2}$, and $\Psi(\omega)=\omega_0$, the skew-product~\eqref{North-south} becomes $$F_T(\omega,x)\eqdef (\sigma(\omega),T^{\omega_0}(x)), \qquad \omega=(\omega_i)_{i\geq 0} \in \{-1,1\}^\mathbb{N}, \ \ x\in M.$$ This map is a generalization to the classical $(T,T^{-1})$-transformation~\cite{kalikow1982t}.
Since $\Psi$ has mean zero and positive finite variance, Proposition~\ref{cor:north-south} applies to~$F_T$.}
Molinek~\cite[Thm.~8]{molinek1994asymptotic} established that if $T$ is a north-south diffeomorphism of $M$, the natural extension of this~$F_T$ to $\{-1,1\}^{\mathbb{Z}}
\times M$, exhibits historical behavior almost everywhere. Since $(\Bar{\Esp}, \bar{\PP}, \sigma)=(\mathcal{A}^\mathbb{Z},p^\mathbb{Z},\sigma)$ is a  measure-theoretic extension of $(\Esp, \PP, \sigma)$, Proposition~\ref{thm:extended results-ergodic conjugated} allows us to immediately generalize our result to the bi-sequence base space. This extends Molinek's result:

\begin{maincor}\label{cor:morse-smale}
Let $T$ be a Morse-Smale diffeomorphism of period one on a one-dimensional compact manifold $M$, and consider the one-step skew product
\[
\Bar{F}_{\smash{{T^\Psi}}}: \Bar{\Esp} \times M \to \Bar{\Esp} \times M, \quad \Bar{F}_{{T^\Psi}}(\omega, x) = (\sigma(\omega), T^{{\Psi}(\omega_0)}(x)).
\]
{If $\mathbb{E}[\Psi]=0$ and $0<\mathbb{E}[\Psi^2]<\infty$}, then $\Bar{F}_{\smash{{T^\Psi}}}$ exhibits historical behavior for $(\bar{\PP} \times \lbb)$-a.e. point, and for every pair of consecutive fixed points $p$ and $q$ of $T$,
\[
\mathcal{L}(\omega,x) = \left\{ \lambda \delta_p + (1 - \lambda) \delta_q : \lambda \in [0,1] \right\} \quad \text{for $\bar{\PP}$-a.e. $\omega \in \Bar{\Esp}_2$ and $x \in {(p,q)}$}.
\]
\end{maincor}

{
\subsubsection{Coupling random walks}
Consider $(\Esp, \mathscr{F}, \PP)= (\mathcal{A}^\mathbb{N}, \mathcal{F}^\mathbb{N},p^\mathbb{N})$ as in the previous subsection. Let $T: I \to I$ be a {north-south homeomorphism}, i.e., a homeomorphism of the interval $I=[0,1]$ that fixes only the endpoints. Let $\{p_k\}_{k\in \mathbb{Z}} \subset (0,1)$ be strictly increasing ladder points with \(\lim_{k\to-\infty}p_k=0\) and \(\lim_{k\to +\infty}p_k=1\). Denote \(I_k=[p_k,p_{k+1})\) and $d_k=p_{k+1}-p_k$.
Consider one-step random variables $Z:\Omega \to \mathbb{Z}$ and $\Psi:\Omega\to \mathbb{Z}$ and define the one-step skew product
\begin{equation} \label{eq:F-coupling}
\fu{F^{}_{\smash{T^\Psi,Z}}}{\Esp \times I}{\Esp \times I},
\qquad
F^{}_{\smash{\smash{T^\Psi,Z}}}(\omega, x) = (\sigma(\omega), f^{}_{\omega_0}(x)),
\end{equation}
where the fiber maps are given, for $x\in I_k$, by
\begin{equation}\label{eq:coupling-fiber}
          f^{}_{\omega_0}(x)
          =p^{}_{k+Z(\omega_0)} + d^{}_{k+Z(\omega_0)}\cdot T^{\Psi(\omega_0)}(u) \in I_{k+Z(\omega_0)},
          \qquad
          u=\frac{x-p_k}{d_k},
\end{equation}
with boundary conditions \(f^{}_{\omega_0}(0)=0,\ f^{}_{\omega_0}(1)=1\) (see Figure~\ref{fig:coupling}).

\begin{figure}[t]
    \centering
    \begin{tikzpicture}

        \node[anchor=south west, inner sep=0] (image) at (0,0) {
            \includegraphics[width=0.45\linewidth]{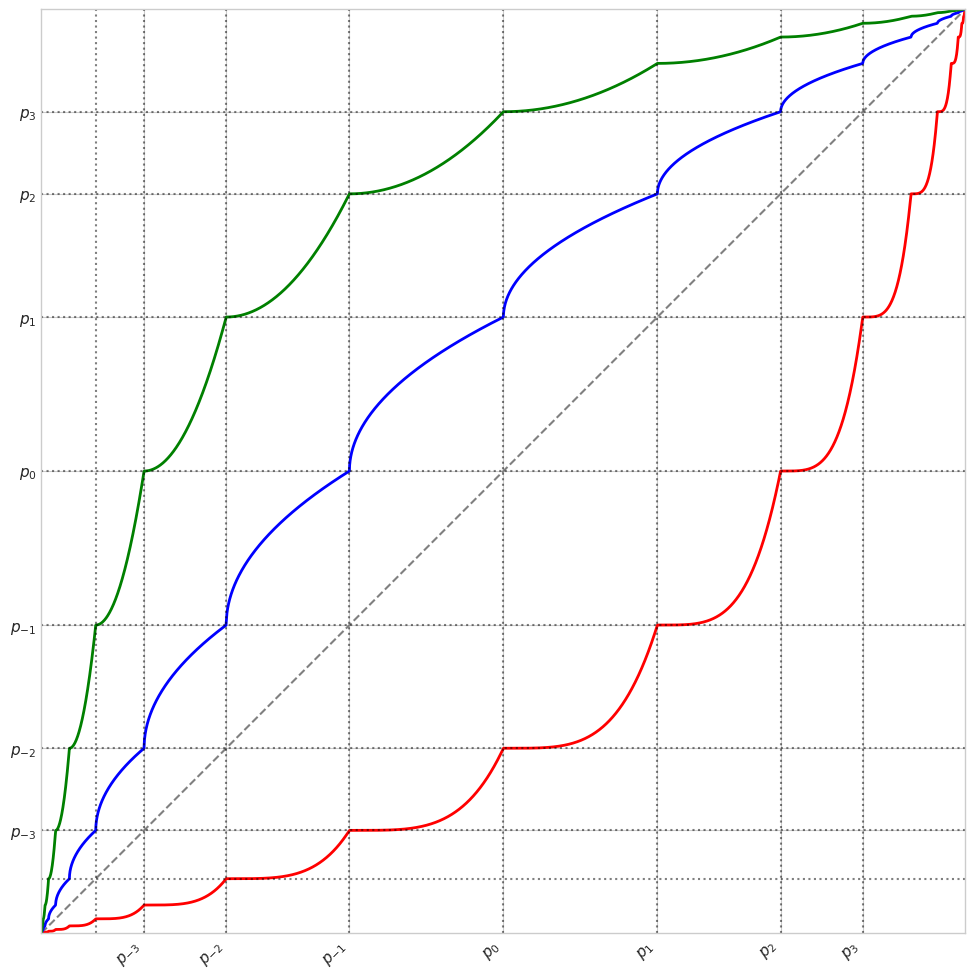}
        };

        \begin{scope}[
            x={(image.south east)},
            y={(image.north west)},
            font=\large
        ]
            \node[red] at (0.84, 0.45) {$\smash{ f_{-2}}$};
            \node[blue] at (0.43, 0.7) {$f_{1}$};
            \node[green!50!black] at (0.42, 0.9) {$f_{3}$};
            \end{scope}
    \end{tikzpicture}
   \caption{{Fiber maps of a coupling random walk skew product with $\mathcal{A} = \{-2, 1, 3\}$ and probability vector $p=(\frac{1}{2}, \frac{1}{4}, \frac{1}{4})$. The one-step random variables are $Z(\omega_0) = \omega_0$, $\Psi(-2) = 2$, $\Psi(1) = -1$, and $\Psi(3) = 1$, using a north-south map $T(u) = u^2$ with partition points $p_k = \frac{2^k}{1 + 2^k}$ for $k \in \mathbb{Z}$.}}
    \label{fig:coupling}
\end{figure}

The process $X_n^x(\omega)=f^n_{\omega}(x)$ need not be conjugate to a random walk, but can be written as
 \begin{equation}\label{eq:coupling-randomwalk}
        X_n^x = p^{}_{Z_n} + d^{}_{Z_n} \cdot u_n,  \quad  n\geq 0,
 \end{equation}
where
\begin{equation}\label{eq:process-definitions}
\begin{aligned}
Z_0(\omega) &= k,         &\quad Z_n(\omega) &= Z_{n-1}(\omega) + Z(\omega_{n-1}),  &\quad n&\geq 1, \\
S_0(\omega) &= 0,         &\quad S_n(\omega) &= S_{n-1}(\omega) + \Psi(\omega_{n-1}), &\quad n&\geq 1, \\
u_0(\omega) &= u,         &\quad u_n(\omega) &= T^{S_n(\omega)}(u_0),                &\quad n&\geq 1.
\end{aligned}
\end{equation}
Since $\{u_n\}_{n\geq 1}$ is conjugate to the random walk $\{S_n\}_{n\ge0}$, the process $\{X_n^x\}_{n\geq 0}$ can be regarded as a coupling of two random walks: a {\emph{macro} walk $\{Z_n\}_{n\ge0}$} and a {\emph{micro} walk $\{S_n\}_{n\ge0}$}, both evolving on the integer lattice \(\mathbb{Z}\).

Theorem~\ref{thm:B-random-walks} and Proposition~\ref{mainpropo-limitset} give:

\begin{mainprop}\label{maincor:coupling} Let $\Psi$ and $Z$ be one-step random variables with $\mathbb{E}[Z]=0$ and $0<\mathbb{E}[Z^2]<\infty$. Then,
 $F^{}_{\smash{T^\Psi,Z}}$ as in~\eqref{eq:F-coupling} has historical behavior for $(\PP \times \lbb)$-a.e.~point. Moreover, for every~$x\in (0,1)$,
$$
\mathcal{L}(\omega,x)=\big\{\lambda \delta_0 + (1-\lambda) \delta_1\colon \lambda\in [0,1]\big\} \quad \text{for $\mathbb{P}$-a.e.~$\omega\in \Omega$}.
$$
\end{mainprop}
}

\subsubsection{Skew-flow transformations} \label{ss:skew-flow}
{Let $(\Omega,\mathbb{P}, \sigma)$ be a measure-preserving dynamical system, where $\sigma:\Omega \to \Omega$ is either
\begin{itemize}
  \item a topologically mixing (one-sided or two-sided) subshift of finite type, or
  \item the restriction of a $C^1$-diffeomorphism to a topologically mixing hyperbolic basic set.
\end{itemize}
The probability $\mathbb{P}$ is an equilibrium state for a H\"older potential, hence a \emph{H\"older Gibbs measure}~\cite{Bow75}.}

Let $M$ be a compact one-dimensional manifold, either $\mathbb{T}=\mathbb{R}/\mathbb{Z}$ or $I=[0,1]$. Let $\varphi:\mathbb{R} \times M \to M$ be a $C^1$ flow on $M$, and let $\phi:\Omega \to \mathbb{R}$ be a {measurable} function. Define the skew product
\begin{equation} \label{eq:skew-flow}
    F_{\varphi,\phi}(\omega,x) = (\sigma(\omega), f_\omega(x)) \quad \text{where} \quad f_\omega(x) = \varphi(\phi(\omega), x).
\end{equation}
We study $F_{\varphi,\phi}$ via conjugation with the skew translation $T_\phi \colon \Esp \times \mathbb{R} \to \Esp \times \mathbb{R}$ defined by $T_\phi(\omega, x) \eqdef (\sigma(\omega), x + \phi(\omega)).$
{By Theorem~\ref{thm:A-unified}, $F_{\varphi,\phi}$ satisfies the vanishing occupational time~\eqref{OT}} provided  $\phi$ satisfies the following cohomological condition:
\begin{enumerate}[leftmargin=1.5cm, label=(C\arabic*), start=1]
    \item\label{C1} the equation $\phi = u \circ \sigma - u$ has no solution in $L^2(\mathbb{P})$.
\end{enumerate}
When such a solution $u$ exists, $\phi$ is called an \emph{(additive) coboundary}. Ji and Molinek~\cite{JM00} considered the case where $\varphi$ is a north-south flow on $M=\mathbb{T}$, showing that $F_{\varphi,\phi}$ exhibits historical behavior almost everywhere whenever $\phi$ has zero expectation, is H\"older continuous, and is not a coboundary. Under the same assumptions on $\phi$, Theorem~\ref{prop:skew-translation-ergi-arcsine} implies that $F_{\varphi,\phi}$ satisfies an arcsine law.
A direct application of Theorem~\ref{thm:ergodic-assumption} also requires ergodicity of the reference measure, which does not readily follow from the additive non-coboundary assumption~\rmref{C1}. However, the conjugacy between skew-flows and skew-translations allows us to analyze their long-term behavior. The skew-translation's additive, random-walk-like fiber structure makes asymptotic behavior independent of the initial point. We thus derive the asymptotic properties from base ergodicity, together with the fluctuation law and the vanishing interior occupational time. This gives:

\begin{mainprop} \label{maincor:skew-flow-preserving-orientation}
Let $\varphi$ and $\phi$ be, respectively, a Morse-Smale flow on $M$ and a H\"older continuous function satisfying~\rmref{C1} and $\mathbb{E}[\phi]=0$. Then,  $F_{\varphi,\phi}$ exhibits historical behavior for $(\PP \times \lbb)$-a.e.~point.  Moreover, given $x\in (p,q)$ where $p$ and $q$ are two consecutive equilibrium points of $\varphi$, we have that
$$
\mathcal{L}(\omega,x)=\big\{\lambda \delta_p + (1-\lambda) \delta_q\colon \lambda\in [0,1]\big\} \quad \text{for $\mathbb{P}$-a.e.~$\omega\in \Omega$.}
$$
\end{mainprop}

Now, consider the skew product~$\tilde{F}_{\varphi,\phi}\colon \Esp \times \mathbb T \to \Esp \times \mathbb T$, defined by
\begin{equation*}
    \tilde{F}_{\varphi,\phi}(\omega,x) = (\sigma(\omega), f_\omega(x)), \quad \text{where} \quad f_\omega(x) = -\varphi(\phi(\omega), x) \mod 1.
\end{equation*}
The fiber maps \( f_\omega \) are orientation-reversing, whereas those in~\eqref{eq:skew-flow} are orientation-preserving. Crovisier et al.~\cite{Crovisier2020} studied the case where the base dynamics $\sigma$ is a $C^2$ volume-preserving Anosov diffeomorphism on $\Omega=\mathbb{T}^2$, and $\varphi$ is a north-south flow on $\mathbb{T}$ induced by a $1$-periodic vector field $X$ on $\mathbb{R}$, sufficiently close to zero, such that \mbox{$X(-x)=-X(x)$} and $X(0)=X(1/2)=0$. Thus, $\varphi$ is a \emph{symmetric Morse-Smale} flow, satisfying
$$\varphi(t,-x \ \mathrm{mod}\, 1)=-\varphi(t,x) \mod 1 \quad \text{for every  $(t,x)\in \mathbb{R}\times \mathbb T$.}$$
This symmetry implies that $x=1/2$ is necessarily an equilibrium point of $\varphi$. They showed that $\tilde{F}_{\varphi,\phi}$ exhibits historical behavior almost everywhere, provided $\phi:\Omega\to \mathbb{R}$ is H\"older, has zero expectation, and satisfies the following cohomological condition:
\begin{enumerate}[leftmargin=1.5cm, label=(C\arabic*), start=2]
     \item\label{C2}     there do not exist
     \(\lambda\in \mathbb{R}\setminus\{0\}\) and  \(\psi:\Omega\to\mathbb S^1\) in $L^2(\PP)$ such that
\begin{equation} \label{*}
e^{i\lambda\phi(\omega)} \;=\; \frac{\psi(\sigma(\omega))}{\psi(\omega)}
\quad\text{for \(\PP\)-a.e.\ \(\omega\in \Omega\).}
\end{equation}
 \end{enumerate}
When $\lambda\neq0$ and $\psi$ satisfy~\eqref{*}, the function $\phi$ (or more precisely the character $e^{i\lambda\phi}$) is said to be a \emph{multiplicative coboundary}. Every additive coboundary is also a multiplicative coboundary. Thus, condition~\rmref{C2} implies~\rmref{C1}, while the converse need not hold.
As observed in Remark~\ref{rem:equivalence-ergodicity-C2}, condition~\rmref{C2} is equivalent to ergodicity of every compact circle quotient of $T_\phi$. 
%The skew-flow conclusion below follows from the occupation-time proof under~\rmref{C1} in Section~\ref{ss:proof-cor-skew-flow}.

\begin{rem}\label{rem:C1-C2}
When $\phi$ is H\"older continuous, Liv\v{s}ic's theorem~\cite[Thm.~9]{Liv72} (see also Parry--Pollicott~\cite[Props.~4.2 and~4.12]{PP90} and~\cite[Thms.~1 and~2]{PP97}) ensures that any measurable solution of the cohomological equation agrees almost everywhere with a continuous one. Hence, if $\phi$ is an additive (resp.~multiplicative) coboundary, then $\phi(\alpha)=0$ (resp.~$\lambda\phi(\alpha)\in2\pi\mathbb{Z}$) for every fixed point $\alpha\in\Esp$ of $\sigma$. Thus,~\ref{C1} holds whenever $\phi(\alpha)\neq0$ for some fixed point $\alpha$. Moreover,~\ref{C2} holds if there exist fixed points $\alpha,\beta\in\Esp$ such that $\phi(\alpha)$ and $\phi(\beta)$ are rationally independent.
\end{rem}

For symmetric Morse-Smale flows \( \varphi \), the symmetry implies
$\pi(\varphi(t, -x\ \mathrm{mod}\, 1)) = \varphi(t, \pi(x))$ for all \( (t, x) \in \mathbb{R} \times \mathbb{T} \),
where
\begin{equation*}
    \pi(x) =
    \begin{cases}
        x, & x \in [0, 1/2], \\
        -x \ \mathrm{mod}\, 1, & x \in [1/2, 1].
    \end{cases}
\end{equation*}
Hence, we have $\Pi \circ \tilde{F}_{\varphi,\phi} = F_{\varphi,\phi} \circ \Pi$, where $\Pi(\omega, x) = (\omega, \pi(x) )$.
By setting \( I = [0, 1/2] \), since \( \Pi \colon \Omega \times \mathbb T \to \Omega \times I \) is a continuous surjection, \( \tilde{F}_{\varphi,\phi} \) is an extension of the restriction of \( F_{\varphi,\phi} \) to \( \Omega \times I \) such that {\( \Pi_*(\mathbb{P} \times \lbb) = \mathbb{P} \times (2\cdot\lbb) \)}. Hence, Propositions~\ref{maincor:skew-flow-preserving-orientation} and~\ref{thm:extended results-ergodic conjugated} extend~\cite{Crovisier2020} as follows:

\begin{maincor}\label{cor:generalized-crovisier}
Let $\varphi$ and $\phi$ be, respectively, a symmetric Morse-Smale flow and a H\"older continuous function satisfying~\rmref{C1} and $\mathbb{E}[\phi]=0$. Then, $\tilde{F}_{\varphi,\phi}$ exhibits historical behavior for $(\PP \times \lbb)$-a.e.~point.  Moreover, given $x\in (p,q)$ where $p$ and $q$ are two consecutive equilibrium points of $\varphi$ in $[1/2,1]$,
$$
\mathcal{L}(\omega,x)=\bigg\{\lambda \frac{\delta_p +\delta_{\pi(p)}}{2} + (1-\lambda) \frac{\delta_q + \delta_{\pi(q)}}{2}\colon \lambda\in [0,1]\bigg\}
\quad \text{for $\mathbb{P}$-a.e.~$\omega\in \Omega$}.
$$
\end{maincor}

\begin{rem}\label{rmk:partial-diffeo}
When \( \Omega=\mathbb{T}^n\), $n\geq 2$, and $\sigma:\mathbb{T}^n \to \mathbb{T}^n$ is a \( C^2 \) volume-preserving Anosov diffeomorphism, the skew products \( F_{\varphi,\phi} \) and \( \tilde{F}_{\varphi,\phi} \) can be realized as partially hyperbolic diffeomorphisms of \(  \mathbb{T}^{n+1} \), arbitrarily close to $\sigma\times \mathrm{Id}$ and \( \sigma \times (-\mathrm{Id})\) respectively. Here \( F_{\varphi,\phi} \) is non-transitive, while \( \tilde{F}_{\varphi,\phi} \) is transitive exactly when \( \varphi \) is north-south.
\end{rem}

\subsubsection{Zero Schwarzian derivative} \label{ss:zero-derivative}
The Schwarzian derivative of a $C^3$ interval diffeomorphism $f$ is
\begin{equation*}
	Sf(x)\eqdef\frac{f'''(x)}{f'(x)}-\frac{3}{2}\pa{\frac{f''(x)}{f'(x)}}^2.
\end{equation*}
Here, we consider a skew product as defined in \eqref{skew product-principal}, where $\sigma\colon\Omega \to \Omega$ is a Bernoulli shift on a product space
$(\Esp,\mathscr{F},\PP)=(\mathcal{A}^{\mathbb{N}},\mathcal{F}^{\mathbb{N}},p^{\mathbb{N}})$  where {$\mathcal{A}$ is at most countable alphabet, $p$ is a probability measure with full support} and the fiber maps $f_\omega$ are $C^3$ interval diffeomorphisms
satisfying the following conditions:
\begin{enumerate}[leftmargin=1.4cm,label=(S\arabic*),itemsep=0.1cm]
    \item\label{BM1} $f_\omega(0)=0$, $f_\omega(1)=1$, and $f_\omega\not =\mathrm{id}$ for $\PP$-a.e.~$\omega\in\Esp$,
    \item\label{BM2}
    $Sf_\omega(x)=0$, for~$(\PP\times\lbb)$-a.e.~$(\omega,x)\in\C$,  \item\label{BM3} The Lyapunov exponents $\lambda(\delta_i)\eqdef\int \log(f_\omega'(i)) \, d\PP=0$ for $i=0,1$.
\end{enumerate}
{As indicated in~\cite{Bonifant},} under conditions~\ref{BM1}--\ref{BM3}, the skew product can be written as
\begin{equation} \label{eq:fracional-linear}
 \fu{F_a}{\Esp \times I}{\Esp \times I},\quad F_a(\omega,x)=(\sigma(\omega), f_\omega(x)) \ \ \ \text{with}	\ \	\ f_\omega(x)=\frac{a(\omega)x}{1+(a(\omega)-1)x}
	\end{equation}
where $\fu{a}{\Esp}{(0,\infty)}$ is a measurable function such that
\begin{equation} \label{eq:fracional-linear-b}
		\int \log a(\omega) \, d\PP =0 \qquad  \text{and} \qquad
a(\omega)\not=1 \quad \text{for every $\omega\in \Esp$}.
	\end{equation}
When needed, we also assume:
\begin{enumerate}[leftmargin=1.4cm,label=(E\arabic*),start=0,itemsep=0.1cm]
    \item\label{E0} the function $a(\omega)$ depends only on the zero-coordinate of $\omega=(\omega_i)_{i\geq 0}\in \Omega$,
    \item\label{BM4} $\int (\log a(\omega))^2 \, d\PP<\infty$.
\end{enumerate}
Under condition~\ref{E0}, the skew product $F_a$ is one-step.

Bonifant and Milnor claimed in~\cite[Thm.~6.2]{Bonifant} that one-step skew products satisfying~\ref{BM1}--\ref{BM3}
exhibit  historical behavior for $(\mathbb{P}\times \mathrm{Leb})$-a.e.~point. Their proof sketch leaves several steps incomplete. We give a complete proof via Proposition~\ref{maincor:conjugation-random-walk}, including the missing pieces (the finite moment condition~\ref{BM4} and the countability of the alphabet $\mathcal{A}$).

\begin{mainprop}\label{thm:D-Bonifant-Milnor-historical-behavior}
Let $F_a$ be a skew product as in~\eqref{eq:fracional-linear} where $\fu{a}{\Esp}{(0,\infty)}$ satisfies~\eqref{eq:fracional-linear-b},~\ref{E0}, and~\ref{BM4}. Then $F_a$ exhibits historical behavior for $(\PP \times \lbb)$-a.e.~point. Moreover,
$$
\mathcal{L}(\omega,x)=\{\lambda \delta_0 + (1-\lambda) \delta_1\colon \lambda\in [0,1]\} \quad \text{for  any $x\in (0,1)$ and
$\mathbb{P}$-a.e.~$\omega\in \Omega$.}
$$
\end{mainprop}

Above, $F_a$ is one-step. We extend the result to mild skew products by relating $F_a$ to a skew-flow. Namely, let $\phi(\omega) = \log a(\omega)$ and define
\[
\varphi(t, x) = \frac{e^t x}{1 + (e^t - 1)x}, \quad (t, x) \in \mathbb{R} \times I.
\]
The flow $\varphi$ is generated by the differential equation $\dot{x}=x(1-x)$ and is north-south on $I=[0,1]$. The fiber maps of $F_a$ in~\eqref{eq:fracional-linear} are $f_\omega(x)=\varphi(\phi(\omega),x)$ for $(\omega,x)\in\Omega\times I$. When the alphabet $\mathcal{A}$ is finite, the Bernoulli measure $\mathbb{P}=p^{\mathbb{N}}$ is an equilibrium state for the H\"older potential $\psi:\Omega\to\mathbb{R}$ given by $\psi(\omega)=\log p(\omega_0)$, and hence a H\"older Gibbs measure for the full shift $\sigma:\Omega\to\Omega$. If, in addition, $\phi$ is H\"older continuous, satisfies~\ref{C1}, and $\mathbb{E}[\phi]=0$, the skew product in~\eqref{eq:fracional-linear} is a skew-flow to which Proposition~\ref{maincor:skew-flow-preserving-orientation} applies. This proves a conjecture of Bonifant and Milnor~\cite[conjecture before Hypothesis~6.1]{Bonifant}.

\begin{maincor}\label{cor:Bonifant-Milnor-general}
Let $F_a$ be a skew product as in~\eqref{eq:fracional-linear}, {where the alphabet $\mathcal{A}$ is finite}, $\phi(\omega)=\log a(\omega)$ is H\"older continuous and satisfies~\eqref{eq:fracional-linear-b} and~\ref{C1}. Then $F_a$ exhibits historical behavior for $(\PP \times \lbb)$-a.e.~point.  Moreover,
$$\mathcal{L}(\omega,x)=\{\lambda \delta_0 + (1-\lambda) \delta_1\colon \lambda\in [0,1]\} \quad \text{for $(\mathbb{P}\times \lbb)$-a.e.~$(\omega,x)\in \Omega\times I$.}
$$
\end{maincor}

{

\begin{rem}[Solution to the Bonifant--Milnor conjecture]
Bonifant and Milnor~\cite{Bonifant} studied skew products on the cylinder $\mathbb{T}\times I$ of the form
$F(x,y)=(\ell x \bmod 1, f_x(y))$, $\ell\ge2$, with fiber maps $f_x$ of zero Schwarzian derivative.
They proved historical behavior after replacing the base by a Bernoulli shift and assuming $F$ is one-step, and conjectured the same for the expanding map.
Since the map $x\mapsto \ell x \bmod 1$ is measure-theoretically isomorphic to a Bernoulli shift,
$F$ is isomorphic to a system $F_a$ satisfying the hypotheses of Corollary~\ref{cor:Bonifant-Milnor-general},
provided the dependence $x\mapsto f_x$ is H\"older (in this setting, this H\"older continuity is equivalent to that of $\phi=\log a$).
By Proposition~\ref{thm:extended results-ergodic conjugated}, Corollary~\ref{cor:Bonifant-Milnor-general} confirms that the original Bonifant--Milnor skew product over the expanding map exhibits historical behavior for $(\lbb\times\lbb)$-a.e.~point.
\end{rem}
}

\subsubsection{Interval functions}\label{ss:interval}
If $\Omega$  is a singleton, the skew product~\eqref{skew product-principal} can be interpreted as a measurable function $\fu{f}{I}{I}$. Aaronson et al.~\cite[see comments after Thm.~2]{AarThaZwei05} observe that the so-called \emph{Thaler functions} exhibit historical behavior almost everywhere. A \emph{Thaler function} $f\colon I\to I$ satisfies the following conditions for some $c\in(0,1)$ and $p>1$:
\begin{enumerate}[leftmargin=1.3cm, itemsep=0.1cm, label=(T\arabic*)]
	\item \label{T1} $f$ is \emph{full branch}: the restrictions
	$f|_{(0,c)}$ and $f|_{(c,1)}$ are increasing, onto, and $C^2$, and admit $C^2$-extensions to the closed intervals~$[0,c]$ and $[c,1]$, respectively;
	\item \label{T2} $f$ is \emph{almost expanding}: $f'(x) > 1$ {for every $x \in (0,c^-]\cup[c^+,1)$}, $f'(0) = f'(1) = 1$, and $f$ is convex and concave in neighborhoods of $0$ and $1$, respectively;
	\item \label{T3} $f(x) - x = h(x) x^{p+1}$ for $x \in (0,c)$, where $h(kx) \sim h(x)$ as $x \to 0$ for every $k \geq 0$;
	\item \label{T4} There exists $a \in (0,\infty)$ such that $(1 - x) - f(1 - x) \sim a^p (f(x) - x)$ as $x \to 0$.
\end{enumerate}
A standard example is the symmetric Manneville--Pomeau map
\begin{equation}\label{eq:manneville}
	{f(x) = \left\{ \begin{array}{ccl}
		x + 2^p x^{p+1} & \text{if} & x \in [0, \frac{1}{2}), \\
		x - 2^p (1 - x)^{p+1} & \text{if} & x \in [\frac{1}{2}, 1],
	\end{array}
	\right.
 \quad p > 1.}
\end{equation}
See Figure~\ref{fig:thaler-fu}. For $p=2$, the map~\eqref{eq:manneville} has no physical measure and satisfies the occupation-time condition~\rmref{H3b}; see~\cite{Ino00,BlaBunimo03,Keller04}.

This class of functions was introduced by Thaler~\cite{Thaler80,Thaler83}, who proved that any such function is {conservative and exact} with respect to $\lbb$ and admits a $\oldsigma$-finite ergodic measure~$\mu$ equivalent to the Lebesgue measure $\lbb$ {such that, for every $\epsilon > 0$, we have $\mu((\epsilon,1-\epsilon)) < \infty$. A direct consequence is that $\mu$ is ergodic with respect to $f$, and applying the ergodic theorem for infinite measure spaces~(cf.~\cite[Exercise 2.2.1]{AaronBook97}), we also have the condition of vanishing occupation time on~$(\epsilon,1-\epsilon)$ as in \eqref{OT} in Proposition~\ref{mainpropo-limitset}.
Subsequently, Thaler~\cite{Thaler02} established that every such map satisfies
\begin{equation*}
    \lim_{n\to\infty}\mathrm{Leb}\bigg(\,\s\ind_{I_i(\gamma)}(f^j(x))\leq\alpha\,\bigg) = G_{\alpha, \beta}(x) \quad \text{for every $\gamma\in(0,1)$ and $i \in \{0,1\}$,}
\end{equation*}
where $G_{\alpha, \beta}$ is the distribution function on $(0,1)$, given by
\begin{equation*}
G_{\alpha, \beta}(x) = \frac{1}{\pi\alpha}  \operatorname{arccot}\left(\frac{\beta (1-x)^\alpha}{x^\alpha (1-\beta) \sin(\pi\alpha)} + \cot(\pi\alpha) \right).
\end{equation*}
The parameters $\alpha$ and $\beta$ are determined by the properties \rmref{T1}--\rmref{T4}, namely
\begin{equation*}
\alpha = \frac{1}{p} \quad \text{and} \quad \beta = \frac{f'(c^-)}{f'(c^-) + f'(c^+)/a}.
\end{equation*}
Hence $f$ satisfies the fluctuation law.}
By Theorem~\ref{thm:ergodic-assumption} and Proposition~\ref{mainpropo-limitset} we recover~\cite{AarThaZwei05}, with $\mathcal L(x)$ denoting the weak$^*$ accumulation set of $n^{-1}(\delta_x+\cdots+\delta_{f^{n-1}(x)})$.

\begin{figure}
    \begin{minipage}[t]{0.5\textwidth}
  \hspace{1cm}\begin{tikzpicture}
    \begin{axis}[
      ticklabel style={font=\small},
      axis lines=box,
      xmin=0, xmax=1,
      ymin=0, ymax=1,
      xtick={0,1},
      xticklabels={$0$,$1$},
      ytick={0,1},
      yticklabels={},
      width=7cm, height=7cm,
      grid=major,
      grid style={gray!30},
      major tick length={0},
      major y tick style={draw=none},
      x axis line style={draw=none},
      y axis line style={draw=none},
      fill=white,
      draw=black,
    ]
      \draw (0,0) rectangle (1,1);
      \addplot[domain=0:0.5, samples=100, blue, line width=0.5pt] {x + 2^2 * x^(2+1)};
      \addplot[domain=0.5:1, samples=100, blue, line width=0.5pt] {x - 2^2 * (1-x)^(2+1)};
      \addplot[domain=0:1, samples=100, black, line width=0.5pt, dashed] {x};
      \node[blue] at (0.7, 0.5) {$f$};
      \draw[dashed, thin] (axis cs:0.5,0) -- (axis cs:0.5,1);
    \end{axis}
  \end{tikzpicture}
  \caption{\mbox{Symmetric Manneville-Pomeau}}
        \label{fig:manneville-pomeau}
    \end{minipage}
    \hspace{1.2cm}
    \begin{minipage}[t]{0.4\textwidth}
        \begin{tikzpicture}
\begin{axis}[
      ticklabel style={font=\small},
      axis lines=box,
      xmin=0, xmax=1,
      ymin=0, ymax=1,
      xtick={0,0.25,1},
      xticklabels={$0$,$c$,$1$},
      ytick={0,1},
      yticklabels={},
      width=7cm, height=7cm,
      grid=major,
      grid style={gray!30},
      major tick length={0},
      major y tick style={draw=none},
      x axis line style={draw=none},
      y axis line style={draw=none},
      fill=white,
      draw=black,
      xticklabel style={text height=1ex},
    ]
      \draw (0,0) rectangle (1,1);
      \addplot[domain=0:0.25, samples=100, red, line width=0.5pt] {x + (1-1/4)*(x/(1/4))^(2+1)};
      \addplot[domain=0.25:1, samples=100, red, line width=0.5pt] {x - (1/4)*((1-x)/(1-1/4))^(2+1)};
      \draw[dashed, thin] (axis cs:0.25,0) -- (axis cs:0.25,2);
      \node[red] at (0.7, 0.5) {$f$};
      \addplot[domain=0:1, samples=100, black, line width=0.5pt, dashed] {x};
    \end{axis}
  \end{tikzpicture}
        \caption{Thaler functions}
        \label{fig:thaler-fu}
    \end{minipage}
\end{figure}

\begin{maincor}\label{cor:Thaler-maps}
Every Thaler function $f$ has historical behavior almost everywhere. Moreover,
$$
\mathcal{L}(x)= \{\lambda \delta_0 + (1-\lambda)\delta_1 \colon \lambda \in [0,1]\},\quad\text{{for $\lbb$-\aep\, $x\in (0,1)$}}.
$$
\end{maincor}

Coates et al.~\cite{Stefano-comu} study Manneville--Pomeau-type interval maps with possibly zero or infinite derivatives at discontinuities. They proved the existence of a $\oldsigma$-finite ergodic measure equivalent to~$\lbb$. Subsequently, Coates and Luzzatto~\cite{CoaLuza23} demonstrated that such functions exhibit historical behavior almost everywhere by proving a condition on the occupational times which is a consequence of the fluctuation law (see condition~\rmref{H3b} in Section~\ref{s:weaks-arcsine-laws} and Proposition~\ref{prop:equivalence arcsine law}).
Theorem~\ref{thm:theorem-H3b} shows that this occupation-time condition and
ergodicity of the reference measure suffice for historical behavior almost everywhere. More recently, Coates et al.~\cite{CoaMelTale24} proved historical behavior for almost every point for \emph{generalized Thaler} functions {(i.e.~maps with $k$ branches and $k$ neutral fixed points, for~$k \geq 2$)}.
Thaler's result~\cite{Thaler80,Thaler83} still gives ergodicity of $\lbb$ for $f$.  Moreover, Sera and Yano~\cite{SeraYano19} and Sera~\cite{Ser20} concluded that these generalized Thaler functions also satisfy a generalization of the arcsine law (c.f.~\cite[Thm.~2.7]{CoaMelTale24}).

Coates et al.~\cite{CoaMelTale24} also observed that generalized Thaler functions have a unique \emph{natural measure}.
A measure $\nu$ is \emph{natural} for $f$ if some absolutely continuous probability measure $\lambda$ satisfies
$$
    \lim_{n\to\infty}\frac{1}{n}\sum_{j=0}^{n-1}f^j_*\lambda = \nu
$$
in the weak$^*$ topology. Every physical measure is natural, but not conversely. Its existence for generalized Thaler functions raises a question about systems with historical behavior almost everywhere:

\begin{question}\label{question1}
    Let $F$ be a skew product as in~\eqref{skew product-principala} that exhibits historical behavior almost everywhere, as studied here. Does a natural measure for~$F$ exist?
\end{question}

Molinek~\cite{molinek1994asymptotic} gives a partial positive answer to Question~\ref{question1} for one-step skew products of the generalized $(T, T^{-1})$-transformation type, as in~\eqref{North-south}. It remains open for mild skew products. In the context of Remark~\ref{rmk:partial-diffeo}, Question~\ref{question1} is closely related to a problem posed by Misiurewicz in~\cite[Question $9.4$]{book-problems07}, which addresses the existence of natural measures that are not physical measures in smooth dynamical systems.

\subsection{Organization of the paper.}

Section~\ref{s:historical-skewprod} recalls historical behavior for skew products, and Section~\ref{s:weaks-arcsine-laws} relates the fluctuation laws to occupation times. Sections~\ref{Sec:theorema A} and~\ref{Sec:proof-thm-ergodicskew} prove Theorems~\ref{thm:B-random-walks} and~\ref{thm:ergodic-assumption}; Section~\ref{ss:limit-set} proves Proposition~\ref{mainpropo-limitset}. Section~\ref{ss:Proof-cor-I-II-VI} treats the one-step
examples, while Section~\ref{ss:skew-translation} establishes the required occupation-time and fluctuation properties of skew-translations.
Finally, Section~\ref{ss:proof-cor-skew-flow} completes the skew-flow applications.

}

%% file: Sec2.tex
\section{Historical behavior on skew product maps}
\label{s:historical-skewprod}

We recall historical behavior for skew products~\eqref{skew product-principala} and suitable extensions.

\subsection{Definition of historical behavior} \label{ss:def-hb}
Let $(\Omega,\mathscr{F},\mathbb{P})$ be a standard probability space and $M$ a compact manifold.
Since $\Omega$ may be non-compact, so may $\Omega\times M$. We use the fiberwise notion from~\cite[Def.~1]{nakano2017historic}.
\begin{defi}\label{def:historical-behavior}
A skew product $F$ as in~\eqref{skew product-principala} exhibits \emph{historical behavior almost everywhere} if, for $(\PP\times\lbb)$-a.e.~$(\omega,x)$, there exists a continuous $\varphi:M\to\R$ such that
\[
\lim_{n\to\infty}\frac{1}n\sum_{j=0}^{n-1}\varphi(f^j_\omega(x))\quad\text{does not exist}.
\]\end{defi}

\begin{rem}\label{rem:randomhist-imply-skewhist}
Definition~\ref{def:historical-behavior} implies non-convergence of Birkhoff averages for $F$ by taking $\phi(\omega,x)=\varphi(x)$. The usual equivalence with non-convergence
of the \emph{empirical measures}~$\frac{1}{n}\sum_{j=0}^{n-1}\delta_{F^j(\omega,x)}$ requires $\Omega$ to be compact.
\end{rem}

\begin{rem}\label{rem:1}
Since $M$ is compact, Definition~\ref{def:historical-behavior} is equivalent to the non-convergence, in the weak* topology, of the
\emph{fiber empirical measures} $\frac{1}{n}\sum_{j=0}^{n-1}\delta_{f^j_\omega(x)}$ for $(\PP\times\lbb)$-a.e.~point.
\end{rem}

\begin{rem}\label{rem:2}
In
Theorems~\ref{thm:B-random-walks} and~\ref{thm:ergodic-assumption}, and in the analogous results below, Definition~\ref{def:historical-behavior}
holds with the same observable $\varphi$ for $(\PP\times\lbb)$-a.e.~point.
\end{rem}
\subsection{Historical behavior for extension maps}\label{ss:extension}

For the natural base extension, forward fiber orbits are unchanged, so historical behavior is preserved. The next proposition, following~\cite[Lem.~5.1]{BNRR22}, also covers extensions changing the fiber space. If $F$ exhibits historical behavior uniformly for almost every point, as in Remark~\ref{rem:2}, then every suitable extension $\bar F$ also exhibits historical behavior.

For continuous $\varphi:M\to\R$, the set of \emph{$\varphi$-irregular points} of $F$ is
\[
\mathrm I(F,\varphi)\eqdef
\Big\{(\omega,x)\in\Omega\times M:
\lim_{n\to\infty}\frac1n\sum_{j=0}^{n-1}\varphi(f^j_\omega(x))
\ \text{does not exist}\Big\}.
\]
\begin{prop}\label{thm:extended results-ergodic conjugated}

Let $\bar F:\bar\Omega\times\bar M\to\bar\Omega\times\bar M$ and $F:\Omega\times M\to\Omega\times M$ be skew products. Assume that there exists a continuous map of the form
$\Pi(\omega,x)=(\theta(\omega,x),\pi(x))$ such that
\[
F\circ\Pi=\Pi\circ\bar F,
\qquad
\Pi_*(\bar\PP\times\overline\lbb)=\PP\times\lbb,
\]
where $\bar\PP,\PP$ are the reference measures on $\bar\Omega,\Omega$ and $\overline\lbb,\lbb$ are the normalized Lebesgue measures on~$\bar M,M$. Then, for every continuous $\varphi:M\to\R$, it holds
\[
(\PP\times\lbb)(\mathrm I(F,\varphi))
\le (\bar\PP\times\overline\lbb)(\mathrm I(\bar F,\varphi\circ\pi)).
\]
In particular, if $F$ exhibits historical behavior almost everywhere with respect to the same observable~$\varphi$, then $\bar F$ does so with respect to $\varphi\circ\pi$.
\end{prop}

%% file: Sec4.tex
\section{fluctuation laws}
\label{s:weaks-arcsine-laws}

Let $F$ be a skew product as in~\eqref{skew product-principal} with fiber dynamics $f^n_\omega$ given by~\eqref{eq:fiber-dynamics}. Recall the \emph{pointwise-fiber} and \emph{skew-product fluctuation laws} in Definitions~\ref{def:arcsine-law-one-s} and~\ref{def:arcsine-law-skew-p}. For $\gamma \in (0,1)$, define
$I_0(\gamma)\eqdef[0,\gamma]$, $J_0(\gamma)\eqdef(0,\gamma)$,
$I_1(\gamma)\eqdef[\gamma,1]$, and $J_1(\gamma)\eqdef(\gamma,1)$.
{Write $L_i(\gamma)$ for either $I_i(\gamma)$ or $J_i(\gamma)$, $i=0,1$.} The next result relates the two fluctuation laws.

{

\begin{prop}\label{prop:skew-arc-fiber-arc-a}
Given $\gamma_0,\gamma_1, \alpha \in (0,1)$ and $i\in\{0,1\}$, we have that
\begin{equation} \label{(i)}
\liminf_{n\to\infty} \,(\PP\times\lbb)\big(\,\s\ind_{I_i(\gamma_i)}(f^j_\omega(x)) \leq \alpha\,\big)<1
\end{equation}
implies that there  exists a set $B\subset I$ with positive $\lbb$-measure such that
\begin{equation} \label{(ii)}
    \liminf_{n\to\infty} \,\PP\big(\,\s\ind_{I_i(\gamma_i)}(f^j_\omega(x)) \leq \alpha\,\big)<1 \quad \text{for every $x\in B$}.
\end{equation}
\end{prop}

\begin{proof}
Define
$$
g_n(x) = \PP\big( \frac{1}{n}\sum_{j=0}^{n-1}\ind_{I_i(\gamma_i)}(f^j_\omega(x)) \leq \alpha \big) \quad \text{and} \quad g(x) =\liminf\limits_{n\to\infty} g_n(x).
$$
By Fubini's theorem,~\eqref{(i)} is equivalent to $\liminf_{n\to\infty} \int g_n \,d\lbb < 1$ and~\eqref{(ii)} is equivalent to the set $B=\{x \in I : g(x) < 1\}$ having positive $\lbb$-measure.

Since $\{g_n\}_{n\geq 0}$ are non-negative measurable functions, Fatou's lemma gives
$$
\int g \,d\lbb \leq \liminf_{n\to\infty} \int g_n \,d\lbb < 1.
$$
Suppose~\eqref{(ii)} is false. Then $\lbb(B)=0$. Since $0 \leq g_n \leq 1$, we have $0 \leq g \leq 1$. Therefore, $g(x) = 1$ for $\lbb$-a.e.~$x \in I$, which implies $\int g \,d\lbb = 1$. This contradicts the preceding inequality, so $B$ has positive measure, proving~\eqref{(ii)}.
\end{proof}

\begin{defi}
The system $F$ satisfies the
\emph{fiberwise fluctuation law}
if there are $\gamma_0,\gamma_1\in(0,1)$ such that for every
$\alpha\in(0,1)$ there are
$B_0,B_1\subset I$ with positive $\mathrm{Leb}$-measure satisfying
\begin{align*}
\liminf_{n\to\infty}
\PP\bigg(
\frac{1}{n}\sum_{j=0}^{n-1}
\ind_{I_i(\gamma_i)}(f_\omega^j(x))
\leq\alpha
\bigg)
&<1,
\quad
\text{for every $x\in B_i$ and $i=0,1$.}
\end{align*}
\end{defi}

\begin{rem} \label{rem:SF-FF} Proposition~\ref{prop:skew-arc-fiber-arc-a} gives the following relation among the three fluctuation laws:
$$
\text{skew-product fluctuation} \  \implies \   \text{fiberwise fluctuation} \  \implies  \  \text{pointwise-fiber fluctuation}.
$$
\end{rem}
}

We give sufficient conditions for the {fiberwise fluctuation law.}

\begin{prop} \label{prop:tt}
 Let $\xi:(0,1)\to (0,1)$ be a measurable function such that there exist sets $B_0,B_1 \subset I$ with positive $\lbb$-measure, and constants $\gamma_0, \gamma_1\in(0,1)$ satisfying that
 $$\xi(x_0) \leq \gamma_0 \quad \text{{and}} \quad  \gamma_1 \leq\xi(x_1) \quad \text{for every $x_0\in B_0$ and $x_1\in B_1$}$$
 and
\begin{equation*}
    \liminf_{n\to\infty} \,\PP\bigg(\,\s\ind_{L_i(\xi(x))}(f^j_\omega(x)) \leq \alpha\,\bigg)<1, \ \
\text{for every $\alpha \in (0,1)$, $x\in B_i$ and $i=0,1$,}
\end{equation*}
Then $F$ satisfies the {fiberwise fluctuation law} with constants $\gamma_0$ and $\gamma_1$.
\end{prop}
\begin{proof}
Since $\xi(x_0) \leq \gamma_0$ and $\gamma_1 \leq \xi(x_1)$, we have
$\ind_{L_i(\xi(x_i))} \leq \ind_{I_i(\gamma_i)}$ for $i=0,1$
Then,
$$
\PP\bigg(\,\s\ind_{I_i(\gamma_i)}(f^j_\omega(x_i)) \leq \alpha\,\bigg) \leq
\PP\bigg(\,\s\ind_{L_i(\xi(x_i))}(f^j_\omega(x_i)) \leq \alpha\,\bigg) \quad i=0,1.
$$
{Taking $\liminf$ and by the assumption, we get that $F$ satisfies the fiberwise fluctuation law.}
\end{proof}

\begin{rem}
  \label{rem:pointwise-fiber}
The same implication holds pointwise: if the hypothesis of Proposition~\ref{prop:tt} is assumed at two points $x_0,x_1$, then $F$ satisfies the pointwise-fiber fluctuation law with $\gamma_0=\xi(x_0)$ and $\gamma_1=\xi(x_1)$.
\end{rem}

\begin{cor}\label{cor:fiber-implies-weak-arcsine}
    Let $\xi:(0,1)\to (0,1)$ be a continuous non-constant function such
    \begin{equation*}
    \liminf_{n\to\infty} \,\PP\bigg(\,\s\ind_{L_i(\xi(x))}(f^j_\omega(x)) \leq \alpha\,\bigg)<1, \ \
\text{for every $\alpha \in (0,1)$, $x\in (0,1)$ and $i=0,1$}.
\end{equation*}
Then, {$F$ satisfies the fiberwise fluctuation law with constants $0<\gamma_0<\gamma_1<1$.}
\end{cor}
\begin{proof}
Since $\xi$ is non-constant, there is $\tilde{x}_0, \tilde{x}_1 \in (0,1)$ such that $\xi(\tilde{x}_0)<\xi(\tilde{x}_1)$. Moreover, by the continuity of $\xi$, we find closed neighborhoods $B_0,B_1 \subset I$ of $\tilde{x}_0, \tilde{x}_1$ respectively such that
$\gamma_0= \max \{\xi(x): x\in B_0\} < \min \{\xi(x): x\in B_1\}= \gamma_1$. 
Since $B_0,B_1$ have positive $\lbb$-measure, Proposition~\ref{prop:tt} yields the fiberwise fluctuation law.
\end{proof}

The next result gives the fiberwise fluctuation law by conjugation.  { For $s\in \R$, set
\[
\bar I_0(s)\eqdef (-\infty,s],\qquad \bar I_1(s)\eqdef [s,\infty) \qquad
\bar J_0(s)\eqdef (-\infty,s),\qquad \bar J_1(s)\eqdef (s,\infty).
\]
As before, we use the unified notation $\bar L_i(s)$ to denote either $\bar I_i(s)$ or $\bar{J}_i(s)$ for $i=0,1$.

\begin{prop} \label{prop:arcsine-to-conjugateA}
Let $\mathcal O\subset I$ be forward $f_\omega$-invariant for $\mathbb P$-a.e.\ $\omega\in \Omega$ and let
$h:\mathcal O\to h(\mathcal O)\subset\mathbb R$
be a strictly monotone injection. Define $g_\omega\eqdef h\circ f_\omega\circ h^{-1}$ on $h(\mathcal O)$. Fix $x\in\mathcal O$, $\gamma\in(\inf\mathcal O,\sup\mathcal O)$
and put $t\eqdef h(x)$.
Define
\[
c_\gamma\eqdef
\begin{cases}
h(\gamma), & \gamma\in\mathcal O,\\[2mm]
\sup\{h(y):y\in\mathcal O,\ y<\gamma\},
    & \gamma\notin\mathcal O \text{ and } h \text{ is increasing},\\[2mm]
\inf\{h(y):y\in\mathcal O,\ y<\gamma\},
    & \gamma\notin\mathcal O \text{ and } h \text{ is decreasing}.
\end{cases}
\]
%For each $i\in\{0,1\}$, the half-line below is taken closed at
%$c_\gamma$ if $c_\gamma\in h(\mathcal O\cap L_i(\gamma))$, and open otherwise.

Then for every $j\ge0$ and $\mathbb P$-a.e.~$\omega\in \Omega$, we have
\[
\ind_{L_i(\gamma)}(f_\omega^j(x))
=
\begin{cases}

\ind_{\bar L_i(c_\gamma)}(g_\omega^j(t)),
& \text{if $h$ is increasing},\\[4pt]

\ind_{\bar L_{1-i}(c_\gamma)}(g_\omega^j(t)),
& \text{if $h$ is decreasing},
\end{cases}
\qquad i\in\{0,1\}.
\]
\end{prop}

\begin{proof}
Since $\mathcal O$ is forward invariant we have for $\PP$-a.e.\ $\omega\in\Omega$ and every $j\ge0$ the conjugacy relation $g_\omega^j(t)=h(f_\omega^j(x))$. If $h$ is increasing, strict monotonicity and the definition of
$c_\gamma$ give
$h(\mathcal O\cap L_i(\gamma))
=
h(\mathcal O)\cap\bar L_i(c_\gamma)$. 
%where the endpoint convention is the one specified in the statement.
Hence
$f_\omega^j(x)\in L_i(\gamma)$ if and only if $g_\omega^j(t)\in\bar L_i(c_\gamma)$. 
If $h$ is decreasing, the inequalities reverse and similarly $h(\mathcal O\cap L_i(\gamma))
=
h(\mathcal O)\cap\bar L_{1-i}(c_\gamma)$,
which gives
$
f_\omega^j(x)\in L_i(\gamma)$ if and only if $
g_\omega^j(t)\in\bar L_{1-i}(c_\gamma)$. 
The identities of indicator functions follow.
\end{proof}
}

\subsection{Asymptotic occupational times}
\label{sec:essential-condition-H3ab}
Consider the following conditions for $F$:
\begin{enumerate}[itemsep=0.5cm,leftmargin=1.4cm,label=(OT\arabic*)]
	\item\label{H3a}
	There are $x_0,x_1,\gamma_0,\gamma_1 \in (0,1)$  {such that}
	\begin{equation*}
	\limsup_{n \to \infty} \frac{1}{n} \sum_{j=0}^{n-1} \ind_{I_i(\gamma_i)}(f^j_\omega(x_i)) = 1 \quad \text{for $\PP$-\aep~$\omega \in \Esp$}, \quad i = 0, 1.
\end{equation*}
	\item\label{H3b}
There exist  sets $B_0, B_1 \subset I$ of positive $\lbb$-measure and constants $\gamma_0,\gamma_1\in (0,1)$
such that for every $x_0 \in  B_0$ and $x_1\in B_1$
\begin{equation*}
	\limsup_{n \to \infty} \frac{1}{n} \sum_{j=0}^{n-1} \ind_{I_i(\gamma_i)}(f^j_\omega(x_i)) = 1 \quad
    \text{for $\PP$-a.e.~$\omega\in \Esp$},
    \quad i = 0, 1.
\end{equation*}
\end{enumerate}

\begin{rem}
The fluctuation laws impose distributional conditions at
deterministic times, whereas~\rmref{H3a}--\rmref{H3b}
concern almost-sure upper occupation times. Under the
ergodicity assumptions used below, the former imply the
latter; the converse does not follow directly from the
definitions. We use fluctuation laws because they can be
verified through probabilistic limit theorems, notably
the arcsine law. These conditions are related to \emph{intermittency} as studied
in~\cite{AbbasiGharaeiHomburg18} where typical orbits make infinitely many excursions away from a distinguished set while the proportion of time spent near that set converges to one. However, our setting describes
two-sided occupation-time oscillations, rather than
asymptotically full occupation of a single region.
\end{rem}

The following proposition and Corollary~\ref{cor:weak-pointwise-to-OT1} in the next section derive these occupation times from the preceding fluctuation laws.

\begin{prop}\label{prop:equivalence arcsine law}
Let $F$ be a skew product $F:\Omega \times I \to \Omega \times I$. 
    If $\PP \times \lbb$ is ergodic and  {$F$ satisfies the fiberwise fluctuation law,  then  \rmref{H3b} holds.}

\end{prop}

\begin{proof} Assume that $F$ satisfies the fiberwise
fluctuation law. Hence, there exist constants
$\gamma_0,\gamma_1\in(0,1)$ such that, for every $\alpha\in(0,1)$ and
$i=0,1$, there exists a set $B_i(\alpha)\subset I$ of positive
$\lbb$-measure satisfying
\begin{equation}\label{eq:pp}
    \liminf_{n\to\infty}
    \PP\bigg(
        \s\ind_{I_i(\gamma_i)}(f_\omega^j(x))
        \leq \alpha
    \bigg)<1,\quad\text{for every $x\in B_i(\alpha)$.}
\end{equation}
Fix $\alpha\in(0,1)$ and write $B_i\eqdef B_i(\alpha)$ for $i=0,1$. Define
\begin{align*}
         A_i(\alpha) &\eqdef\la{
         (\omega,x)\in\Esp \times I\colon \,  \limsup_{n\to\infty}\s\ind_{I_i(\gamma_i)}(f^j_\omega(x))
         \geq \alpha } \\
B^n_i(\alpha)&\eqdef\la{(\omega,x)\in \Esp \times B_i\colon\s\ind_{I_i(\gamma_i)}(f^j_\omega(x)) >\alpha} \quad \text{ for $n\geq1$},
   \end{align*}
and, for $x\in B_i$, denote its fiber section by
\[
B_i^n(\alpha,x)
\eqdef
\la{
\omega\in\Esp\colon
\s\ind_{I_i(\gamma_i)}(f_\omega^j(x))
>\alpha
}.
\]
Since $\limsup_{n\to\infty}B_i^n(\alpha)\subset A_i(\alpha)$, by Fatou's inequality for events, Fubini's theorem and~\eqref{eq:pp},
\begin{align*}
(\PP\times\lbb)(A_i(\alpha))
&\geq
\int_{B_i}
\PP\left(
\limsup_{n\to\infty}B_i^n(\alpha,x)
\right)\,d\lbb
\geq
\int_{B_i}
\limsup_{n\to\infty}
\PP(B_i^n(\alpha,x))\,d\lbb
>0.
\end{align*}
Moreover, $F^{-1}(A_i(\alpha))=A_i(\alpha)$, since deleting the first term of an occupation average does not change its limsup. Hence, by ergodicity of
$\PP\times\lbb$, $(\PP\times\lbb)(A_i(\alpha))=1.$
Let $\{\alpha_k\}_{k\geq1}\subset(0,1)$ satisfy $\alpha_k\to 1$.
Since each $A_i(\alpha_k)$ has full $(\PP\times\lbb)$-measure,
\[
\limsup_{n\to\infty}
\s\ind_{I_i(\gamma_i)}(f_\omega^j(x))
=1\quad\text{for $(\PP\times\lbb)$-a.e.~$(\omega,x)\in\Esp\times I$ and $i=0,1$.}
\]
Finally, by Fubini's theorem, for each $i=0,1$ there exists a set
$\Bar B_i\subset I$ of full $\lbb$-measure such that, for every
$x_i\in\Bar B_i$,
\[
\limsup_{n\to\infty}
\s\ind_{I_i(\gamma_i)}(f_\omega^j(x_i))
=1
\qquad
\text{for $\PP$-a.e.~$\omega\in\Esp$}.
\]
This proves~\rmref{H3b}.
\end{proof}

%% file: Sec5.tex
\section{Proof of Theorem~\ref{thm:B-random-walks}}\label{Sec:theorema A}

Given a bounded Borel measurable function
$\fu{\f}{I}{\mathbb{R}}$, we define
 \begin{equation*}\label{eq:def-Lyapunov}
	\begin{split}
		U_\f(\omega, x)  \eqdef\limsup_{n\to\infty}\s \f(f^j_{\omega}(x))\quad\text{and}\quad
		L_\f(\omega, x)  \eqdef\liminf_{n\to\infty}\s \f(f^j_{\omega}(x)).
	\end{split}
 \end{equation*}
Both are constant along forward $F$-orbits and hence are \emph{first integrals} of $F$. The next results establish properties of $U_\varphi$ and $L_\varphi$ in our setting.

\subsection{First integral functions for one-step skew products}
\label{sec:constant-functions}
To prove~Theorem~\ref{thm:B-random-walks}, we first show that the first integrals are constant, and then evaluate these constants. In what follows, $F$ denotes a skew product as in~\eqref{one-skew product-principal}, with fiber dynamics $f^n_\omega$ given by~\eqref{eq:fiber-maps-one-step}.

\begin{prop}\label{prop:nondepend-omega} Assume that $F$ satisfies~\rmref{H0}. Then, there exist two real-valued functions ${u}$, $\ell$ on~$I$ such that
		$U_\f(\omega, x)={u}(x)$ and $L_\f(\omega, x)=\ell(x)$  for every $x\in (0,1)$ and $\PP$-$\mathrm{a.e.}$ $\omega\in\Esp$.
\end{prop}

\begin{proof}
For fixed $x$ and $b\in\mathbb R$, the event
$
A(b,x)\eqdef\{\omega\in \Esp\colon U_\f(\omega,x)<b\}
$
is a tail event. Hence~\rmref{H0}
gives $\PP(A(b,x))\in\{0,1\}$. Let
\[
u(x)=\inf\{b:\PP(A(b,x))=1\}.
\]
Since the function $b\mapsto A(b,x)$ is increasing, for every $m\ge1$ we have $\PP(A(u(x)+1/m,x))=1$ and $\PP(A(u(x)-1/m,x))=0$.
Taking the countable intersection of these full-measure events yields
$U_\f(\omega,x)=u(x)$ for $\PP$-a.e.~$\omega\in\Esp$.
Applying the same argument to
$\bar A(b,x)=\{\omega\in\Esp:L_\f(\omega,x)>b\}$ gives
$L_\f(\omega,x)=\ell(x)$ for $\PP$-a.e.~$\omega\in\Esp$, where
$\ell(x)=\sup\{b:\PP(\bar A(b,x))=1\}$.
\end{proof}

\begin{prop}\label{prop:constant-random orbit}
Let ${u}$, $\ell$ be the functions given in Proposition~\ref{prop:nondepend-omega}. Then, for every $x\in(0,1)$,
\begin{equation*}
    {u}(x)={u}(f^k_\omega(x))
    \quad\text{and}\quad
    \ell(x)=\ell(f^k_\omega(x))
    \quad     \text{for every $\omega\in\Omega$ and $k\geq0$}.
\end{equation*}
\end{prop}

\begin{proof}
We only prove the statement for $u$. The proof for $\ell$ is similar
and hence omitted. For every $x\in I$, by
Proposition~\ref{prop:nondepend-omega}, we can choose a full-measure set
$\Omega_x$ on which $U_\f(\omega,x)=u(x)$.
For $n\geq1$, let $\mathcal A^n$ denote the set of all words of length
$n$, and define
\begin{equation}\label{eq:set-sigmafinito}
		A_n\eqdef\bigcap_{\omega\in\Esp}\sigma^{-n}\big(\Omega_{f^{n}_\omega(x)}\big)=\bigcap_{\bar{w}=
  w^{}_0\dots w^{}_{n-1}\in\mathcal{A}^n}\sigma^{-n}\big({\Omega^{}_{f^{}_{w^{}_{n-1}}\circ\dots\circ f^{}_{w^{}_{0}}(x)}}\big).
	\end{equation}
Since $\mathcal A$ is countable and $\PP$ is $\sigma$-invariant,
$\PP(A_n)=1$ for every $n\geq1$. Hence $\Lambda\eqdef\Omega_x\cap\bigcap_{n\geq1}A_n$
has full measure. Now, given a word $\bar w=w_0\dots w_{k-1}$ of size $k$,
by the full-support assumption on $\PP$, its cylinder has positive
measure and therefore intersects $\Lambda$.
Choose $\omega$ in this intersection. Since $\omega\in A_k$ and its
first $k$ symbols form the word $\bar w$, we have
$\sigma^k\omega\in\Omega_{f_{\bar w}(x)}$ and
$f_\omega^k(x)=f_{\bar w}(x)$.
Moreover, since $\omega\in\Omega_x$ and $U_\f$ is constant along
forward $F$-orbits,
\[
u(f_\omega^k(x))
=
U_\f(\sigma^k(\omega),f_\omega^k(x))
=
U_\f(\omega,x)
=
u(x).
\]
Since $\bar w$ is arbitrary, we have $u(f_{\bar w}(x))=u(x)$ for every finite word $\bar w$. Consequently,
$u(f_\omega^k(x))=u(x)$ for every $\omega\in\Omega$ and $k\geq0$.
\end{proof}

\begin{rem}\label{rmk:contableassumtion}
In Propositions~\ref{prop:nondepend-omega} and~\ref{prop:constant-random orbit},
the fiber space $M$ may be replaced by an arbitrary measurable space,
since only measurability of the fiber maps is used. Moreover, the
countability of $\mathcal A$ is not needed for
Proposition~\ref{prop:nondepend-omega}, which remains valid when
$(\mathcal A,\mathcal F,p)$ is an arbitrary probability space.
In contrast, countability of $\mathcal A$ is essential in
Proposition~\ref{prop:constant-random orbit}, since it ensures that
the intersection defining $A_n$ in~\eqref{eq:set-sigmafinito} is
countable.
\end{rem}

\begin{prop}\label{prop:constants-limits}
	Assume that $F$ satisfies conditions~\rmref{H0}--\rmref{H2} and consider  a monotonically increasing function $\varphi:I\to \mathbb{R}$. Let ${u}$, $\ell$ be the functions given in Proposition~\ref{prop:nondepend-omega}.     Then there exist  constants $\bar{{u}},\bar{\ell}\in\R$ such that
$u(x)=\bar u$ and $\ell(x)=\bar\ell$ for every $x\in(0,1)$.
\end{prop}
\begin{proof}
	Again, we only prove the proposition for the function ${u}$ and omit the details for $\ell$.  By~\rmref{H2},
	given any $x\in (0,1)$,
	there are
	$\alpha,\beta\in\Esp$  such that
	$f_\alpha(x) < x < f_\beta(x)$.

	\begin{claim}
		The function ${u}$ is  increasing.
	\end{claim}
	\begin{proof}

        Consider $x_1,x_2\in (0,1)$ with $x_1<x_2$. Take $\omega\in\Esp$, then by~\rmref{H1}
and the monotonicity of $\f$ we have $\f(f_\omega^n(x_1))\le\f(f_\omega^n(x_2))$ for every $n$. Taking Birkhoff averages and limsups, and using Proposition~\ref{prop:nondepend-omega}, yields $u(x_1)\le u(x_2)$.
	\end{proof}

Since the Bernoulli measure has full support and the skew product is one-step, Proposition~\ref{prop:constant-random orbit} with $k=1$ implies ${u}(f_{\omega_0}(x))={u}(x)$ for every one-step fiber map $f_{\omega_0}$. Hence, the~$\alpha,\beta\in\Omega$ given by~\rmref{H2} satisfy ${u}(f_\alpha(x))={u}(x)={u}(f_\beta(x))$. The claim then implies that $u$ is constant on $(f_\alpha(x),f_\beta(x))$. Since $x$ is arbitrary and belongs to this interval, ${u}$ is locally constant on~$(0,1)$. As $(0,1)$ is connected, there exists $\bar{{u}}\in\R$ such that ${u}(x)=\bar{{u}}$ for every~$x\in(0,1)$.
\end{proof}

\begin{cor}\label{coro:constant-FunctionsLH}
	Let $F$ be a one-step skew product  as  in~\eqref{one-skew product-principal} satisfying conditions
 \rmref{H0}--\rmref{H2} and consider a monotonically increasing function $\varphi\colon I\to \mathbb{R}$. Then, there exist  constants $\bar{{u}},\bar{\ell}\in\R$ such that for every $x\in(0,1)$ and  $\PP$-a.e.~$\omega\in\Esp$, it holds
		$U_\f(\omega, x)=\bar{{u}}$ and
        $L_\f(\omega, x)=\bar{\ell}$.
\end{cor}

\begin{proof}
	By Proposition~\ref{prop:nondepend-omega}, we have that $U_\f(\omega,x)={u}(x)$ and $L_\f(\omega,x)=\ell(x)$ for $\PP$-a.e.~$\omega\in\Esp$. Now, by Proposition~\ref{prop:constants-limits}, ${u}(x)=\bar{{u}}$ and $\ell(x)=\bar{\ell}$ for every $x\in (0,1)$, proving the corollary.
\end{proof}

\begin{cor}\label{cor:weak-pointwise-to-OT1} If~\rmref{H0}--\rmref{H2} and the
pointwise-fiber fluctuation law hold, then $F$ satisfies~\rmref{H3a}. 
\end{cor}

\begin{proof}
Let $\gamma_0,\gamma_1\in(0,1)$ be the constants
in the pointwise-fiber fluctuation law, and consider
the bounded monotonically increasing functions
$\varphi_0=
1-\ind_{I_0(\gamma_0)}
=\ind_{(\gamma_0,1]}$ and 
$\varphi_1=
\ind_{I_1(\gamma_1)}$.  Since $U_{\ind_{I_0(\gamma_0)}}
=1-L_{\varphi_0}$ and $U_{\ind_{I_1(\gamma_1)}}
=U_{\varphi_1}$, by Corollary~\ref{coro:constant-FunctionsLH},
there are $c_0,c_1\in[0,1]$
such that
\[
U_{\ind_{I_i(\gamma_i)}}(\omega,x)=c_i
\quad
\text{for every $x\in(0,1)$ and $\PP$-a.e.~$\omega\in\Esp$},
\qquad i=0,1.
\]
We claim that $c_0=c_1=1$.
Otherwise, suppose that $c_i<1$ for some
$i\in\{0,1\}$, and choose
$\alpha\in(c_i,1)$. By the pointwise-fiber fluctuation law,
there exists $x_i=x_i(\alpha)\in(0,1)$
such that
\[
\liminf_{n\to\infty}
\PP\bigg(
\frac1n\sum_{j=0}^{n-1}
\ind_{I_i(\gamma_i)}(f_\omega^j(x_i))
\leq\alpha
\bigg)<1.
\]
On the other hand, since
$\limsup_{n\to\infty}
\frac1n\sum_{j=0}^{n-1}
\ind_{I_i(\gamma_i)}(f_\omega^j(x_i))
=c_i<\alpha$ for $\PP$-a.e.~$\omega\in\Esp$, 
we have
\[
\lim_{n\to\infty}
\PP\bigg(
\frac1n\sum_{j=0}^{n-1}
\ind_{I_i(\gamma_i)}(f_\omega^j(x_i))
\leq\alpha
\bigg)=1,
\]
contradicting the fluctuation law. In particular,~\rmref{H3a} holds.
\end{proof}

\subsection{Historical behavior from one-step skew products} Denote by $\mathrm{id}$ the identity function. 
\begin{thm}\label{thm:theorem-H3a}
	Let $F$ be a one-step skew product as in~\eqref{one-skew product-principal} satisfying conditions~\mbox{\rmref{H0}--\rmref{H2}}. We also assume that $F$ satisfies~\rmref{H3a} with constant $\gamma_0,\gamma_1 \in (0,1)$ where $\gamma_0 <  \gamma_1$.
	Then
	\begin{align*}
		L_{\mathrm{id}}(\omega,x) \leq \gamma_0 \quad \text{and} \quad   \gamma_1 \leq  U_\mathrm{id}(\omega,x)
		  \quad \text{for every $x\in (0,1)$ and  $\PP$-a.e.~$\omega\in\Esp$.}
	\end{align*}
	In particular, $F$ has historical behavior for $(\PP\times\lbb)$-a.e.~point.
\end{thm}

\begin{proof}
By Corollary~\ref{coro:constant-FunctionsLH}, the functions $U_\mathrm{id}$ and $L_\mathrm{id}$ are constant. We first evaluate $U_\mathrm{id}$.

\begin{lem}\label{l.gkleL}
For every $x\in (0,1)$ it holds $U_\mathrm{id}(\omega,x)  \geq \gamma_1$ for $\PP$-a.e.~$\omega\in\Esp$.
\end{lem}

	\begin{proof}By~\rmref{H3a}, there exist $x_1\in(0,1)$ and $\Omega^+ \subset \Omega$ with $\mathbb{P}(\Omega^+)=1$
    such that
		\begin{equation}\label{eq:H3-for-the sequence}
			\limsup_{n\to \infty}\s\ind_{I_1(\gamma_1)}(f^j_\omega(x_1))=1\quad\text{for every $\omega\in\Omega^+$.}
		\end{equation}
		Since $\ind_{I_1(\gamma_1)}(f^j_\omega(x_1))\cdot\gamma_1\leq f^j_\omega(x_1)$ for every $\omega\in\Omega^+$, applying the upper limit obtain
		\begin{align*}
			\limsup_{n\to\infty}\,\s \ind_{I_1(\gamma_1)}(f^j_\omega(x_1))\cdot\g_1\leq\limsup_{n\to \infty}\s f^j_{\omega}(x_1)=U_{\mathrm{id}}(\omega,x_1).
		\end{align*}
		By \eqref{eq:H3-for-the sequence} it follows that $\gamma_1\leq U_{\mathrm{id}}(\omega,x_1)$ for every $\omega\in\Omega^+$.

		By Corollary~\ref{coro:constant-FunctionsLH}, there is a constant $\bar{u}$ such that, for every $x\in(0,1)$, there exists $\Omega_x\subset\Esp$ with $\PP(\Omega_x)=1$ and $U_\mathrm{id}(\omega,x)=\bar{u}$ for every $\omega\in\Omega_x$. Let $\Omega_x^+=\Omega_x\cap\Omega^+$. Since both sets have probability one, $\PP(\Omega_x^+)=1$. Moreover, for any $\omega\in\Omega^+_{x_1}$, $\gamma_1\leq U_\mathrm{id}(\omega,x_1)=\bar{u}$. Hence, $\gamma_1\leq U_\mathrm{id}(\omega,x)$ for every $\omega\in\Omega_x^+$, proving the lemma.
	\end{proof}

	The proof of the statement for the function $L_\mathrm{id}$ is a variation of
	the proof of Lemma~\ref{l.gkleL}.

\begin{lem}\label{lm:lower-ergodic}
   For $\lbb$-a.e.~$x\in (0,1)$ it holds $L_{\mathrm{id}}(\omega,x)\leq \gamma_0$ for $\PP$-a.e.~$\omega\in\Esp$.
\end{lem}

 Lemmas~\ref{l.gkleL} and~\ref{lm:lower-ergodic} give the first assertion. In particular, since the lower and the upper Lyapunov function are different, $F$ has historical behavior for $(\PP\times \lbb)$-a.e.~point.
\end{proof}

\subsection{Proof of Theorem~\ref{thm:B-random-walks}}
As $F$ satisfies the {pointwise-fiber fluctuation law,}  by Corollary~\ref{cor:weak-pointwise-to-OT1}, ~\rmref{H3a} holds. Then, by Theorem~\ref{thm:theorem-H3a}, $F$ has historical behavior for~$(\PP\times\lbb)$-a.e.~point.

%% file: Sec6.tex
\section{Proof of Theorem~\ref{thm:ergodic-assumption}}\label{Sec:proof-thm-ergodicskew}

\begin{thm}\label{thm:theorem-H3b}
Let  $F$ be a skew product as in~\eqref{skew product-principal} satisfying that $\PP\times\lbb$ is ergodic  and \rmref{H3b} with constants $\gamma_0,\gamma_1\in(0,1)$ where  $\gamma_0<\gamma_1$.
Then
	\begin{align*}
		 \liminf_{n\to\infty}\s f^j_{\omega}(x)
        \leq \gamma_0 \quad \text{and} \quad  \gamma_1 \leq
        \limsup_{n\to\infty}\s f^j_{\omega}(x)
	\end{align*}
    for $(\PP\times \lbb)$-a.e.~$(\omega,x)\in\Esp \times I$.
	In particular, $F$ has historical behavior for $(\PP\times\lbb)$-a.e.~point.
 \end{thm}
\begin{proof}
First, observe that by~\rmref{H3b}, for each $i=0,1$, the set of $(\omega,x)\in \Omega \times I$ for which
\begin{equation} \label{eq:H3-omega_k}
	\limsup_{n\to \infty}\s\ind_{I_i(\gamma_i)}(f^j_\omega(x))=1
\end{equation}
has positive $(\PP\times \lbb)$-measure.
Since these sets are $F$-invariant and $\PP\times \lbb$ is ergodic, \eqref{eq:H3-omega_k} holds for $i=0,1$ and for every $(\omega,x)$ in a set $E\subset \Omega\times I$ with $(\PP\times \lbb)(E)=1$. Moreover, since $1=\ind_{I_0(\gamma_0)}+\ind_{(\gamma_0,1]}$,   we also have that  for every $(\omega,x)\in  E$,
\begin{equation} \label{eq:liminfo}
    \liminf_{n\to\infty}\s\ind_{(\gamma_0,1]}(f^j_\omega(x))=0.
\end{equation}
Since $ \ind_{I_1(\gamma_1)}(f^j_\omega(x))\cdot\gamma_1\leq f^j_\omega(x)$, taking average and applying the upper limit, we obtain
\begin{align*}
	\limsup_{n\to\infty}\,\s \ind_{I_1(\gamma_1)}(f^j_\omega(x))\cdot\gamma_1\leq\limsup_{n\to \infty}\s f^j_{\omega}(x) \eqdef U_\mathrm{id}(\omega,x).
\end{align*}
By \eqref{eq:H3-omega_k}, it follows that $\gamma_1\leq U_\mathrm{id}(\omega,x)$ for every $(\omega,x)\in E$.

Similarly,   since
$		f^j_\omega(x)\leq\ind_{I_0(\gamma_0)}(f^j_\omega(x))\cdot\gamma_0+\ind_{(\gamma_0,1]}
        (f^j_\omega(x))$, taking average and
applying the lower limit,~\eqref{eq:H3-omega_k} and~\eqref{eq:liminfo},
we obtain $L_\mathrm{id}(\omega,x) \leq \gamma_0$ for every $(\omega,x)\in E$. 
\end{proof}

\subsection{Proof of Theorems~\ref{thm:ergodic-assumption}}
Since $F$ satisfies the {skew-product fluctuation law,  by Remark~\ref{rem:SF-FF} and} Proposition~\ref{prop:equivalence arcsine law}, condition~\rmref{H3b} holds. Hence Theorem~\ref{thm:theorem-H3b} implies that $F$ has historical behavior for~$(\PP \times \lbb)$-a.e.~point.

%% file: Sec7.tex
\section{Proof of Proposition~\ref{mainpropo-limitset} }\label{ss:limit-set}

{

We first interchange quantifiers in~\eqref{OT}:

\begin{lem}\label{equivalent-OT}
    Condition~\eqref{OT} is equivalent to the following: for $(\mathbb{P}\times \lbb)$-a.e.~$(\omega,x)\in \Omega \times I$,
    \begin{equation}  \label{OTT}
            \lim_{n\to\infty}\frac{1}{n} \sum_{j=0}^{n-1} \ind_{[\epsilon,1-\epsilon]}(f^j_\omega(x))=0  \quad \text{for every \ $0<\epsilon<1/2$}.
    \end{equation}
\end{lem}
\begin{proof} It is clear that~\eqref{OTT} implies~\eqref{OT}. Let us show the converse.   For each rational number $\epsilon \in \mathbb{Q} \cap (0, 1/2)$, let $S_\epsilon$ be the full-measure set where the limit~\eqref{OT} is zero. The intersection $S_0 = \bigcap_{\epsilon \in \mathbb{Q} \cap (0,1/2)} S_\epsilon$, being a countable intersection of full-measure sets, also has full $(\mathbb{P}\times\lbb)$-measure.  Let $(\omega,x) \in S_0$. We must show that the limit condition~\eqref{OT} holds for all (not just rational) $\epsilon \in (0,1/2)$. To do this, fix $\epsilon \in (0,1/2)$. Since rationals are dense in reals, we can choose a rational number $q$ such that $ 0<q < \epsilon$. This implies $[\epsilon, 1-\epsilon] \subset [q, 1-q]$, and thus $\ind_{[\epsilon,1-\epsilon]} \le \ind_{[q,1-q]}$. Then
\[
0 \le \frac{1}{n}\sum_{j=0}^{n-1} \ind_{[\epsilon,1-\epsilon]}(f^j_\omega(x)) \le \frac{1}{n}\sum_{j=0}^{n-1} \ind_{[q,1-q]}(f^j_\omega(x)).
\]
Since $(\omega,x) \in S_0 \subset S_q$, taking the limit as $n\to\infty$ shows that the right-hand side converges to 0 and thus the limit of the left-hand side is also 0. As $\epsilon$ was arbitrary, condition~\eqref{OTT}~holds.
\end{proof}

\begin{prop} \label{prop:limit_set_from_KD}
Let $F$ be a skew product as in~\eqref{skew product-principal} and consider a point $(\omega,x)\in \Omega \times I$ for which~\eqref{OTT} holds. Then
$
\mathcal L(\omega,x)\;\subseteq\; \{\lambda\delta_1+(1-\lambda)\delta_0:\lambda\in[0,1]\}$.

\end{prop}

\begin{proof}
Let $\nu$ be an accumulation point in $\mathcal{L}(\omega,x)$, so that for some subsequence $\{n_k\}$, the empirical measures $e_{n_k}(\omega, x)$ converge to $\nu$ in the weak* topology. We aim to show that the support of $\nu$ is contained in $\{0,1\}$, which is equivalent to showing $\nu((0,1))=0$.

Let $K$ be a compact subset of $(0,1)$. A probability measure on $\R$ has at most a countable number of atoms. This implies that the set of $0<\epsilon <1/2$ for which the interval $C_\epsilon = [\epsilon, 1-\epsilon]$ is not a $\nu$-continuity set (i.e., $\nu(\{\epsilon\})+\nu(\{1-\epsilon\}) > 0$) is at most countable. Thus, since $K$ is a compact subset of $(0,1)$, we can choose $0<\epsilon < 1/2$ such that $C_\epsilon$ is a $\nu$-continuity set and $K  \subseteq C_\epsilon$.
By the Portmanteau theorem, weak* convergence gives $\nu(C_\epsilon) = \lim_{k\to\infty} e_{n_k}(\omega,x)(C_\epsilon)$. From~\eqref{OTT}, this limit is zero, i.e.,
\[
 \nu(C_\epsilon) = \lim_{k\to\infty} \frac{1}{n_k} \sum_{j=0}^{n_k-1} \ind_{[\epsilon,1-\epsilon]}(f^j_\omega(x)) = 0.
\]
Since $K \subseteq C_\epsilon$, monotonicity gives $\nu(K)=0$. As $K$ was an arbitrary compact subset of $(0,1)$, and since the open interval $(0,1)$ is a countable union of such compact sets, it follows by $\oldsigma$-additivity that $\nu((0,1)) = 0$. This shows that $\nu$ is a convex combination of the Dirac measures $\delta_0$ and $\delta_1$ and
the proposition is proven.
\end{proof}

\begin{prop} \label{mainpropo-step1}
Let $F$ be a skew product satisfying the assumptions of Theorem~\ref{thm:B-random-walks} or Theorem~\ref{thm:ergodic-assumption}.  Assume also~\eqref{OT}. Then, for $(\PP\times \lbb)$-a.e.~$(\omega,x)$,
  \begin{equation} \label{eq:01}
         L_{\mathrm{id}}(\omega, x)\eqdef \liminf_{n\to\infty}\s f^j_{\omega}(x)=
         0 \quad
         \text{and} \quad
        U_{\mathrm{id}}(\omega,x)\eqdef \limsup_{n\to\infty}\s f^j_{\omega}(x)=
        1
    \end{equation}
  Moreover, under the assumptions of Theorem~\ref{thm:B-random-walks}, \eqref{eq:01} holds for every $x\in (0,1)$ and $\mathbb{P}$-a.e.~$\omega\in\Omega$.
\end{prop}

}

\begin{proof} Under the assumption of Theorem~\ref{thm:B-random-walks}, according to Corollary~\ref{coro:constant-FunctionsLH}, there exist constants $\bar{u},\bar{\ell} \in \mathbb{R}$ such that
$L_{\mathrm{id}}(\omega,x)=\bar{\ell}$ and $U_{\mathrm{id}}(\omega,x)=\bar{u}$ for $\PP$-a.e.~$\omega\in\Omega$ and every $x\in (0,1)$.  Moreover, Corollary~\ref{cor:weak-pointwise-to-OT1} gives~\rmref{H3a}. Consider the bounded monotonically increasing functions $\varphi_0=1-\ind_{I_0(\gamma_0)}$ and $\varphi_1=\ind_{I_1(\gamma_1)}$.
By Corollary~\ref{coro:constant-FunctionsLH} and the identities $U_{\ind_{I_0(\gamma_0)}}=1-L_{\varphi_0}$ 0and
$U_{\ind_{I_1(\gamma_1)}}=U_{\varphi_1}$, 
there are constants $c_0,c_1\in[0,1]$ such that, for every $x\in(0,1)$ and $\PP$-a.e.~$\omega\in\Esp$,
\[
\limsup_{n\to\infty}\s\ind_{I_i(\gamma_i)}(f^j_\omega(x))=c_i,
\qquad i=0,1.
\]
Condition~\rmref{H3a} gives points $x_0,x_1\in(0,1)$ for which the corresponding limsups are equal to one. Hence $c_0=c_1=1$.

To prove the proposition under the assumptions of Theorem~\ref{thm:ergodic-assumption}, we use the following lemma.

\begin{lem}\label{prop:functions-U-L-constant}
    If $\PP\times \lbb$ is ergodic with respect to $F$, then, for every bounded Borel measurable function $\varphi\colon I\to\mathbb{R}$, there are constants $\bar{u},\bar{\ell}$ such that
    \begin{equation*}
        U_{ \varphi}(\omega, x)=\bar{u}\quad\text{and} \quad L_{ \varphi}(\omega, x)=\bar{\ell}\quad\text{for $(\PP\times \lbb)$-a.e.~$(\omega,x)\in\Esp\times I.$}
    \end{equation*}
\end{lem}

\begin{proof}
    Given a constant $u\in \mathbb{R}$ define the set
    $A(u)\eqdef\{(\omega,x)\in \Esp\times I\colon U_{ \varphi}(\omega,x)<u\}$.
    Since $A(u)$ is an $F$-invariant set, by ergodicity of $\PP\times \lbb$, we have $(\PP\times\lbb)(A(u))\in\{0,1\}$.
 Let
 $$\bar{u}\eqdef\inf\{u\colon (\PP\times\lbb)(A(u))=1\}.
 $$
    \begin{claim}
        \label{cl.barb1}
        $U_{ \varphi}(\omega,x)=\bar{u}$ for $(\PP\times\lbb)$-a.e.~$(\omega,x)\in\Esp\times I$.
    \end{claim}

\begin{proof}
 If $(\PP\times\lbb)(A(\bar{u}))=0$, then $U_{ \varphi}(\omega,x)\geq \bar{u}$ for $(\mathbb{P}\times \lbb)$-a.e.~$(\omega,x)$.
 Moreover, by definition of $\bar{u}$, we have that  $U_{ \varphi}(\omega,x)<\bar{u}+1/n$ for every $n\geq 1$ and $(\mathbb{P}\times \lbb)$-a.e.~$(\omega,x)$. Hence, by taking $n\to \infty$, we find that  $U_{ \varphi}(\omega,x)=\bar{u}$ for $(\mathbb{P}\times \lbb)$-a.e.~$(\omega,x)$ proving the claim in this case. 
 
 To conclude the proof, we need to analyze also the case when $(\PP\times\lbb)(A(\bar{u}))=1$.  Suppose, by contradiction, that there exists a set $B \subset \Omega\times I$ with $(\mathbb{P}\times \lbb)(B)>0$ such that $U_{ \varphi}(\omega,x)\neq{\bar{u}}$ for every $(\omega,x)\in B$. By definition of  ${\bar{u}}$, we have that $(\mathbb{P}\times \lbb)(A({\bar{u}}-\frac{1}{n}))=0$ for all $n\geq 1$. Consider the monotonically increasing sequence of sets
        \begin{equation*}
            B_n\eqdef A\big({\bar{u}}-\frac{1}{n}\big)\cap B, \quad \text{for $n\geq 1$},
        \end{equation*}
        and note that $(\mathbb{P}\times \lbb)(B_n)=0$ and therefore
        $\lim\limits_{n\to\infty}(\mathbb{P}\times \lbb)(B_n)=0$. However, by the monotonicity and since $(\mathbb{P}\times \lbb)(A(\bar{u}))=1$,
$$            \lim\limits_{n\to\infty}(\mathbb{P}\times \lbb)(B_n)=(\mathbb{P}\times \lbb)(A({\bar{u}})\cap B)=(\mathbb{P}\times \lbb)(B)>0,$$
        which leads to a contradiction.
        Hence $U_{ \varphi}(\omega,x)={\bar{u}}$ for $(\mathbb{P}\times \lbb)$-a.e.~point.
    \end{proof}

    The above claim provides $\bar{u}$ as in the statement of the proposition.  Considering the sets
    $\bar{A}(\ell)=\{ (\omega,x)\in\Esp\times I\colon L_{ \varphi}(\omega,x)>\ell\}$
    and arguing similarly, we also get that $L_{ \varphi}(\omega,x)=\bar{\ell}$ for \mbox{$(\mathbb{P}\times \lbb)$-a.e.}~$(\omega,x)$ where
 $\bar{\ell}=\sup \{\ell: (\mathbb{P}\times \lbb)(\bar{A}(\ell))=1\}$. 
\end{proof}

Now assume the hypotheses of Theorem~\ref{thm:ergodic-assumption}. By Remark~\ref{rem:SF-FF} and Proposition~\ref{prop:equivalence arcsine law}, condition~\rmref{H3b} holds. Applying Lemma~\ref{prop:functions-U-L-constant} to $\varphi=\ind_{I_i(\gamma_i)}$, there is  $c_i\in[0,1]$ such that
\[
\limsup_{n\to\infty}\s\ind_{I_i(\gamma_i)}(f^j_\omega(x))=c_i
\quad\text{for $(\PP\times\lbb)$-a.e.~$(\omega,x)$},
\qquad i=0,1.
\]
By~\rmref{H3b}, for each $i$ the same limsup is equal to one on a set of positive $(\PP\times\lbb)$-measure. Hence $c_0=c_1=1$.

Thus, in either case,
\begin{equation}\label{eq:occupation-endpoints}
\limsup_{n\to\infty}\s\ind_{I_i(\gamma_i)}(f^j_\omega(x))=1,
\qquad \text{for $(\PP\times\lbb)$-a.e.~$(\omega,x)$ and $i=0,1$.} 
\end{equation}
 By Lemma~\ref{equivalent-OT}, we can take a common full-measure set on which~\eqref{eq:occupation-endpoints} and~\eqref{OTT} hold. Fix $(\omega,x)$ in this set and take $0<\epsilon<\min\{\gamma_0,1-\gamma_1,1/2\}$.
Since $I_0(\gamma_0)=[0,\gamma_0]$ and $I_1(\gamma_1)=[\gamma_1,1]$, 
\[
\begin{aligned}
0&\leq \s\ind_{[0,\gamma_0]}(f^j_\omega(x))
      -\s\ind_{[0,\epsilon]}(f^j_\omega(x))
\leq \s\ind_{[\epsilon,1-\epsilon]}(f^j_\omega(x)),\\
0&\leq \s\ind_{[\gamma_1,1]}(f^j_\omega(x))
      -\s\ind_{[1-\epsilon,1]}(f^j_\omega(x))
\leq \s\ind_{[\epsilon,1-\epsilon]}(f^j_\omega(x)).
\end{aligned}
\]
By~\eqref{OTT} and~\eqref{eq:occupation-endpoints},
\[
\limsup_{n\to\infty}\s\ind_{[0,\epsilon]}(f^j_\omega(x))=1
\quad\text{and}\quad
\limsup_{n\to\infty}\s\ind_{[1-\epsilon,1]}(f^j_\omega(x))=1.
\]
Moreover,
\[
(1-\epsilon)\s\ind_{[1-\epsilon,1]}(f^j_\omega(x))
\leq \s f^j_\omega(x)
\leq \epsilon+1-\s\ind_{[0,\epsilon]}(f^j_\omega(x)).
\]
Therefore,
\[
0\leq\liminf_{n\to\infty}\s f^j_\omega(x)\leq\epsilon
\quad\text{and}\quad
1-\epsilon\leq\limsup_{n\to\infty}\s f^j_\omega(x)\leq1.
\]
Since $\epsilon>0$ can be taken arbitrarily small,~\eqref{eq:01} follows for $(\PP\times\lbb)$-a.e.~$(\omega,x)$.

 Finally, under the assumptions of Theorem~\ref{thm:B-random-walks}, the constants $\bar\ell$ and $\bar u$ obtained at the beginning of the proof are independent of $x$. Since the argument above gives
$L_{\mathrm{id}}(\omega,x)=0$ and  $U_{\mathrm{id}}(\omega,x)=1$
for $(\PP\times\lbb)$-a.e.~$(\omega,x)$, it follows that
$\bar\ell=0$ and $\bar u=1$. Therefore, for every $x\in(0,1)$,
\eqref{eq:01} holds for $\PP$-a.e.~$\omega\in\Omega$. This completes the proof of the proposition. 
\end{proof}

{
\begin{prop} \label{prop-delta-L} 
Let $(\omega,x)$ be in $\Omega \times I$ satisfying~\eqref{eq:01}.
%Let $(\omega,x)$ be in $\Omega \times I$ satisfying~\eqref{OTT} and~\eqref{eq:01}. 
Then  $\delta_0,\delta_1\in \mathcal{L}(\omega,x)$.
\end{prop}
}

\begin{proof}
Given $(\omega,x)$ for which~\eqref{eq:01} holds, let $\{n_k\}_{k\ge1}$ be such that
$
\lim_{k\to\infty} \frac{1}{n_k}\sum_{j=0}^{n_k-1} f^j_\omega(x)=0$.
Let $\mu$ be an accumulation point in the weak$^*$ topology of the subsequence of empirical measures $e_{n_k}(\omega,x)$. For notational simplicity, we assume that $e_{n_k} \to \mu$ as $k\to\infty$.  Then,
$$
    0=\lim_{k\to\infty} \frac{1}{n_k}\sum_{j=0}^{n_k-1} f^j_\omega(x) =\lim_{k\to\infty}\int \mathrm{id} \, de_{n_k}  \to \int \mathrm{id}\, d\mu.
$$
Thus $\mu=\delta_0$. Analogously, $\delta_1$ is an accumulation point of $e_n(\omega,x)$.
\end{proof}

{
\subsection{Proof of Proposition~\ref{mainpropo-limitset}}
Let $S_0$ and $S_1$ be the full-measure sets where, respectively,~\eqref{OTT} and~\eqref{eq:01} hold. Hence, $S=S_0 \cap S_1$ has full $(\mathbb{P}\times \lbb)$-measure and, for any $(\omega,x)\in S$ we can apply Propositions~\ref{prop:limit_set_from_KD} and~\ref{prop-delta-L} to conclude that $$\delta_0,\delta_1\in \mathcal{L}(\omega,x) \subset L\eqdef\{ \lambda \delta_0 + (1-\lambda) \delta_1 \colon \lambda \in [0,1] \}.$$
{Since $\mathcal{L}(\omega,x)$ is connected, contained in the segment $L$, and contains both endpoints, it must be the entire segment concluding the proof.}

\begin{rem} Under the assumptions of Theorem~\ref{thm:B-random-walks}, suppose that for every fixed $x\in (0,1)$ and  $0<\epsilon <1/2$, the limit in~\eqref{OT} holds $\mathbb{P}$-almost surely. Thus, the set $S_\epsilon$ in the proof of Lemma~\ref{equivalent-OT} is a bundle over $(0,1)$ with fiber sets of full $\mathbb{P}$-measure and, consequently, so is $S_0$. Similarly, in this one-step case, from Proposition~\ref{mainpropo-step1}, we have that $S_1$ is a similar bundle over $(0,1)$. Thus, $S=S_0\cap S_1$ is also a bundle over $(0,1)$ with fiber sets of full $\mathbb{P}$-measure and the observation in Remark~\ref{rem-ppp} follows.
\end{rem}

}

%% file: Sec8.tex
\section{Proof of Propositions~\ref{maincor:conjugation-random-walk}, \ref{cor:north-south},
\ref{maincor:coupling}
 {and~\ref{thm:D-Bonifant-Milnor-historical-behavior}}}
\label{ss:Proof-cor-I-II-VI}

{
 Let $F$ be a one-step skew product as in~\eqref{one-skew product-principal} and fix $x\in (0,1)$. We study historical behavior of $F$ when the Markov chain  $\{X^x_n\}_{n\geq 0}$ with $X^x_n(\omega)=f^n_\omega(x)$ is conjugate to a random walk on the  additive group $G=\mathbb{Z}$ or $\mathbb{R}$. Recall that according to Definition~\ref{def:conjugated-G-randomwalk}, this means that  there is a strict monotonic injection $h\colon \mathcal{O}(x) \to G$ such that the step random variables
$Y_n^t = S^t_n - S_{\smash{n-1}}^t$, $n \geq 1$,
are i.i.d., where $t = h(x)$ and $S^t_n(\omega) = h \circ f^n_\omega \circ h^{-1}(t)$. Here, the set $\mathcal{O}(x)$ denotes the orbit $\{X_n^x(\omega)\colon \omega \in \Omega,\, n\geq0\}$. Now, define the random walk
\begin{equation}\label{eq-sum-realwalk}
    S_0 = 0, \qquad S_n(\omega) \eqdef \sum_{j=1}^{n} Y_j^t(\omega) = S_n^t(\omega) - t \quad \text{for } n \geq 1.
\end{equation}
Notice that
\begin{equation}\label{eq:multiple-notation}
S_n^t(\omega)=t+S_n(\omega)=g^n_\omega(t) \quad \text{where} \ \ g_\omega=h\circ f_\omega \circ h^{-1}.
\end{equation}
}

\begin{rem}\label{rem:classical-random-walk-facts}
Let $\{S_n\}_{n\geq0}$ be the partial sums of i.i.d.~increments. We use the following facts.

\begin{enumerate}
\renewcommand{\labelenumi}{\textnormal{(\roman{enumi})}}

\item
The tail $\oldsigma$-algebra of $\{S_n\}_{n\geq0}$ is trivial
as a consequence of the Hewitt--Savage zero--one law; see
\cite{HeSavage55-01law}, \cite[\S26, Thm. B]{ProbLoeve},
and~\cite{TailBerHollan89}. Moreover, if the increments have mean zero
and positive finite variance, the Erd\H{o}s--Kac theorem
\cite{ErdKac47} gives the arcsine law for the proportion of indices
$j$ for which $S_j>0$.

\item
Suppose that
\[
\frac1n\sum_{j=0}^{n-1}\ind_K(S_j)\longrightarrow0
\quad\text{a.s. for every compact set $K\subset\mathbb R$.}
\]
For $s\in\mathbb R$,
define $A_n^s\eqdef
\frac1n\sum_{j=0}^{n-1}\ind_{[s,\infty)}(S_j).$
Then, for every fixed $s\in\mathbb R$,
\begin{equation}\label{eq:shift-halfline}
\left|A_n^s-A_n^0\right|
\leq
\frac1n\sum_{j=0}^{n-1}
\ind_{[\min\{0,s\},\,\max\{0,s\}]}(S_j)
\longrightarrow0
\quad\text{a.s.}
\end{equation}
The estimate also holds for all other open or closed half-lines.
\end{enumerate}
\end{rem}

\begin{lem}
\label{cla:conjugated-real-random}
Let $\{S_n\}_{n\ge 0}$ be a $G$-valued random walk starting $S_0=0$ with mean zero and positive finite variance. Then, for every $\kappa,t\in \mathbb{R}$,
			\begin{equation}\label{eq:arcsin-G-ranwalk}
		\lim_{n\to\infty}\PP\big(\,\frac{1}{n} \sum_{j=0}^{n-1}\ind_{{\Bar{I}_i(\kappa)}}(t+S_j(\omega))\leq\alpha\,\big)=\frac{2}{\pi}\arcsin{\sqrt{\alpha}}  \qquad \text{for every $\alpha \in (0,1)$ and $i=0,1,$}
			\end{equation}
where $\Bar{I}_0(\kappa)=(-\infty,\kappa]$ and $\Bar{I}_1(\kappa)=[\kappa,\infty)$.
\end{lem}

\begin{proof}
We prove~\eqref{eq:arcsin-G-ranwalk} for $i=1$; the case $i=0$ follows analogously.
Define
\begin{align}\label{eq:N-indacator-real-rw}
    N_n(\omega) &\eqdef  \#\{j \in \{0, \dots, n-1\} : S_j(\omega) > 0\} = \sum_{j=0}^{n-1} \ind_{(0,\infty)}(S_j(\omega)).
\end{align}
By assumption, the step variables $Y_n=S_{n}-S_{n-1}$, $n\geq 1$, are i.i.d.~with zero mean and positive finite variance.  By the Erd\H{o}s--Kac arcsine law~\cite{ErdKac47} and~\eqref{eq:N-indacator-real-rw}, we obtain~\eqref{eq:arcsin-G-ranwalk} {for the random process starting at $t=0$ with target set $\Bar{J}_1(0)=(0,\infty)$ instead of $\Bar{I}_1(0)=[0,\infty)$. The difference between the time averages for the closed and open intervals is the occupation time of the point $\{0\}$. Since the random walk has mean zero, it is recurrent. A classical result for one-dimensional recurrent random walks is that the occupational time spent at any compact set vanishes in the limit. Proposition~\ref{cor:main-cases} provides a new direct proof of this classical fact. Hence, since the time average over the open interval converges in distribution to the arcsine law, and the difference between the two averages converges almost surely (and thus in probability) to zero, Slutsky's theorem implies that the time average over the closed interval converges to the same arcsine distribution.

Thus $A_n^0$ converges in distribution to the arcsine law. By Proposition~\ref{cor:main-cases} and~\eqref{eq:shift-halfline}, $A_n^{\kappa-t}-A_n^0\to0$ a.s. Hence $A_n^{\kappa-t}$ has the same limiting distribution. This proves~\eqref{eq:arcsin-G-ranwalk}.}
\end{proof}

By~\eqref{eq:multiple-notation}, the previous lemma, and Proposition~\ref{prop:arcsine-to-conjugateA}, we obtain:

\begin{cor}\label{prop:arcsine-to-conjugate}
Let $F$ be a one-step skew product as in~\eqref{one-skew product-principal}  such that the sequence of random variables $\{X^x_n\}_{n\geq 0}$ is
 conjugate to a $G$-valued random walk {with mean zero and positive finite variance}. Then, for every   $\gamma\in (\inf\mathcal O(x),\sup\mathcal O(x))$,
 		\begin{align*} 	\lim_{n\to\infty}\PP\big(\,\s\ind_{{I_i(\gamma)}}(f^j_\omega(x))\leq \alpha\,\big)=\frac{2}{\pi}\arcsin\sqrt{\alpha} \qquad \text{ for every $\alpha\in(0,1)$ and $i=0,1$.}
 		\end{align*}
 \end{cor}

In the next proposition, we prove that the fiber maps satisfy~\rmref{H0}.

\begin{prop}\label{prop:conjugate-H0}
Let $F$ be a one-step skew product as in~\eqref{one-skew product-principal}  such that
the sequence of random variables $\{X^x_n\}_{n\geq 0}$ is
conjugate to a $G$-valued random walk.
{Then the tail $\oldsigma$-algebra $\mathcal{T}(\{X^x_n\}_{n\geq 1})$
is trivial. In particular, if the conjugation is established for all $x\in (0,1)$, then $F$ satisfies~\rmref{H0}.}
\end{prop}
{
\begin{proof}
Consider again the random walk $\{S_n\}_{n\geq 0}$ in~\eqref{eq-sum-realwalk}.
By the Hewitt--Savage zero--one law (e.g.~\cite[\S26, Thm. B]{ProbLoeve}), $\mathcal T(\{S_n\}_{n\ge1})$  is trivial.
Since tail $\oldsigma$-algebras are invariant under translations,
we have $\mathcal{T}(\{S_n\}_{n \geq 1}) =\mathcal{T}(\{S_n^t - t\}_{n \geq 1}) = \mathcal{T}(\{S^t_n\}_{n \geq 1})$.
Thus, $\mathcal{T}(\{S_n^t\}_{n \geq 1})$ is trivial. Finally, as the conjugation $h$ is a measurable bijection, the $\oldsigma$-algebra generated by $\{X_n^x\}_{n \geq 1}$ equals that generated by $\{S_n^t\}_{n \geq 1}$ where $t=h(x)$. Hence $\mathcal{T}(\{X^x_n\}_{n \geq 1})$ is trivial.
\end{proof}}

\subsection{Proof of Proposition~\ref{maincor:conjugation-random-walk}}
By assumption,  $\{X^x_n\}_{n \geq 1}$ is conjugate to a $G$-valued random walk for all $x\in (0,1)$. Hence Corollary~\ref{prop:arcsine-to-conjugate} gives the arcsine law for $F$. Proposition~\ref{prop:conjugate-H0} also gives~\rmref{H0}.  If $F$ also satisfies~\rmref{H1} and~\rmref{H2}, then $\inf\mathcal O(x)<x<\sup\mathcal O(x)$ by~\rmref{H2}. Thus, for any $0<\gamma_0<\gamma_1<1$, the arcsine law at $x_i=\gamma_i$ gives the pointwise-fiber fluctuation law. Theorem~\ref{thm:B-random-walks} therefore gives~\eqref{eq:liminf-limsup} and historical behavior for $(\PP\times\lbb)$-a.e.~point.

\subsection{Proof of Proposition~\ref{cor:north-south}} \label{proof:Proposition-TPsi}
{
Recall that $\Psi$ is a $\mathbb{Z}$-valued one-step random variable
and $T$ is a Morse-Smale homeomorphism of period one. Hence $f_\omega=T^{\Psi(\omega)}$ is increasing on each open interval $J$ between consecutive fixed points $p$ and $q$, so~\rmref{H1} holds. Since $\mathbb{E}[\Psi]=0$ and $\Psi\neq0$ because $0<\mathbb{E}[\Psi^2]<\infty$, we have $\alpha,\beta\in\Omega$ such that $\Psi(\alpha)\Psi(\beta)<0$. Since $T$ is Morse--Smale of period one, $T^{\Psi(\alpha)}<\mathrm{id}<T^{\Psi(\beta)}$ or vice versa on $J$. Thus~\rmref{H2} holds.
}

For any $x \in J$, we define  the strictly monotonic measurable bijection
\begin{equation*}
h \colon \mathcal{O}_T(x) \to \mathbb{Z}, \quad h(T^t(x)) \eqdef t, \quad \text{{for every} } t \in \mathbb{Z},
\end{equation*}
where $\mathcal{O}_T(x) = \{T^t(x) : t \in \mathbb{Z}\}$ is the full orbit of $x$. Since $h^{-1}(t) = T^t(x)$,
\[
g_\omega(t) = h \circ f_\omega \circ h^{-1}(t) = h(f_\omega(T^t(x))) = h(T^{t + \Psi(\omega_0)}(x)) = t + \Psi(\omega_0),
\]
{which defines a $\mathbb{Z}$-valued random walk driven by $\Psi$.
In particular, the process $X^x_n(\omega)=f^n_\omega(x)= T^{\sum_{j=0}^{n-1}\Psi(\sigma^j\omega)}(x)$ is conjugate to a random walk on $\mathbb{Z}$ with mean zero and positive finite variance.}
By Proposition~\ref{maincor:conjugation-random-walk}, $F_{T^\Psi}|_{\Omega \times J}$ exhibits historical behavior for $(\PP \times \lbb)$-almost every point in $\Omega \times J$. Since $M$ is the union of finitely many such intervals $J$, $F_{T^\Psi}$ exhibits historical behavior for $(\mathbb{P} \times \lbb)$-almost every point. Finally, by Corollary~\ref{cor:lim-set}, for any $x \in J = (p, q)$,  the limit set is given by $\mathcal{L}(\omega, x) = \{\lambda \delta_p + (1 - \lambda) \delta_q : \lambda \in [0, 1]\}$. This proves Proposition~\ref{cor:north-south}.

{
\subsection{Proof of Proposition~\ref{maincor:coupling}}

In this subsection, we study the skew product $F_{\smash{T^\Psi,Z}}$.

\begin{prop}\label{lem:H1_H2_coupling}
The skew product $F_{\smash{T^\Psi,Z}}$ in~\eqref{eq:F-coupling} satisfies conditions \rmref{H1} and \rmref{H2}.
\end{prop}
\begin{proof}
To establish condition \rmref{H1}, we must show that \(f_{\omega_0}\) is monotone increasing on $(0,1)$.  Let \(x, y \in (0,1)\) with \(x < y\). We consider two cases.

\emph{Case 1:} \(x, y \in I_k\) for some \(k \in \mathbb{Z}\). On the interval \(I_k = [p_k, p_{k+1})\), according to~\eqref{eq:coupling-fiber}, $f_{\omega_0}$ is the composition \(f_{\omega_0}(z) = v \circ T^{\smash{\Psi(\omega_0)}} \circ u(z)\), where \(u(z)=(z-p_k)/d_k\) and \(v(z)=p_{k+Z(\omega_0)} + d_{k+Z(\omega_0)}z\). The affine maps \(u,v\) are strictly increasing since \(d_k,d_{k+Z(\omega_0)}>0\); so is \(T^{\smash{\Psi(\omega_0)}}\). Their composition \(f_{\omega_0}\) is therefore strictly increasing on \(I_k\), giving \(f_{\omega_0}(x)<f_{\omega_0}(y)\).

\emph{Case 2:} \(x \in I_k\) and \(y \in I_j\) for \(k < j\). By definition of the partition, this implies \(x < p_{k+1} \le p_j < y\). The map \(f_{\omega_0}\) sends the entire interval \(I_k\) to the interval \(I_{k'}\) where \(k' = k + Z(\omega_0)\), and similarly sends \(I_j\) to \(I_{j'}\) where \(j' = j + Z(\omega_0)\). Since \(k < j\), we have \(k' < j'\). The partition is ordered, so the interval \(I_{k'}\) lies entirely to the left of \(I_{j'}\), i.e., \(\sup(I_{k'}) = p_{k'+1} \le p_{j'} = \inf(I_{j'})\). As \(f_{\omega_0}(x) \in I_{k'}\) and \(f_{\omega_0}(y) \in I_{j'}\), it follows that \(f_{\omega_0}(x) < p_{k'+1} \le p_{j'} \le f_{\omega_0}(y)\). Thus, \(f_{\omega_0}(x) < f_{\omega_0}(y)\).

In both cases \(f_{\omega_0}\) is strictly increasing, proving~\rmref{H1}. We verify~\rmref{H2}. By assumption, $\mathbb{E}[Z]=0 < \mathbb{E}[Z^2] < \infty$, so $Z$ is non-degenerate and takes both positive and negative values. Then, for every $k \in \mathbb{Z}$ and $x \in I_k$, there is $\alpha \in \Omega$ with $Z(\alpha) > 0$ such that $f_{\alpha}$ maps $x$ to $I_{k+Z(\alpha)}$, moving toward 1 and thus $f_{\alpha}(x) > x$; and $\beta \in \Omega$ with $Z(\beta) < 0$ such that $f_{\beta}$ maps $x$ to $I_{k+Z(\beta)}$, moving toward 0 and hence $f_{\beta}(x) < x$. 
This completes the verification of \rmref{H2}.
\end{proof}

By~\eqref{eq:coupling-randomwalk}, $F_{\smash{T^\Psi,Z}}$ is a coupling
$X_n^x = p_{Z_n} + d_{Z_n} \cdot u_n$ where  $u_n=T^{S_n}(u_0)$,
of a macro random walk $\{Z_n\}_{n\geq0}$ and a micro walk $\{S_n\}_{n\geq0}$ given in~\eqref{eq:process-definitions}.

\begin{prop}\label{lem:H0_coupling}
The skew product $F_{\smash{T^\Psi,Z}}$ satisfies condition \rmref{H0}, i.e., for every $x \in (0,1)$, the sequence of random variables $\{X_n^x\}_{n\geq 1}$ has a trivial tail $\oldsigma$-algebra.
\end{prop}

\begin{proof}
If $x=p_k$ for some $k \in \mathbb{Z}$, then $\{X_n^x\}_{n \geq 0}$ is conjugate to the random walk $\{Z_n\}_{n \geq 0}$. Proposition~\ref{prop:conjugate-H0} then implies that $\{X_n^x\}_{n \geq 0}$ has a trivial tail $\oldsigma$-algebra. Now consider $x \in (0,1) \setminus \{p_k : k \in \mathbb{Z}\}$.
Define $\phi \colon (0,1) \to (0,1)$ and $\kappa \colon (0,1) \to \mathbb{Z}$ by
\[
\kappa(t) \eqdef \text{the unique integer } k \text{ such that } t \in I_k, \quad \text{and} \quad \phi(t) \eqdef \frac{t - p_{\kappa(t)}}{d_{\kappa(t)}}.
\]
These functions characterize the $\oldsigma$-algebra of $\{X_n^x\}_{n\geq 0}$ in terms of the tail $\oldsigma$-algebras of $\{Z_n\}_{n\geq 0}$ and $\{u_n\}_{n\geq 0}$. The generated $\oldsigma$-algebra is
$\oldsigma(X_n^x : n \geq 0) = \oldsigma( \bigcup_{n \geq 0} \oldsigma(X_n^x))$.

\begin{lem} \label{lem:recover}
Fix $k \in \mathbb{Z}$ and $x \in I_k\setminus \{p_k\}$.
Then the following hold:
\begin{enumerate}[label=$\mathrm{(\roman*)}$]
    \item For every $n \ge 0$, $u_n \in (0,1)$, $X_n^x \in I_{Z_n}$, and $X^x_n \notin \{p_k : k \in \mathbb{Z}\}$.
          Moreover, the functions $\kappa$ and $\phi$ are Borel measurable and satisfy
          \[
          Z_n = \kappa(X^x_n) \quad \text{and} \quad u_n = \phi(X^x_n) \quad \text{for every } n \geq 0.
          \]
    \item For every $m \ge 1$, $\oldsigma(X_n^x : n \ge m) = \oldsigma\big( (Z_n, u_n) : n \ge m \big), $
          and consequently,
          \begin{equation}\label{eq:tail-X-equals-tail-Zu}
          \mathcal{T}(\{X_n^x\}_{n\geq 1}) = \mathcal{T}(\{(Z_n,u_n)\}_{n\geq 1}).
          \end{equation}
\end{enumerate}
\end{lem}

\begin{proof}
For (i),  $x \in I_k\setminus \{p_k\}$ implies $u_0=(x-p_k)/d_k \in (0,1)$. Since $T$ is a north-south homeomorphism, $T^n((0,1)) \subset (0,1)$ for all $n \in \mathbb{Z}$, and thus $u_n = T^{S_n}(u_0) \in (0,1)$ for every $n \geq 0$. Therefore, for each $n \geq 0$, $X_n^x = p_{Z_n} + d_{Z_n} \cdot u_n \in (p_{Z_n}, p_{Z_n+1}) = I_{Z_n}$. Uniqueness of the interval containing $X_n^x$ makes $\kappa$ well-defined at $t = X_n^x$ and gives $Z_n = \kappa(X_n^x)$. It follows that
\[
u_n = \frac{X_n^x - p_{Z_n}}{d_{Z_n}} = \frac{X_n^x - p_{\kappa(X_n^x)}}{d_{\kappa(X_n^x)}} = \phi(X_n^x).
\]
The Borel measurability of $\kappa$ and $\phi$ follows from the fact that $\kappa$ is piecewise constant on the measurable partition $\{I_k\}_{k \in \mathbb{Z}}$, and $\phi$ is piecewise affine on the same partition.

For (ii), Fix $m \geq 1$ and note that $(Z_n, u_n) = (\kappa(X_n^x), \phi(X_n^x))$ for every $n \geq m$. Since $\kappa$ and $\phi$ are Borel measurables, the pair $(Z_n, u_n)$ is a measurable function of $X_n^x$. This implies
$
\oldsigma(Z_n, u_n) \subset \oldsigma(X_n^x).
$
Conversely, for every $n \geq m$, the identity $X_n^x = p_{Z_n} + d_{Z_n} \cdot u_n$ shows that $X_n^x$ is a Borel function of $(Z_n, u_n)$, and hence
$
\oldsigma(X_n^x) \subset \oldsigma(Z_n, u_n).
$
Hence $\oldsigma(X_n^x)=\oldsigma(Z_n,u_n)$ for every $n\geq m$. Taking the $\oldsigma$-algebra generated by the family~$\{X_n^x : n \ge m\}$, we obtain
\[
\oldsigma(X_n^x : n \ge m) = \oldsigma\big( \bigcup_{n \geq m} \oldsigma(X_n^x) \big) = \oldsigma\big( \bigcup_{n \geq m} \oldsigma(Z_n, u_n) \big) = \oldsigma\big( (Z_n, u_n) : n \ge m \big),
\]
where the second equality follows from the identity $\oldsigma(X_n^x) = \oldsigma(Z_n, u_n)$ for each $n$. Taking the intersection over all $m \ge 1$ yields $\mathcal{T}(\{X_n^x\}_{n\geq 1}) = \mathcal{T}(\{(Z_n,u_n)\}_{n\geq 1})$.
\end{proof}
Now, as mentioned, $\{u_n\}_{n\geq0}$ is conjugate to a random walk $\{S_n\}_{n\geq0}$. That is, there is a  monotonic bijection $h \colon \mathcal{O}_T(u_0) \to \mathbb{Z}$ such that $S_n=h(u_n)$ where $\mathcal{O}_T(u_0)=\{T^t(u_0):t\in\mathbb{Z}\}$. See~\S\ref{proof:Proposition-TPsi} for more details.  Define the homeomorphism $H \eqdef \mathrm{id} \times h$ from $\mathbb{Z}\times \mathcal{O}_T(u_0)$ to~$\mathbb{Z}^2$. This map  conjugates $(Z_n, u_n)$ to $(Z_n, S_n)$, i.e., $(Z_n,S_n)=H(Z_n,u_n)$. By~\eqref{eq:process-definitions}, $(Z_n, S_n)$ is driven by the one-step random variables $Z$ and $\Psi$. This means that the step increments are
$$\Delta_n(\omega)=(Z_n(\omega)-Z_{n-1}(\omega), \, S_{n}(\omega)-S_{n-1}(\omega))=(Z(\omega_{n-1}),\Psi(\omega_{n-1})).$$
Here $\Delta_n(\omega)=\chi(\omega_{n-1})$, where $\chi(a)=(Z(a),\Psi(a))$ is measurable on $\mathcal{A}$. Since the background process $\{\omega_n\}_{n\ge0}$ is i.i.d., we also get that $\Delta_n$ is i.i.d. and thus, the joint process $(Z_n,S_n)$ is a random walk on the abelian group $\mathbb{Z}^2$. Hence the Hewitt--Savage zero--one law gives that $\mathcal T(\{(Z_n,S_n)\}_{n\ge0})$ is trivial. Thus Lemma~\ref{lem:recover} and the homeomorphism $H$ give  $\mathcal{T}(\{X_n^x\}_{n\geq 0})=\mathcal{T}(\{(Z_n,u_n)\}_{n\geq 0}) = \mathcal{T}(\{(Z_n,S_n)\}_{n\geq 0})$ is trivial, proving the proposition.
\end{proof}

\begin{lem}\label{lem:weak_arcsine_coupling} If the random variable $Z$ has mean zero and finite positive variance, then the skew product $F_{T^\Psi,Z}$ satisfies the pointwise fiber fluctuation law with constants $\gamma_0<\gamma_1$.
\end{lem}

\begin{proof} For $i\in \{0,1\}$, set $x_i=p_i$. Since the chain $\{X^{x_i}_n\}_{n\geq 0}$ is conjugate to the $\mathbb{Z}$-valued random walk $\{Z_n\}_{n\geq 0}$, which has step increments $Z_n(\omega)-Z_{n-1}(\omega)=Z(\omega_{n-1})$ with mean zero and positive finite variance, and since repeating positive and negative values of $Z$ makes $\mathcal O(x_i)$ accumulate at $0$ and $1$, Corollary~\ref{prop:arcsine-to-conjugate} implies that
\begin{equation*}
\lim_{n\to\infty}\PP\big(\,\s\ind_{{I_i(\gamma)}}(f^j_\omega(x_i))\leq \alpha\,\big)=\frac{2}{\pi}\arcsin\sqrt{\alpha} \quad \text{ for every $\alpha,\gamma\in(0,1)$ and $i=0,1$}.  
\qedhere
 		\end{equation*}
\end{proof}

By Lemmas~\ref{lem:H0_coupling},~\ref{lem:H1_H2_coupling} and~\ref{lem:weak_arcsine_coupling},  $F_{\smash{T^\Psi,Z}}$ satisfies \rmref{H0}--\rmref{H2} and  pointwise-fiber fluctuation law. Hence Theorem~\ref{thm:B-random-walks} gives that the skew product $F_{\smash{T^\Psi,Z}}$ exhibits historical behavior for $(\PP \times \lbb)$-almost every point. This proves the first part of Proposition~\ref{maincor:coupling}.

To verify the asserted limit set, fix $x\in I_k$. Since $X_n^x\in I_{Z_n}$ for every $n$, any compact set $K\subset(0,1)$ meets only finitely many cells $I_j$. Hence there is a finite set $A_K\subset\mathbb Z$ such that
$\ind_K(X_n^x)\le \ind_{A_K}(Z_n)$. 
By Proposition~\ref{cor:main-cases}, the occupation time of $A_K$ by the random walk $Z_n$ tends to zero almost surely. Taking a countable compact exhaustion of $(0,1)$, we obtain the vanishing interior occupation condition~\eqref{OTT} for every fixed $x$. 
The assumptions of Theorem~\ref{thm:B-random-walks} were verified above. Proposition~\ref{mainpropo-limitset} and Remark~\ref{rem-ppp} therefore give
\[
\mathcal L(\omega,x)=\{\lambda\delta_0+(1-\lambda)\delta_1:\lambda\in[0,1]\}
\]
for every fixed $x\in(0,1)$ and $\mathbb P$-a.e.~$\omega$.

% To verify the asserted limit set, fix $x\in I_k$. Every compact subset of $(0,1)$ meets finitely many cells, so its occupation is bounded by that of a finite set of integers by the macro walk $Z_n$, and tends to zero by Proposition~\ref{cor:main-cases}. Comparison with cell boundaries and Lemma~\ref{cla:conjugated-real-random} give the arcsine limit for every interior threshold. By~\rmref{H0}, each occupation limsup is a tail-measurable constant; the arcsine limit forces it to be one. A countable sequence of thresholds approaching the endpoints gives~\eqref{eq:01}. The corrected Proposition~\ref{prop-delta-L} now gives the entire endpoint segment for every fixed $x$, almost surely.

}

\subsection{Proof of Proposition~\ref{thm:D-Bonifant-Milnor-historical-behavior}}

Let $F_a$ be the skew product given in~\eqref{eq:fracional-linear} where $a:\Omega\to (0,\infty)$  satisfies
~\eqref{eq:fracional-linear-b},~\rmref{E0} and~\rmref{BM4}, i.e., $a(\omega)=a(\omega_0)$ for every $\omega=(\omega_i)_{i\ge0}\in \Omega$, and
\begin{equation*}
		\int \log a(\omega) \, d\PP =0, \quad
 \int (\log a(\omega))^2 \, d\PP<\infty  \quad  \text{and} \quad
a(\omega)\not=1 \quad \text{for every $\omega\in \Esp$}.
	\end{equation*}
We consider the  orientation-preserving homeomorphism
 \begin{equation*}
		\fu{h}{(0,1)}{\R}, \quad h(x)=\log{\pa{\frac{x}{1-x}}}.
\end{equation*}
Then,  since  for any $x\in (0,1)$
		\begin{align*}
			h\circ f_\omega (x)
				 = \log\left(\frac{a(\omega_0)x}{1-x}\right)= 	  h(x) + \log(a(\omega_0))
		\end{align*}
it follows that, $g_\omega(t) \eqdef h\circ f_\omega\circ h^{-1}(t) = t +\log(a(\omega_0))$, where \ $t=h(x)$.
From this and taking into account that
\begin{equation*}
    S_n^t(\omega) \eqdef g_{\omega_{n-1}} \circ \dots \circ g_{\omega_{0}}(t) = t + \log(a(\omega_{0})) + \dots + \log(a(\omega_{n-1})), \quad n \geq 1,
\end{equation*}
we obtain that the sequence of step random variables
\begin{align*}
    Y_{n}^t(\omega) \eqdef   S_{n}^t(\omega)- S_{n-1}^t(\omega)=\log(a(\omega_{n-1})) , \qquad n \geq 1,
\end{align*}
is independent and identically distributed. We also have
that
\begin{align*}
    \mu &=\mathbb{E}[Y^t_{1}]=\int \log(a(\omega))\,d\PP =0 \ \ \text{and} \ \
    \oldsigma^2 = \mathbb{E}[(Y^t_1-\mu)^2]=\int (\log a(\omega))^2 \, d\PP\in(0,\infty).
\end{align*}
Thus for every $x\in (0,1)$, the sequence $\{X^x_n\}_{n\geq 1}$ of random variables $X_n^x(\omega)=f^n_\omega(x)$ is conjugate to a random walk on $\mathbb{R}$ with mean zero and finite variance.

It remains to show that $F_a$ satisfies conditions~\rmref{H1} and~\rmref{H2}. Clearly,~\rmref{H1} holds since the maps $f_\omega$ are interval diffeomorphisms with $f_\omega(0)=0$ and $f_\omega(1)=1$.

We begin by showing that $F$ satisfies the condition~\rmref{H2}.
Define
	$$
	\Omega_- \eqdef \{\omega=(\omega_i)_{i\geq 0} \in\Esp \colon a(\omega_0)<1\} \quad \text{and} \quad
	\Omega_+ \eqdef \{\omega=(\omega_i)_{i\geq 0} \in\Esp \colon a(\omega_0)>1\}.$$
	Note that since $a(\omega_0) \ne 1$ for every $\omega=(\omega_i)_{i \geq 0}\in \Omega$,
	we have $\PP (\Omega_-)+\PP (\Omega_+)=1$.

\begin{lem}\label{cl.positiveprobability}
		$\PP (\Omega_-)>0$ and
		$\PP (\Omega_+)>0$.
	\end{lem}

	\begin{proof}
		By contradiction, suppose that
		$\PP (\Omega_-)=1$. Hence,
		\begin{align*}
			\lambda(\delta_0) &\eqdef\int\log(f'_\omega(0))\,d\PP = \int_{\Omega_-} \log(f'_\omega(0))\,d\PP
			 = \int_{\Omega_-} \log(a(\omega_0))\,d\PP <0.
		\end{align*}
        Since, as mentioned in the introduction, the representation of the fiber maps $f_\omega$ of $F_a$ is equivalent to conditions~\ref{BM1}--\ref{BM3}, we arrive at a contradiction with
	 $\lambda(\delta_0)=0$. A similar contradiction arises assuming that
		$\PP (\Omega_+)=1$.
	\end{proof}

We have $f_\omega(x) < x$ for every $x\in (0,1)$ and $\omega \in \Omega_-$.
	Similarly,  $f_\omega(x) > x$ for every $x\in (0,1)$ and $\omega \in \Omega_+$. Hence, for every $x\in (0,1)$, we have that
 $\Omega_- \subset \{\omega\in\Omega \colon f_\omega(x)<x \}$ and
  $\Omega_+ \subset \{\omega\in\Omega \colon f_\omega(x)>x \}$.
  Now, Lemma~\ref{cl.positiveprobability} implies
	both sets have positive probability, proving the first part of the proposition. From this, it immediately follows that the skew product $F_a$ satisfies the condition~\rmref{H2}.

By Proposition~\ref{maincor:conjugation-random-walk} and Corollary~\ref{cor:lim-set}, $F_a$ exhibits historical behavior for $(\PP \times \mathrm{Leb})$-a.e.~point. Moreover, for every $x \in (0,1)$, $
\mathcal{L}(\omega,x) = \{\lambda \delta_0 + (1-\lambda) \delta_1\colon \lambda \in [0,1]\}$ for $\mathbb{P}$-a.e.~$\omega\in\Omega$, completing the proof of Proposition~\ref{thm:D-Bonifant-Milnor-historical-behavior}.

%% file: Sec9.tex
{
\section{Skew-translations} \label{ss:skew-translation}

Let \((\Omega,\mathcal F,\mathbb P)\) be a standard probability space and
\(\sigma:\Omega\to\Omega\) an ergodic, measure-preserving transformation.
Let \(\phi:\Omega\to\mathbb R\) be a measurable function  and
consider the skew-translation
\begin{equation}\label{skew-translation0}
    T_\phi:\Omega\times\mathbb R\to\Omega\times\mathbb R,\qquad T_\phi(\omega,y)=(\sigma(\omega),y+\phi(\omega)).
\end{equation}
The infinite product measure \(\mu=\PP\times\lbb\) on
\(\Omega\times\mathbb R\) is \(T_\phi\)-invariant.

\subsection{Vanishing interior occupational time for skew-translations}

\begin{thm} \label{thm:A-unified}
Let $T_\phi$ be a skew-translation as in~\eqref{skew-translation0}. Assume one of the following conditions:
\begin{enumerate}[leftmargin=0.75cm]
    \item[\textnormal{(1)}] $\phi: \Omega \to \R$ is not a multiplicative coboundary, i.e.,~\rmref{C2} holds;
    \item[\textnormal{(2)}] $(\Omega,\sigma)$ is a topologically mixing (one-sided or two-sided) subshift of finite type over a finite alphabet or a topologically mixing hyperbolic basic set of a $C^1$ diffeomorphism and $\mathbb{P}$ is a full-support $\sigma$-invariant H\"older Gibbs measure. Moreover, $\phi:\Omega \to \mathbb{R}$ is H\"older and is not an additive coboundary, i.e.,~\rmref{C1} holds;
    \item[\textnormal{(3)}] $(\Omega, \mathbb{P},\sigma)$  is a (one-sided or two-sided) Bernoulli shift and $\phi:\Omega \to \R$ is a non-zero one-step~map.
\end{enumerate}
Then for every fixed initial point $y \in \mathbb{R}$ and every compact set $K \subset \mathbb{R}$,
\[
\lim_{n\to\infty}\frac{1}{n}\sum_{j=0}^{n-1}\ind_{K}(y+S_j(\omega))=0 \quad \text{for $\mathbb{P}$-a.e.\ $\omega\in\Omega$,}
\]
where $S_0=0$ and $S_j(\omega)\eqdef\sum\limits_{i=0}^{j-1}\phi(\sigma^i(\omega))$ for $j\geq1$.
\end{thm}
\begin{proof} We project the dynamics to a compact space. Fix $L>0$ and denote by $\mathbb{T}_L=\mathbb{R}/(L\mathbb{Z})$ the circle of length $L$,  equipped with normalized Lebesgue probability measure $\lbb_L$.
Define
\[
F_L:\Omega\times\mathbb{T}_L\to\Omega\times\mathbb{T}_L,\qquad
F_L(\omega,y)\eqdef(\sigma\omega,\ y+\phi(\omega)\bmod L),
\]
which preserves $\mu_L\eqdef\mathbb{P}\times\lbb_L$.  We use the following proposition.

\begin{prop} \label{prop:ergodicity} Under the assumption of Theorem~\ref{thm:A-unified}, the set of lengths $L$ for which $F_L$ is not ergodic with respect to $\mu_L$ is at most countable. In particular, there exists a sequence $\{L_m\}_{m\geq 1}$ with $L_m\to \infty$ such that $F_{L_m}$ is ergodic with respect to $\mu_{L_m}$.
\end{prop}

Assuming the proposition, we conclude the theorem.   If $F_{L}$ is ergodic with respect to $\mu_{L}$,  Birkhoff Ergodic Theorem implies that the orbits are uniformly distributed for $\mu_{L}$-a.e.~$(\omega, y)$.  This means that  for any continuous function $g:\mathbb{T}_{L} \to \R$,
\begin{equation} \label{eq:Bir}
\lim_{n\to\infty}\frac{1}{n}\sum_{j=0}^{n-1} g(y+S_j(\omega)\bmod L) = \int g\,d\lbb_{L}
\end{equation}
for  $\mu_L$-a.e.~$(\omega,y)$. We strengthen this as follows:

\begin{claim} For every $y\in\mathbb R$, there is a full-measure set $\Omega_{L,y}\subset\Omega$ on which~\eqref{eq:Bir} holds.
\end{claim}

\begin{proof}
Fix a continuous function $g:\mathbb{T}_L\to\mathbb{R}$ and  $\omega\in \Omega$. Define
\[
A^\omega_n(y)\eqdef\frac{1}{n}\sum_{j=0}^{n-1} g\big(y+S_j(\omega) \bmod L \big)\qquad y\in\mathbb{T}_L.
\]
By uniform continuity of $g$, for each $\varepsilon>0$ choose $\delta>0$ such that $d_{\mathbb{T}_L}(u,v)<\delta$ implies that $|g(u)-g(v)|<\varepsilon/2$. If $y,y'\in\mathbb{T}_L$ satisfy $d_{\mathbb{T}_L}(y,y')<\delta$, then for every $j\ge0$,
$$
d_{\mathbb{T}_L}(y+S_j(\omega),\,y'+S_j(\omega))=d_{\mathbb{T}_L}(y,y')<\delta,$$
hence $|g(y+S_j(\omega)\bmod L)-g(y'+S_j(\omega)\bmod L)|<\varepsilon/2$. Hence, for every $n\ge1$,
\[
|A^\omega_n(y)-A^\omega_n(y')|
\le \frac{1}{n}\sum_{j=0}^{n-1}|g(y+S_j(\omega)\bmod L)-g(y'+S_j(\omega)\bmod L)|<\frac{\varepsilon}{2}.
\]
Thus the family $(A^\omega_n)_{n\ge1}$ is equicontinuous.

By Fubini and~\eqref{eq:Bir}, for $\mathbb{P}$-a.e.~$\omega\in \Omega$, there exists $\mathcal{T}_\omega \subset \mathbb{T}_L$ with  $\lbb_L(\mathcal{T}_\omega)=1$ such that $A^\omega_n(y) \to \ell\eqdef\int g \, d\lbb_L$ for all $y\in \mathcal{T}_\omega$ as $n\to \infty$.
Fix $y_0\in\mathbb{T}_L$. Since $\mathcal{T}_\omega$ has full measure, it is dense, so choose $y\in \mathcal{T}_\omega$ with $d_{\mathbb{T}_L}(y_0,y)<\delta$. By hypothesis $A^\omega_n(y)\to \ell$, so there exists $N\geq 1$ with $|A^\omega_n(y)-\ell|<\varepsilon/2$ for all $n\ge N$. For such $n$ we obtain
\[
|A^\omega_n(y_0)-\ell|\le |A^\omega_n(y_0)-A^\omega_n(y)|+|A^\omega_n(y)-\ell|<\frac{\varepsilon}{2}+\frac{\varepsilon}{2}=\varepsilon,
\]
where the first term is $<\varepsilon/2$ by equicontinuity and the second is $<\varepsilon/2$ by choice of $N$. As $\varepsilon>0$ and $y_0$ were arbitrary, $A^\omega_n(y_0)\to\ell$ for every $y_0\in\mathbb{T}_L$. Intersect these full-measure base sets over a countable uniformly dense family in $C(\mathbb T_L)$. Uniform approximation then gives~\eqref{eq:Bir} for every continuous $g$ on the same full-measure set.
\end{proof}

Fix $y\in\mathbb{R}$ and a compact set $K\subset\mathbb{R}$.  Let $\{L_m\}_{m\geq 1}$ be the sequence of lengths given by
Proposition~\ref{prop:ergodicity} that can be assumed to be
$\lbb(K)<L_m$ for all $m\geq 1$.
Let $K_m\subset\mathbb{T}_{L_m}$ be the projection of $K$.
For each $\varepsilon>0$, choose a continuous $g_{\varepsilon,m}:\mathbb{T}_{L_m}\to[0,1]$ such that
\[
\ind_{K_m}\le g_{\varepsilon,m}
\quad\text{and}\quad
\int g_{\varepsilon,m}\,d\lbb_{L_m}\le \frac{\lbb(K)}{L_m}+\varepsilon.
\]
By the claim, for each $m$ there exists a full $\mathbb{P}$-measure set $\Omega_{L_m,y}$ such that for every $\omega\in\Omega_{L_m,y}$,
\[
\limsup_{n\to\infty}\frac{1}{n}\sum_{j=0}^{n-1}\ind_K(y+S_j(\omega))
\le \lim_{n\to\infty}\frac{1}{n}\sum_{j=0}^{n-1} g_{\varepsilon,m}(y+S_j(\omega))
=\int g_{\varepsilon,m}\,d\lbb_{L_m}
\le \frac{\lbb(K)}{L_m}+\varepsilon.
\]
Let $\Omega_y\eqdef\bigcap_{m\ge1}\Omega_{L_m,y}$, which still has full $\mathbb{P}$-measure.
Since $L_m\to\infty$, it follows
\[
\limsup_{n\to\infty}\frac{1}{n}\sum_{j=0}^{n-1}\ind_K(y+S_j(\omega))=0 \qquad \text{for every $\omega\in\Omega_y$}. \qedhere
\]
\end{proof}

\subsubsection{Proof of Proposition~\ref{prop:ergodicity}}

We show that $F_L$ is ergodic for a suitable sequence $L\to\infty$. Let $f\in L^2(\mu_{L})$ be a $F_{L}$-invariant function, i.e., $f \circ F_{L} = f$. Expand in Fourier series in the fiber variable:
$$f(\omega,y)=\sum_{k\in\mathbb Z} c_k(\omega)e^{2\pi i k y/L}.$$
By the $F_{L}$-invariance of $f$, for each $k \in \Z$, it holds that
\begin{equation}\label{eq:Fourier-cocycle}
c_k\circ\sigma \;=\; e^{-i\lambda\phi}\,c_k\qquad\text{$\mathbb P$-a.e.,\quad where }\ \lambda=\frac{2\pi k}{L}.
\end{equation}
For $k=0$, ergodicity of $\sigma$ makes $c_0$ constant a.e.; for $k\neq0$, we show $c_k\equiv0$ for suitable $L$.

Assume $c_k\not\equiv 0$ for some $k\neq 0$, and set $B=\{\omega:\,c_k(\omega)\neq 0\}$. From \eqref{eq:Fourier-cocycle},  $\sigma^{-1}(B)=B$ modulo $\mathbb P$ and, hence $\mathbb P(B)\in\{0,1\}$. If $\mathbb P(B)=0$ we are done; otherwise $\mathbb P(B)=1$. Taking moduli in~\eqref{eq:Fourier-cocycle} gives $|c_k|\circ\sigma=|c_k|$ almost everywhere, so again by ergodicity there is $r>0$ with $|c_k(\omega)|=r$  for $\mathbb{P}$-a.e.~$\omega\in B$. Define
\[
\psi:\Omega\to\mathbb S^1,\ \ \  \psi(\omega)=\frac{ \overline{c_k(\omega)}}{r}  \ \ \ \text{for $\mathbb{P}$-a.e.~$\omega\in B$ and arbitrarily on the null complementary}.
\]
 Then conjugating \eqref{eq:Fourier-cocycle} and dividing by $r$,
we get
\begin{equation} \label{eq:multiplicative_cohom}
      e^{i\lambda \phi} = \frac{\psi \circ \sigma}{\psi} \quad \text{$\mathbb P$-a.e.}  \qquad  \text{with $\lambda=\frac{2\pi k}{L}$.}
\end{equation}

\smallskip
\noindent\textbf{Case (1):} \emph{$\phi: \Omega \to \R$ is not a multiplicative coboundary, i.e.,~\rmref{C2} holds.} Equation~\eqref{eq:multiplicative_cohom}  yields for $t=\lambda$, a nonzero solution $\psi$ of the cohomological equation, which contradicts the hypothesis that $\phi$ satisfies~\rmref{C2}. Thus, $c_k \equiv 0$ for all $k \neq 0$. Thus $F_L$ is ergodic with respect to $\mu_L$ for every $L>0$.

\begin{rem} \label{rem:equivalence-ergodicity-C2}
    In this first case of Proposition~\ref{thm:A-unified},
the conclusion is stronger: condition \rmref{C2} characterizes the
ergodicity of all compact circle quotients $F_L$. Indeed, in view of the above conclusion, we only need to show that if $F_L$ is ergodic with respect to $\mu_L$ for all $L>0$, then \rmref{C2} holds. To see this, assume that \rmref{C2}  fails. Then there exist
$\lambda\neq0$ and a measurable $\psi:\Omega\to\mathbb S^1$ such that
$e^{i\lambda\phi}={\psi\circ\sigma}/{\psi}$ $\PP\text{-a.e.}$. Taking $L=2\pi/|\lambda|$, the nonconstant function $G(\omega,y)=\psi(\omega)e^{-i\lambda y}$
is well defined on $\Omega\times\mathbb T_L$ and satisfies $G\circ F_L=G$ $\mu_L$-a.e. 
Thus $F_L$ is not ergodic.
\end{rem}

\smallskip
\noindent\textbf{Case (2):} \emph{$(\Omega, \mathbb{P},\sigma)$ is a subshift of finite type or a hyperbolic basic set preserving a H\"older Gibbs measure. Moreover, $\phi:\Omega \to \mathbb{R}$ is H\"older and is not an additive coboundary.}  Since $\phi$ is H\"older, so is $e^{i\lambda\phi}$. According to~\cite[Thms.~1 and~2]{PP97}, this regularity, the base assumptions, and equation~\eqref{eq:multiplicative_cohom} imply that there exists a H\"older continuous function $u:\Omega\to\mathbb S^1$ such that
$$\psi=u \quad \text{and} \quad e^{i\lambda\phi}=\frac{u\circ\sigma}{u} \quad \text{$\mathbb P$-a.e.}$$

Full support extends the continuous cohomological identity to every base point. By additive Liv\v{s}ic theory~\cite{Liv72,PP90} and condition~\rmref{C1}, there exists a periodic point $q$ of period $r\ge1$ such that
$a\eqdef\sum_{j=0}^{r-1}\phi(\sigma^j q)\neq0$. Otherwise all periodic sums would vanish and $\phi$ would be an additive coboundary. Evaluating the identity
$e^{i\lambda\phi}=(u\circ\sigma)/u$
along the orbit $q,\sigma q,\ldots,\sigma^{r-1}q$ and multiplying the resulting equalities, the right-hand side telescopes and, since $\sigma^r q=q$, gives
$e^{i\lambda a}=1$. Hence $\lambda a\in2\pi\mathbb Z$. Recalling that $\lambda=2\pi k/L$ with $k\in\mathbb Z\setminus\{0\}$, we obtain
$\frac{ka}{L}\in\mathbb Z\setminus\{0\}$.
Therefore every exceptional length satisfies
\[
L\in |a|\,\mathbb Q_{>0},
\]
which is a countable set. Thus $F_L$ is ergodic for all $L>0$ outside this countable set.
% Full support extends this continuous identity to every base point. By additive Liv\v{s}ic theory~\cite{Liv72,PP90} and~\rmref{C1}, there is a periodic point $q$ of period $r$ with $a=S_r(q)\ne0$; otherwise $\phi$ would be a H\"older additive coboundary. Multiplying the circle-valued identity around this periodic orbit gives $e^{i\lambda a}=1$. Thus, for $\lambda=2\pi k/L$ with $k\ne0$, we have $ka/L\in\mathbb Z\setminus\{0\}$. Every exceptional length belongs to the countable set $|a|\mathbb Q_{>0}$. Hence $F_L$ is ergodic outside this set.

\smallskip
\noindent\textbf{Case (3):} \emph{$(\Omega, \mathbb{P},\sigma)$  is a  Bernoulli shift and $\phi:\Omega \to \R$ is a one-step  function.}
The next lemma reduces the argument to the one-sided Bernoulli shift.
\begin{lem}
The one-sided skew product is ergodic if and only if its two-sided natural extension is.
\end{lem}

\begin{proof}
Denote by $F_L^+$ and $F_L$ the skew-products over the one-sided and two-sided Bernoulli shifts $(\Omega_+,\mathbb{P}_+,\sigma)$ and $(\Omega,\mathbb{P},\sigma)$ respectively and consider $\Pi(\dots,\omega_{-1}\omega_0\omega_1\dots,y)=(\omega_0\omega_1\dots,y)$. Since 
$\Pi\circ F^{}_L=F_L^+\circ\Pi$,  $F_L^+$ is a factor of $F_L$, so ergodicity of $F_L$ immediately implies
ergodicity of $F_L^+$. Conversely, assume that $F_L^+$ is ergodic. Let $\mathcal A=\Pi^{-1}(\mathscr{B}(\Omega^{+}\times \mathbb{T}_L))$ be the factor
$\sigma$-algebra on the natural extension.  Let $f\in L^2(\mu_L)$ satisfy
$f\circ F_L=f$, and set $f_0= \mathbb E[f\mid\mathcal A]$. For every $n\ge0$,
$\mathbb E[f\mid\mathcal F_L^n\mathcal A]=f_0\circ F_L^{-n}$.
Indeed, $f_0\circ F_L^{-n}$ is $\mathcal F_L^n\mathcal A$-measurable, and if $H$ is bounded
and $F_L^n\mathcal A$-measurable, then $H\circ F_L^n$ is $\mathcal A$-measurable; using
the definition of conditional expectation, invariance of $\mu_L$, and
$f\circ F_L^{-n}=f$, we obtain
\[
\int H(f_0\circ F_L^{-n})\,d\mu_L
=
\int (H\circ F_L^n)f_0\,d\mu_L
=
\int (H\circ F_L^n)f\,d\mu_L
=
\int Hf\,d\mu_L.
\]
Consequently, we get that $f_0\circ F^{-n}_L$ is the indicated conditional expectation. Moreover, as the algebras $F_L^n\mathcal A$ increase to the full algebra, by the martingale convergence theorem,
\[
f_0\circ F_L^{-n}
=
\mathbb{E}[f\mid F_L^n\mathcal A]
\longrightarrow f
\qquad\text{in }L^2(\mu_L).
\]
Since $F_L$ preserves $\mu_L$, $\|f_0\circ F_L^{-n}\|_2=\|f_0\|_2$ for every $n$, and therefore $\|f\|_2=\|f_0\|_2$. But conditional expectation is the orthogonal projection onto
$L^2(\mathcal A)$, so
$\|f\|_2^2
=
\|f_0\|_2^2+\|f-f_0\|_2^2$. 
Hence $f=f_0$ almost everywhere. Thus $f$ is $\mathcal A$-measurable and therefore
descends to a function $\bar f\in L^2(\mu_L^+)$ satisfying $f=\bar f\circ\Pi$. Using $\Pi\circ F_L=F_L^+\circ\Pi$ and the invariance of $f$, we obtain
$(\bar f\circ F_L^+)\circ\Pi
=
\bar f\circ\Pi$, and hence $\bar f\circ F_L^+=\bar f$ almost everywhere. Since $F_L^+$ is ergodic,
$\bar f$ is constant, and therefore so is $f$. Thus $F_L$ is ergodic.
\end{proof}

% \begin{proof}
% The two-sided circle skew product is the probability-preserving natural extension of the one-sided one: the past base coordinates and present fiber coordinate determine all past fiber coordinates by subtraction of the cocycle sums. Ergodicity passes to a factor. Conversely, write $\mathcal A$ for the factor $\oldsigma$-algebra on the extension and let $f\in L^2(\mu_L)$ be invariant. Put $f_0=\mathbb E[f\mid\mathcal A]$. For $n\geq0$,
% \[
% \mathbb E[f\mid F_L^n\mathcal A]=f_0\circ F_L^{-n}.
% \]
% The algebras $F_L^n\mathcal A$ increase to the full algebra. Martingale convergence and invariance of the probability measure give $\|f\|_2=\|f_0\|_2$, hence $f=f_0$. Thus $f$ descends to an invariant function on the one-sided factor, and is constant if that factor is ergodic.
% \end{proof}

The next lemma links multiplicative coboundaries to the step characteristic function.

\begin{lem} \label{lem:coboundary_criterion}
Let $(\Omega, \mathbb{P}, \sigma)$ be a one-sided Bernoulli shift and consider a one-step function $\phi:\Omega\to \mathbb{R}$. If there exist $\lambda\in\mathbb{R}$ and a measurable function $\psi \in L^2(\mathbb{P})$, not almost everywhere zero, satisfying the multiplicative cohomological equation
\[
\psi(\sigma(\omega)) = e^{i\lambda\phi(\omega)}\psi(\omega) \quad \text{for $\mathbb{P}$-a.e. } \omega,
\]
then
$$
\Phi_\phi(\lambda)\eqdef\mathbb{E}[e^{i\lambda \phi}] = 1.
$$
\end{lem}

\begin{proof}
Taking moduli in the cohomological equation and using base ergodicity, normalize $\psi$ so that $|\psi|=1$ almost surely. Let $\mathcal F_n$ be generated by coordinates $0,\ldots,n-1$. Iterating the equation and using independence of the remaining coordinates gives
\[
\psi=e^{-i\lambda S_n}\psi\circ\sigma^n,
\qquad \mathbb E[\psi\mid\mathcal F_n]=e^{-i\lambda S_n}\mathbb E[\psi].
\]
If $\mathbb E[\psi]=0$, martingale convergence would imply $\psi=0$, a contradiction. Taking expectations at $n=1$ now gives $\mathbb E[e^{-i\lambda\phi}]=1$, and complex conjugation proves the assertion.
\end{proof}

We complete Proposition~\ref{prop:ergodicity}. By~\eqref{eq:multiplicative_cohom}, there exist $\lambda\not =0$ and $\psi:\Omega\to \mathbb{S}^1$ in $L^2(\PP)$ such that $\psi e^{i\lambda\phi}={\psi\circ \sigma}$  holds $\PP$-a.e. Lemma~\ref{lem:coboundary_criterion} gives
$\Phi_\phi(-\lambda)=1$, where $\Phi_\phi(t)\eqdef\mathbb E[e^{it\phi}]$.

Since a unit-modulus random variable with expectation one equals one almost surely, $e^{-i\lambda\phi}=1$ almost surely. Fix a nonzero point $a$ in the support of the step law; it exists because $\phi$ is not zero almost surely. Continuity implies $e^{-i\lambda a}=1$. With $\lambda=2\pi k/L$, $k\ne0$, this gives $ka/L\in\mathbb Z\setminus\{0\}$, so every exceptional length belongs to $|a|\mathbb Q_{>0}$. Thus $F_L$ is ergodic outside a countable set of lengths, including when the step law is a nonzero Dirac mass.

\subsubsection{Consequences}

\begin{prop}\label{cor:main-cases}
Let $\{S_n\}_{n\ge 0}$ be an $\R$-valued stochastic process with fixed
initial value $S_0=t$. Assume that the increments $Y_n \eqdef S_n - S_{n-1}$ for $n \ge 1$ are i.i.d.~with a common law $\mu \not=\delta_0$  (i.e., it is not frozen). Then for every compact set $K\subset\R$,
\[
\lim_{n\to\infty}\frac{1}{n}\sum_{j=0}^{n-1}\ind_K(S_j)
=0
\qquad\text{almost surely.}
\]
\end{prop}

\begin{proof}
Since $S_n$ is measurable with respect to $\oldsigma(Y_1,\dots,Y_n)$ and $Y_{n+1}$ is independent of this $\oldsigma$-algebra, $Y_{n+1}$ is independent of $S_n$ for each $n$.
Hence the conditional law of the next increment given the present state is $\mu$ almost surely and does not depend on the current state. We construct the canonical Bernoulli model with the same marginal law.
Let $\mathcal A$ be the topological support of $\mu$  endowed with its Borel $\oldsigma$-algebra,
and consider the one-sided Bernoulli shift
$(\Omega,\mathbb P,\sigma)$ with $\Omega=\mathcal A^{\mathbb N}$ and
$\mathbb P=\mu^{\mathbb N}$.
Let $\phi:\Omega\to\R$ be the coordinate projection $\phi(\omega)=\omega_0$
(so $\phi_*\mathbb P=\mu$).
 Consider the skew-translation
$T_\phi(\omega, x) = (\sigma(\omega), x + \phi(\omega)) $
and the corresponding process given by the fiber iteration
\[
S'_n(\omega)=t+\sum_{j=0}^{n-1}\phi(\sigma^j(\omega)),\qquad n\ge0.
\]
Write $Y'_n\eqdef S'_n-S'_{n-1}=\phi(\sigma^{\,n-1}(\omega))$ and note that, by construction,  $\{Y'_n\}_{n\ge1}$
is an i.i.d.~sequence with common law~$\mu$.

The finite-dimensional distributions of $\{S_n\}_{n\geq0}$ are determined by those of its increments $(Y_1,\dots,Y_N)$ for each $N$. Since
both $(Y_1,\dots,Y_N)$ and $(Y'_1,\dots,Y'_N)$ have the product law $\mu^{\otimes N}$,
the finite-dimensional distributions of $(S_0,\dots,S_N)$ and $(S'_0,\dots,S'_N)$ coincide
for every $N$. These consistent laws determine a unique probability measure on $\R^{\mathbb N}$; hence the path-space law of $\{S_n\}_{n\geq 0}$ equals that of $\{S'_n\}_{n\geq0}$.

Since $\mu\not=\delta_0$, $\phi$ is a non-zero one-step map and thus Theorem~\ref{thm:A-unified} (alternative (3)) applies to the skew-translation $T_\phi$. Hence, for every compact $K\subset\R$,
\[
\lim_{n\to\infty}\frac{1}{n}\sum_{j=0}^{n-1}\ind_K(S'_j)=0
\qquad\text{for $\mathbb P$-a.e.~$\omega\in\Omega$.}
\]
Since the processes have the same path-space law, this almost-sure vanishing occupation time also holds for $\{S_n\}_{n\ge0}$.
\end{proof}

\begin{cor}
     \label{cor:main-cases-cor}
     Let $F$ be a one-step skew product as in~\eqref{one-skew product-principal}. Assume that $\{X^x_n\}_{n \geq 0}$ is conjugate to a ${G}$-valued non-frozen random walk where $X^x_n(\omega) = f^n_\omega(x)$ and $G$ is either $\Z$ or $\R$. Assume that the conjugacy $h$ maps $K\cap\mathcal O(x)$ to a bounded set for every compact $K\subset(0,1)$.  Then for every compact set $K \subset (0,1)$,
 \[
 \lim_{n\to\infty}\frac{1}{n}\sum_{j=0}^{n-1}\ind_{K}(f^j_\omega(x))=0 \quad \text{for $\mathbb{P}$-a.e.~$\omega\in\Omega$.}
 \]
\end{cor}
\begin{proof}
 Let $\mathcal{O}(x)$ be the set $\{X^x_n(\omega): \omega \in \Omega, n\geq 0 \}$.
 By hypothesis, there is a strictly monotonic injection  $h:\mathcal O(x)\to G$  such that the step increments
$Y^t_n\eqdef S_n^t-S_{n-1}^t \in G$, $n\ge1$,
are i.i.d.~ non-degenerate random variables (i.e., with law $\mu\not= \delta_0)$  where $S_n^t(\omega) \eqdef (h\circ  f^n_\omega\circ h^{-1})(t)$ and $t=h(x)$. Thus, the random walk $\{S^t_n\}_{n\ge0}$ satisfies the assumption of Proposition~\ref{cor:main-cases}. Now, fix a compact set $K$ of $(0,1)$ and  choose $M$ such that $h(K\cap\mathcal O(x))\subset[-M,M]$ and let $K'=[-M,M]$.  By~Proposition~\ref{cor:main-cases},
\[
\lim_{n\to\infty}\frac{1}{n}\sum_{j=0}^{n-1}\ind_{K'}(S^t_j(\omega))=0 \quad \text{for $\mathbb{P}$-a.e.~$\omega\in\Omega$}.
\]
From here, as in Proposition~\ref{prop:arcsine-to-conjugateA}, since $\ind_{K'}(S^t_j(\omega)) \geq \ind_{K}(f^j_\omega(x))$, the desired vanishing occupation time for the sequence $f^j_\omega(x)$ follows.
 \end{proof}

\subsubsection{Proof of Corollary~\ref{cor:lim-set}}

By Proposition~\ref{maincor:conjugation-random-walk}, $F$ satisfies~\rmref{H0} and the arcsine law~\eqref{eq:arcsine-law}. By~\rmref{H2},
$\inf\mathcal O(z)<z<\sup\mathcal O(z)$ for every $z\in(0,1)$.
Consequently, for any $0<\gamma_0<\gamma_1<1$, the arcsine law with $x_i=\gamma_i$ gives the pointwise-fiber fluctuation law with these thresholds. Applying Theorem~\ref{thm:B-random-walks} along countable sequences $\gamma_0\to0$ and $\gamma_1\to1$, we obtain~\eqref{eq:01} for every fixed $x\in(0,1)$ and $\PP$-a.e.~$\omega$. Since every iterate belongs to $[\inf\mathcal O(x),\sup\mathcal O(x)]$, this implies $\inf\mathcal O(x)=0$ and $\sup\mathcal O(x)=1$.

Fix $x\in(0,1)$ and a nonempty compact set $K\subset(0,1)$. There are $u,v\in\mathcal O(x)$ with $u<\inf K\leq\sup K<v$. Strict monotonicity of the conjugacy $h$ therefore places $h(K\cap\mathcal O(x))$ in the bounded interval with endpoints $h(u)$ and $h(v)$. Corollary~\ref{cor:main-cases-cor} gives~\eqref{OT} for every fixed $x$ and $0<\epsilon<1/2$, almost surely. Proposition~\ref{mainpropo-limitset} and Remark~\ref{rem-ppp} now yield the asserted full segment between $\delta_0$ and $\delta_1$.

% 
% For each fixed $x$ and threshold, the occupation limsup is tail-measurable and therefore constant by Proposition~\ref{prop:conjugate-H0}. The arcsine law forces this constant to be one. Taking countably many thresholds tending to $0$ and $1$ gives~\eqref{eq:01}. A countable compact exhaustion gives~\eqref{OTT}. The corrected Proposition~\ref{prop-delta-L}, together with Proposition~\ref{prop:limit_set_from_KD}, gives the asserted entire segment for every fixed $x$, almost surely.

% 

}

\subsection{Fluctuation Law}

\begin{thm}
\label{prop:skew-translation-ergi-arcsine}
Let $T_\phi$ be a skew product as in~\eqref{skew-translation0} where $(\Omega, \sigma)$ is a topologically mixing (one-sided or two-sided) subshift of finite type over a finite alphabet or a topologically mixing hyperbolic basic set of a $C^1$ diffeomorphism, $\mathbb{P}$ is a  full-support $\sigma$-invariant Hölder Gibbs measure and  $\phi \colon \Esp \to \mathbb{R}$ is a Hölder continuous function with $\mathbb{E}[\phi] = 0$ and satisfying~\ref{C1}. Then, for every $y \in \mathbb{R}$,
\[
\liminf_{n \to \infty} \, \mathbb{P}\bigg( \frac{1}{n} \sum_{j=0}^{n-1} \ind_{\bar{J}_i(y)}\big(y+S_j(\omega)\big) \leq \alpha \bigg) < 1 \quad \text{for every } \alpha \in (0,1) \text{ and } i = 0,1,
\]
where $\bar{J}_0(y) = (-\infty, y)$, $\bar{J}_1(y) = (y, \infty)$
and, $S_0=0$ and $S_j=\phi+\phi\circ \sigma+\dots + \phi\circ \sigma^{j-1}$, $j\geq 1$.
\end{thm}

\begin{proof}
 {It suffices to prove the theorem in the invertible case, since the result for a one-sided subshift $(\Omega_+, \sigma_+, \mathbb{P}_+)$ follows from its natural extension. Let $(\Omega, \sigma, \mathbb{P})$ be this extension, where the measure satisfies $\mathbb{P}_+ = \pi_* \mathbb{P}$ for the canonical projection $\pi: \Omega \to \Omega_+$. If $\mathbb{P}_+$ is a Hölder Gibbs measure, then so is its lift $\mathbb{P}$. For $\phi\circ\pi$, the Birkhoff sums are preserved because $S_j(\omega) = S_j(\pi(\omega))$ for any $\omega \in \Omega$. Thus the probability in the theorem is identical in both systems for every $n$. Hence we assume the base is invertible.}

To prove the statement, it suffices to show it for~$i=1$; the argument for $i=0$ is analogous. The asymptotic variance is positive: zero variance would make $\phi$ an $L^2$ additive coboundary~\cite{PP90}, contrary to~\ref{C1}. Brownian oscillations on unit intervals are $o(t^{1/2-\beta})$ almost surely (decreasing $\beta>0$ if necessary), so the discrete strong approximation also applies to the interpolation below.
Interpolate $\{S_n\}_{n\geq0}$ by
$$
S_t=(\lfloor t \rfloor + 1-t) S_{\lfloor t \rfloor} + (t-\lfloor t \rfloor) S_{\lfloor t \rfloor +1},  \quad t\geq 0.
$$
Since $S_t=S_{\lfloor t \rfloor}+(t-\lfloor t \rfloor) \cdot  \phi \circ \sigma^{\lfloor t \rfloor}$ and $\phi$ is bounded and $\sigma$ invertible, by Rudolph's strong invariance principle~\cite[Prop.~2]{Rud88} (see also~\cite{DP-LIL}), we have $\oldsigma>0$, $0<\beta<1/2$ and {a probability space $(\overline{\Omega},\overline{\mathscr{F}}, \overline{\PP})$ joining $(\Omega,\mathscr{F},\PP)$ with a standard Brownian motion $\{B_t\colon t\geq 0\}$}   such that
$$
0\leq \frac{|S_t - \oldsigma B_t|}{t^{1/2-\beta}} \leq \frac{|S_{\lfloor t \rfloor}-\oldsigma B_t|}{t^{1/2-\beta}} + \frac{|\phi\circ \sigma^{\lfloor t \rfloor}|}{t^{1/2-\beta}} \xrightarrow[\,t\to \infty\,]{} 0 \quad \text{$\overline{\PP}$-almost surely}.
$$
Hence
\begin{equation} \label{eq:ainda-nose}
\lim_{t\to \infty} \frac{|S_t - \oldsigma B_t|}{t^{1/2-\beta}} = 0  \quad \text{$\overline{\PP}$-almost surely}.
\end{equation}
Let $W_n$ and $S^*_n$ be, respectively, the Brownian motion and the random function  defined by rescaling $\{B_t \colon t \geq  0\}$ and $\{S_t\colon t\geq 0\}$
according to
$$
W_n(t) \eqdef \frac{B_{nt}}{\sqrt n} \qquad \text{and} \qquad  S^*_n (t) \eqdef \frac{S_{nt}}{\sqrt n}  \qquad \text{for $t\in [0,1]$ \ \  and \  \ $n\geq 1$}.
$$
By continuity and~\eqref{eq:ainda-nose},
$$
\sup_{0\leq t\leq 1} |S_t - \oldsigma B_t| < \infty \quad \text{and} \quad
\sup_{t\geq 1} \frac{|S_t - \oldsigma B_t|}{t^{1/2-\beta}} < \infty  \quad \text{$\overline{\PP}$-almost surely}.
$$
Consequently,
\begin{align*}
    \|S^*_n-\oldsigma W_n\| &\eqdef \sup_{0\leq t \leq 1} \big|S^*_n(t)-\oldsigma W_n(t)\big|  =  \frac{1}{\sqrt{n}} \sup_{0\leq t\leq n}  \big|S_t - \oldsigma B_t\big|
    \\
    & \leq  \frac{1}{\sqrt{n}} \sup_{0\leq t\leq 1}  \big|S_t - \oldsigma B_t\big| + \frac{1}{\sqrt{n}}\sup_{1\leq t\leq n}  t^{1/2-\beta} \frac{|S_t - \oldsigma B_t|}{t^{1/2-\beta}} \\
    & \leq  \frac{1}{\sqrt{n}} \sup_{0\leq t\leq 1}  \big|S_t - \oldsigma B_t\big| + \frac{n^{1/2-\beta} }{\sqrt{n}} \sup_{1\leq t\leq n}  \frac{|S_t - \oldsigma B_t|}{t^{1/2-\beta}}
    \\ &= \frac{1}{\sqrt{n}} \sup_{0\leq t\leq 1}  \big|S_t - \oldsigma B_t\big| + \frac{1}{n^{\beta}} \, \sup_{ 1\leq t\leq n}  \frac{|S_t - \oldsigma B_t|}{t^{1/2-\beta}}  \xrightarrow[\,n\to \infty\,]{} 0  \quad \text{$\overline{\PP}$-almost surely}.
\end{align*}
Thus,
$
\lim_{n\to  \infty}  \|S^*_n -\oldsigma W_n\|=0
$
{$\overline{\PP}$-almost surely}
where $\|\cdot\|$ is the sup-norm on the space $C^0([0,1])$ of real-valued continuous function on~$[0,1]$. This implies that  \[
\lim_{n\to\infty}
\overline{\PP}
\bigl(\|S_n^*-\oldsigma W_n\|>\epsilon\bigr)=0
\qquad
\text{for every $\epsilon>0$}.
\]
Moreover, by the scaling property of Brownian motion, $W_n$ has
the same distribution as a standard Brownian motion $B$ for every
$n\geq1$. Consequently, by the standard converging-together lemma
(cf.~\cite[proof of the Donsker invariance principle]
{PeYu-book-Brownian} or~\cite[Lem.~6.2]{liu2021brownian}), $S^*_n$ converges in distribution to $\oldsigma B$ where $B$ is a standard Brownian motion.

Consider now the function $\fu{\chi}{C^0([0,1])}{[0,1]}$ defined by
$\chi(f) \eqdef \int_0^1 \ind_{(0,\infty)}(f(t)) \, dt$.
The function $\chi$ is continuous in every $f \in  C^0([0,1])$ with the property
that  $\lbb( \{ t \in [0, 1] \colon f(t)=0 \} )= 0$. Since
$\lbb\big(\{t\in[0,1]\colon \oldsigma B(t)=0\}\big)=0$ for $\overline{\PP}$-a.s.,
the continuous mapping theorem implies that $\chi(S^*_n)$ converges in distribution to $\chi(\oldsigma B)$.

On the other hand,
\begin{align*}
\chi(\oldsigma B) &\eqdef \int_{0}^1 \ind_{(0,\infty)}(\oldsigma B_{t}) \, dt = \int_{0}^1 \ind_{(0,\infty)}(B_{t}) \, dt \eqdef \chi(B)  \\
\intertext{and}
\chi(S^*_n) &=\int_{0}^1 \ind_{(0,\infty)}({S_{nt}})\, dt =\sum_{j=0}^{n-1}  \int_{\frac{j}{n}}^{\frac{j+1}{n}} \ind_{(0,\infty)}({S_{nt}}) \, dt = \frac{1}{n} \sum_{j=0}^{n-1} \int_{j}^{j+1} \ind_{(0,\infty)}({S_\theta})\, d\theta.
\end{align*}

For $\theta\in[j,j+1]$, $|S_\theta-S_j|\leq M$, where $M=\|\phi\|_\infty$. Hence
\[
\left|\chi(S_n^*)-\frac1n\sum_{j=0}^{n-1}\ind_{(0,\infty)}(S_j)\right|
\leq\frac1n\sum_{j=0}^{n-1}\ind_{[-M,M]}(S_j)\longrightarrow0
\quad\text{almost surely}
\]
by Theorem~\ref{thm:A-unified}. Slutsky's theorem and L\'evy's arcsine law~\cite{Lev39,Lev-Brow2} (see also~\cite[Sec.~5.4]{PeYu-book-Brownian}) therefore give, for $\alpha\in(0,1)$,
\begin{equation}\label{eq:brownian-arcsine}
\lim_{n\to\infty}\overline{\mathbb P}\left(\frac1n\sum_{j=0}^{n-1}\ind_{(0,\infty)}(S_j)\leq\alpha\right)
=\frac2\pi\arcsin\sqrt\alpha<1.
\end{equation}

Finally, since $\ind_{(0,+\infty)}(S_j(\omega)) = \ind_{(y,\infty)}(y + S_j(\omega))$ for every $y \in \mathbb{R}$, \eqref{eq:brownian-arcsine} holds for the original process because the joining preserves the law of the sums. Therefore, for every $y \in \mathbb{R}$,
\[
\liminf_{n \to \infty} \mathbb{P} \big( \frac{1}{n} \sum_{j=0}^{n-1} \ind_{(y,\infty)}(y + S_j(\omega)) \leq \alpha \big) < 1 \quad \text{for every } \alpha \in (0,1). \qedhere
\]
\end{proof}

{

These facts give the following almost-sure half-line statement.

\begin{cor}\label{cor:lim1} Under the assumptions of Theorem~\ref{prop:skew-translation-ergi-arcsine}, for any fixed $\kappa, y\in \mathbb{R}$ it holds that
\[
\limsup_{n\to\infty} \frac{1}{n}\sum_{j=0}^{n-1} \ind_{\bar{J}_i(\kappa)}\big(y+S_j(\omega)\big)=1 \quad \text{for $\mathbb{P}$-a.e.~$\omega\in \Omega$, and $i=0,1$.}
\]
\end{cor}
\begin{proof} We prove $i=1$; the case $i=0$ is analogous.
Set
$$
A_n(\omega)\eqdef
\frac{1}{n}\sum_{j=0}^{n-1}\ind_{(0,\infty)}\big(S_j(\omega)\big)
$$
and define $C(\omega)\eqdef\limsup_{n\to\infty} A_n(\omega)$.
Since $S_j(\sigma(\omega))=S_{j+1}(\omega)-\phi(\omega)$, it follows that
\[
A_n(\sigma(\omega))=\frac{1}{n}\sum_{j=0}^{n-1}\ind_{(0,\infty)}\big(S_{j+1}(\omega)-\phi(\omega)\big)
=\frac{1}{n}\sum_{j=1}^{n}\ind_{(0,\infty)}\big(S_{j}(\omega)-\phi(\omega)\big).
\]
Because $\phi$ is H\"older on the compact base, it is bounded. Set $M\eqdef \sup_{\omega\in\Omega}|\phi(\omega)|<\infty$ and $K\eqdef[-M,M]$.
Using the elementary fact
$|\ind_{(0,\infty)}(t-\alpha)-\ind_{(0,\infty)}(t)|\le \ind_{K}(t)$ for all $|\alpha|\le M,\,t\in\mathbb R$,
we obtain the uniform bound
\begin{equation}\label{eq:diff-final}
\big|A_n(\sigma(\omega))-A_n(\omega)\big|
 \le \frac1n+\frac{1}{n} \sum_{j=1}^{n} \ind_{K}(S_j(\omega)).
\end{equation}
By Theorem~\ref{thm:A-unified} (case (2)), we have that
\[
\lim_{n\to\infty}\frac{1}{n} \sum_{j=1}^{n} \ind_{K}\big(S_j(\omega)\big)=0\quad\text{for $\PP$-a.e.\ }\omega\in \Omega.
\]
Together with~\eqref{eq:diff-final}, this gives
$
\lim_{n\to\infty}|A_n(\sigma(\omega))-A_n(\omega)|=0$ for $\PP$-a.e.~$\omega\in\Omega$. Hence $C(\sigma(\omega))=C(\omega)$ for $\mathbb P$-a.e.\ $\omega\in \Omega$. Ergodicity of $(\Omega,\mathscr F,\PP,\sigma)$ makes $C$ almost surely constant. Write $C\equiv C(\omega)\leq 1$ and assume that $C<1$.
Choose $\alpha$ with $C<\alpha<1$.

Since $\limsup_{n\to\infty}A_n=C<\alpha$ almost surely,
$\ind_{\{A_n\geq\alpha\}}\to0$ almost surely. Hence, by dominated
convergence, $\lim_{n\to\infty}P(A_n\geq\alpha)=0$.
On the other hand, by  Theorem~\ref{prop:skew-translation-ergi-arcsine},
the sequence $A_n$ satisfies the fluctuation law, $\liminf_{n\to\infty} \PP(A_n\leq \alpha)<1$ for every $\alpha\in(0,1)$. This implies that
$\limsup_{n\to\infty}\mathbb P(A_n>\alpha)>0$
which contradicts the previous null limit. Therefore the assumption $C<1$ is false, and we must have $\limsup_{n\to\infty} A_n(\omega)=1$ for $\mathbb P$-a.e.~$\omega \in \Omega$.
Now, applying the open-half-line version of the comparison
in~\eqref{eq:shift-halfline} with $s=\kappa-y$, we obtain
\[
\left|
\frac{1}{n}\sum_{j=0}^{n-1}
\ind_{(\kappa,\infty)}\big(y+S_j(\omega)\big)
-A_n(\omega)
\right|
\longrightarrow0
\qquad\text{for $\mathbb P$-a.e.\ }\omega\in\Omega.
\]
Consequently,
\[
\limsup_{n\to\infty}\frac{1}{n}\sum_{j=0}^{n-1}
\ind_{(\kappa,\infty)}\big(y+S_j(\omega)\big)
=
\limsup_{n\to\infty}A_n(\omega)
=1 \quad \text{for $\mathbb P$-a.e.\ $\omega\in\Omega$.} \qedhere
\]
\end{proof}
}

%% file: Sec10.tex
\section{Proof of Proposition~\ref{maincor:skew-flow-preserving-orientation} and Corollary~\ref{cor:generalized-crovisier}}

\label{ss:proof-cor-skew-flow}

\subsection{Proof of Proposition~\ref{maincor:skew-flow-preserving-orientation}}

Consider $F_{\varphi,\phi}$ as in~\eqref{eq:skew-flow}, with Morse--Smale flow $\varphi\colon\mathbb R\times M\to M$ on the compact one-dimensional manifold $M$ and H\"older continuous $\phi\colon\Omega\to\mathbb R$ satisfying $\mathbb E[\phi]=0$. Let $J=(p,q)$ join consecutive equilibria of $\varphi$ and fix $\theta\in J$. Define
$\hat H\colon\Omega\times\mathbb R\to\Omega\times J$,
$\hat H(\omega,t)=(\omega,\varphi(t,\theta))$.
The map $\hat{H}$ is a homeomorphism satisfying $\hat{H} \circ T_\phi = F_{\varphi,\phi} \circ \hat{H}$, where $T_\phi$ is given by~\eqref{skew-translation0}. The map $\hat{h}\colon\mathbb{R}\to J$, $\hat{h}(t)=\varphi(t,\theta)$, is also a homeomorphism, and $g_\omega=\hat{h}^{-1}\circ f_\omega\circ\hat{h}$, where $g_\omega(y)=y+\phi(\omega)$ and $f_\omega$ is given by~\eqref{eq:skew-flow}. Thus $h=\hat{h}^{-1}$ satisfies Proposition~\ref{prop:arcsine-to-conjugateA}.

%}

Here $\phi$ satisfies~\rmref{C1}. Since $h$ is a homeomorphism, it maps compact subsets of $J$ to compact subsets of $\mathbb R$, and it carries $J_0(\gamma)=(p,\gamma)$ and $J_1(\gamma)=(\gamma,q)$ to the two half-lines (interchanged if $h$ is decreasing). Since $\phi$ satisfies the assumptions of Theorem~\ref{thm:A-unified} (case (2)) and Corollary~\ref{cor:lim1}, we have:
\begin{enumerate}[label=$\mathrm{(\roman*)}$]
    \item for every $x\in J$ and every compact set $K\subset J$,
    \[
\lim_{n\to\infty} \frac{1}{n}\sum_{j=0}^{n-1} \ind_{K}\big(f_\omega^j(x)\big)=0 \quad \text{for $\mathbb{P}$-a.e.~$\omega\in \Omega$.}
\]
    \item for every $\gamma,x \in J$
\begin{equation*}
\limsup_{n\to\infty} \frac{1}{n}\sum_{j=0}^{n-1}\ind_{{J}_i(\gamma)}\big(f^{j}_\omega(x)\big) = 1 \quad \text{for $\mathbb{P}$-a.e.~$\omega\in \Omega$, and $i=0,1$.}
\end{equation*}
\end{enumerate}
Fix $\gamma,x\in J$. We work in the compact interval $\overline J$ and identify it affinely with $[0,1]$ when applying Proposition~\ref{prop:limit_set_from_KD}.
By (i), condition~\eqref{OT} (and thus~\eqref{OTT}) holds, and hence Proposition~\ref{prop:limit_set_from_KD} implies that
$\mathcal{L}(\omega,x)\subset \{\lambda\delta_p+(1-\lambda)\delta_q:\lambda\in[0,1]\}$.
By (ii), for $\PP$-a.e.~$\omega\in \Omega$, there exists a subsequence of integers $\{n_k\}_{k\geq 1}$ such that $\lim_{k\to\infty} A_{n_k}(\omega) = 1$ where $A_n(\omega)=n^{-1}\sum_{j=0}^{n-1}\ind_{J_i(\gamma)}(f^j_\omega(x))$.
Consider the corresponding empirical-measure subsequence $\{e_{n_k}^{\omega}\}_{k\geq 1}$. By weak$^*$ compactness of the  probability measures on $\overline J$, a subsequence, again indexed by $\{n_k\}_{k\geq1}$, converges to some $\nu\in\mathcal{L}(\omega,x)$. Hence $\nu = \lambda\delta_p + (1-\lambda)\delta_q$ for some $\lambda \in [0,1]$.

Note that
$A_n(\omega)=e^{\omega}_n(J_i(\gamma))$. If $i=0$, since  $J_0(\gamma)=(p,\gamma) \subset [p,\gamma]$, we have that
$$
1=\lim_{k\to \infty} A_{n_k}(\omega) \leq  \limsup_{k\to\infty} e^{\omega}_{n_k}([p,\gamma]) \leq \nu([p,\gamma])=\nu(\{p\})=\lambda.
$$
Hence $\lambda=1$. This shows that the limit measure is $\nu = \delta_p$. Thus $\delta_p \in \mathcal{L}(\omega,x)$.
Otherwise, if $i=1$, $J_1(\gamma)=(\gamma,q)\subset [\gamma,q]$ and thus,
$$
1=\lim_{k\to \infty} A_{n_k}(\omega)  \leq\limsup_{k\to\infty} e^{\omega}_{n_k}([\gamma,q]) \leq \nu([\gamma,q])=\nu(\{q\})=1-\lambda.
$$
Hence $\lambda=0$ and $\nu=\delta_q\in\mathcal{L}(\omega,x)$.

We have shown that for every $x\in J$,  for $\PP$-a.e.~$\omega\in \Omega$, the set of accumulation points $\mathcal{L}(\omega,x)$ contains both $\delta_p$ and $\delta_q$. By connectedness of $\mathcal{L}(\omega,x)$, it contains the entire segment joining them. This proves the second inclusion and Proposition~\ref{maincor:skew-flow-preserving-orientation} under the assumption~\rmref{C1}. Fubini\textquotesingle s theorem and the finitely many intervals between equilibria give historical behavior for $(\mathbb P\times\lbb)$-almost every point.

\subsection{Proof of Corollary~\ref{cor:generalized-crovisier}}
By Propositions~\ref{maincor:skew-flow-preserving-orientation} and~\ref{thm:extended results-ergodic conjugated}, it remains to identify the accumulation set. Write $\tilde{\mathcal L}(\omega,x)$ and $\mathcal L(\omega,x)$ for the respective limit sets of $\tilde F=\tilde F_{\varphi,\phi}$ and $F=F_{\varphi,\phi}$.

Let $r(x)=-x\pmod 1$, and write
$x_j=f_\omega^j(x)$ and $\tilde x_j=\tilde f_\omega^j(x)$
for the fiber orbits of $F_{\varphi,\phi}$ and $\tilde F_{\varphi,\phi}$,
respectively. Since the symmetry of the flow gives $r\circ \varphi_t=\varphi_t\circ r$ for every $t\in\mathbb R$, where $\varphi_t(x)=\varphi(t,x)$, and
$\tilde f_\omega=r\circ f_\omega$, an induction yields $\tilde x_j=r^j(x_j)$ for $j\ge0$. Let
\[
e_n=\frac1n\sum_{j=0}^{n-1}\delta_{x_j},
\qquad
\tilde e_n=\frac1n\sum_{j=0}^{n-1}\delta_{\tilde x_j} \qquad \text{and}\qquad 
\operatorname{Sym}(\nu)\eqdef\frac12(\nu+r_*\nu).
\]
For $g\in C(\mathbb T)$, put $d=g-g\circ r$. Since $r^j=\mathrm{id}$ for even $j$
and $r^j=r$ for odd $j$, we obtain
\begin{equation} \label{9.1}
\int g \, d\tilde e_n- \int g \,d\operatorname{Sym}e_n
=
\frac1{2n}\sum_{j=0}^{n-1}(-1)^j d(x_j).
\end{equation}
We claim that the right-hand side tends to zero. Since the roof function
$\phi$ is bounded on the compact base, choose $N>0$ such that
$|\phi|\le N$. Moreover, since the endpoints $p,q$ are fixed by the flow and the function $(t,z) \in [-N,N] \times [p,q] \mapsto d(\varphi(t,z))-d(z)$ is continuous, for every $\varepsilon>0$ there exists a compact
set $K_\varepsilon\subset(p,q)$ such that
$|d(\varphi(t,z))-d(z)|<\varepsilon$
whenever $z\in[p,q]\setminus K_\varepsilon$ and $|t|\le N$.
Therefore, $|d(x_{j+1})-d(x_j)|
\le
\varepsilon
+
2\|d\|_\infty\,\ind_{K_\varepsilon}(x_j)$. By the vanishing compact occupation established above,
\[
\frac1n\sum_{j=0}^{n-2}
\ind_{K_\varepsilon}(x_j)\longrightarrow0
\qquad\text{for $\mathbb P$-a.e. }\omega.
\]
Thus
\[
\limsup_{n\to\infty}
\frac1n\sum_{j=0}^{n-2}|d(x_{j+1})-d(x_j)|
\le\varepsilon.
\]
Since $\varepsilon>0$ is arbitrary,
\begin{equation} \label{eq:92}
\frac1n\sum_{j=0}^{n-2}|d(x_{j+1})-d(x_j)|
\longrightarrow0.
\end{equation}
Pairing consecutive terms in the alternating sum in~\eqref{9.1}, and
using~\eqref{eq:92}, gives
\[
\int g \, d\tilde e_n-\int g \, d\operatorname{Sym}(e_n) \longrightarrow0.
\]
Applying this to a countable dense subset of $C(M)$ and intersecting the
corresponding full-measure sets, we obtain
$\tilde e_n-\operatorname{Sym}(e_n)\longrightarrow0$ in the weak$^*$ topology. Consequently,
$\tilde{\mathcal L}(\omega,x)
=
\operatorname{Sym}(\mathcal L(\omega,x))$. 
Indeed, if $e_{n_k}\to\nu$, then continuity of $\operatorname{Sym}$ gives
$\tilde e_{n_k}\to\operatorname{Sym}(\nu)$.
Conversely, if $\tilde e_{n_k}$ converges, compactness of the set of probability measures on $\mathbb{T}$ allows us, after passing to a subsequence, to assume that
$e_{n_k}\to\nu$, and the same asymptotic identity identifies the limit of
$\tilde e_{n_k}$ with $\operatorname{Sym}(\nu)$.

By Proposition~\ref{maincor:skew-flow-preserving-orientation}, $\mathcal L(\omega,x)
=
\{\lambda\delta_p+(1-\lambda)\delta_q:\lambda\in[0,1]\}$. 
Since $\operatorname{Sym}(\delta_z)
=
\mu_z
\eqdef
\frac12(\delta_z+\delta_{r(z)})$, 
we finally obtain $\tilde{\mathcal L}(\omega,x)
=
\{\lambda\mu_p+(1-\lambda)\mu_q:\lambda\in[0,1]\}$,
for every fixed $x$ between consecutive equilibria $p,q$ and for
$\mathbb P$-almost every $\omega$, as desired.

% 
% Let $r(x)=-x\pmod1$, and write $x_j$, $\tilde x_j$ for the unreflected and reflected fiber orbits starting at $x$. Symmetry of the flow gives $r\circ\varphi_t=\varphi_t\circ r$, hence $\tilde x_j=r^j(x_j)$. For $g\in C(M)$ put $d=g-g\circ r$. Their empirical measures satisfy
% \[
% \tilde e_n(g)-\tfrac12\big(e_n(g)+e_n(g\circ r)\big)
% =\frac1{2n}\sum_{j=0}^{n-1}(-1)^j d(x_j).
% \]
% The roof is bounded, say $|\phi|\leq M$. Since the endpoints $p,q$ are fixed by the flow, joint continuity gives, for each $\varepsilon>0$, a compact $K_\varepsilon\subset(p,q)$ such that
% \[
% |d(\varphi(t,z))-d(z)|<\varepsilon
% \quad\text{if }z\in(p,q)\setminus K_\varepsilon,\ |t|\leq M.
% \]
% Vanishing compact occupation therefore implies $n^{-1}\sum_{j=0}^{n-2}|d(x_{j+1})-d(x_j)|\to0$ almost surely. Pair consecutive terms in the alternating sum; a possible last term has absolute value at most $\|d\|_\infty/(2n)$. Thus $\tilde e_n(g)-\operatorname{Sym}(e_n)(g)\to0$, where $\operatorname{Sym}(\nu)=(\nu+r_*\nu)/2$.

% Apply this to a countable dense family of continuous test functions. Compactness and continuity of $\operatorname{Sym}$ give
% \[
% \tilde{\mathcal L}(\omega,x)=\operatorname{Sym}(\mathcal L(\omega,x))
% =\{\lambda\mu_p+(1-\lambda)\mu_q:\lambda\in[0,1]\},
% \qquad \mu_z=\tfrac12(\delta_z+\delta_{r(z)}),
% \]
% for every fixed $x$ between consecutive equilibria $p,q$ and almost every $\omega$, as desired.

% 

\section*{Acknowledgement}

This work is based on the second author's master's thesis at the Pontifical Catholic University of Rio de Janeiro. We thank his advisor, L.~J.~Díaz, for his guidance and support. We also thank the examination committee members, M.~Anderson, S.~Luzzatto, R.~Ruggiero, and A.~Tahzibi, for their insightful comments and suggestions.   We thank C. González-Tokman and D. Coates for their helpful comments and questions.
The first author acknowledges support from MCIN under grant PID-2023-147461NB-I00, PQ 305352/2020-2 (CNPq), and JCNE E-26/201.305/2022 (FAPERJ). The second author was supported [in part] by CAPES – Finance Code 001 and FAPERJ Nota-$10$.

\section*{Declarations}

\noindent \textbf{Conflict of Interest:} The authors declare no competing interests.

\vspace{0.2cm}

\noindent \textbf{Data Availability:} This is a theoretical study and no empirical data were used.

%% file: reference.bib
@article{AbbasiGharaeiHomburg18,
  author  = {Abbasi, Neda and Gharaei, Masoumeh
             and Homburg, Ale Jan},
  title   = {Iterated function systems of logistic maps:
             synchronization and intermittency},
  journal = {Nonlinearity},
  volume  = {31},
  number  = {8},
  pages   = {3880--3913},
  year    = {2018},
  doi     = {10.1088/1361-6544/aac637}
}

@article{DP-LIL,
 author = {Denker, Manfred and Philipp, Walter},
 title = {Approximation by {Brownian} motion for {Gibbs} measures and flows under a function},
 fjournal = {Ergodic Theory and Dynamical Systems},
 journal = {Ergodic Theory Dyn. Syst.},
 issn = {0143-3857},
 volume = {4},
 pages = {541--552},
 year = {1984},
 language = {English},
 doi = {10.1017/S0143385700002637},
 zbMATH = {3883357},
 Zbl = {0554.60077}
}

@article{PP97,
  title={The Livsic cocycle equation for compact Lie group extensions of hyperbolic systems},
  author={Parry, William and Pollicott, Mark},
  journal={Journal of the London Mathematical Society},
  volume={56},
  number={2},
  pages={405--416},
  year={1997},
  publisher={Cambridge University Press}
}

@article{talebi2022non,
  title={Non-statistical rational maps},
  author={Talebi, Amin},
  journal={Mathematische Zeitschrift},
  volume={302},
  number={1},
  pages={589--608},
  year={2022},
  publisher={Springer}
}

@InCollection{book-problems07,
 Author = {Hasselblatt, Boris},
 Title = {Problems in dynamical systems and related topics},
 BookTitle = {Dynamics, ergodic theory and geometry. Dedicated to Anatole Katok. Based on the workshop on recent progress in dynamics, Berkeley, CA, USA, from late September to early October, 2004},
 ISBN = {978-0-521-87541-7},
 Pages = {273--324},
 Year = {2007},
 Publisher = {Cambridge: Cambridge University Press},
 Language = {English},
 URL = {www.msri.org/communications/books/Book54/},
 zbMATH = {5282791},
 Zbl = {1153.37300}
}

@Article{Ino00,
 Author = {Inoue, Tomoki},
 Title = {Sojourn times in small neighborhoods of indifferent fixed points of one-dimensional dynamical systems},
 FJournal = {Ergodic Theory and Dynamical Systems},
 Journal = {Ergodic Theory Dyn. Syst.},
 ISSN = {0143-3857},
 Volume = {20},
 Number = {1},
 Pages = {241--257},
 Year = {2000},
 Language = {English},
 DOI = {10.1017/S0143385700000110},
 zbMATH = {1420109},
 Zbl = {0974.37025}
}

@Article{BlaBunimo03,
 Author = {Blank, Michael and Bunimovich, Leonid},
 Title = {Multicomponent dynamical systems: {SRB} measures and phase transitions},
 FJournal = {Nonlinearity},
 Journal = {Nonlinearity},
 ISSN = {0951-7715},
 Volume = {16},
 Number = {1},
 Pages = {387--401},
 Year = {2003},
 Language = {English},
 DOI = {10.1088/0951-7715/16/1/322},
 zbMATH = {1894893},
 Zbl = {1022.37003}
}

@Article{Keller04,
 Author = {Keller, Gerhard},
 Title = {Completely mixing maps without limit measure},
 FJournal = {Colloquium Mathematicum},
 Journal = {Colloq. Math.},
 ISSN = {0010-1354},
 Volume = {100},
 Number = {1},
 Pages = {73--76},
 Year = {2004},
 Language = {English},
 DOI = {10.4064/cm100-1-6},
 zbMATH = {2106972},
 Zbl = {1048.37038}
}

@article {HeSavage55-01law,
    AUTHOR = {Hewitt, Edwin and Savage, Leonard J.},
     TITLE = {Symmetric measures on {C}artesian products},
   JOURNAL = {Trans. Amer. Math. Soc.},
  FJOURNAL = {Transactions of the American Mathematical Society},
    VOLUME = {80},
      YEAR = {1955},
     PAGES = {470--501},
      ISSN = {0002-9947,1088-6850},
   MRCLASS = {60.0X},
  MRNUMBER = {76206},
MRREVIEWER = {C.\ Pauc},
       DOI = {10.2307/1992999},
       URL = {https://doi.org/10.2307/1992999},
}

@Article{kalikow1982t,
    AUTHOR = {Kalikow, Steven Arthur},
     TITLE = {{$T,\,T\sp{-1}$}\ transformation is not loosely {B}ernoulli},
   JOURNAL = {Ann. of Math. (2)},
  FJOURNAL = {Annals of Mathematics. Second Series},
    VOLUME = {115},
      YEAR = {1982},
    NUMBER = {2},
     PAGES = {393--409},
      ISSN = {0003-486X,1939-8980},
   MRCLASS = {28D20},
  MRNUMBER = {647812},
       DOI = {10.2307/1971397},
       URL = {https://doi.org/10.2307/1971397},
}

@article {MR3204667,
    AUTHOR = {Bufetov, A. I.},
     TITLE = {Ergodic decomposition for measures quasi-invariant under
              {B}orel actions of inductively compact groups},
   JOURNAL = {Mat. Sb.},
  FJOURNAL = {Matematicheski\u i\ Sbornik},
    VOLUME = {205},
      YEAR = {2014},
    NUMBER = {2},
     PAGES = {39--70},
      ISSN = {0368-8666,2305-2783},
   MRCLASS = {28Dxx (37A15 37A35 60A10)},
  MRNUMBER = {3204667},
MRREVIEWER = {Olena\ Karpel},
       DOI = {10.1070/sm2014v205n02abeh004371},
       URL = {https://doi.org/10.1070/sm2014v205n02abeh004371},
}

@Article{greschonig2000ergodic,
Author = {Greschonig, Gernot and Schmidt, Klaus},
 Title = {Ergodic decomposition of quasi-invariant probability measures},
 FJournal = {Colloquium Mathematicum},
 Journal = {Colloq. Math.},
 ISSN = {0010-1354},
 Volume = {84-85},
 Pages = {495--514},
 Year = {2000},
 Language = {English},
 DOI = {10.4064/cm-84/85-2-495-514},
 URL = {https://eudml.org/doc/210829},
 zbMATH = {1538160},
 Zbl = {0972.37003}
}

@Article{AarThaZwei05,
AUTHOR = {Aaronson, Jon and Thaler, Maximilian and Zweim\"uller, Roland},
     TITLE = {Occupation times of sets of infinite measure for ergodic
              transformations},
   JOURNAL = {Ergodic Theory Dynam. Systems},
  FJOURNAL = {Ergodic Theory and Dynamical Systems},
    VOLUME = {25},
      YEAR = {2005},
    NUMBER = {4},
     PAGES = {959--976},
      ISSN = {0143-3857,1469-4417},
   MRCLASS = {37A40 (28D05)},
  MRNUMBER = {2158393},
MRREVIEWER = {Stanley\ J.\ Eigen},
       DOI = {10.1017/S0143385704001051},
       URL = {https://doi.org/10.1017/S0143385704001051},
}

@book{ProbLoeve,
   Author = {Loeve, M.},
 Title = {Probability theory. {II}. 4th ed},
 FSeries = {Graduate Texts in Mathematics},
 Series = {Grad. Texts Math.},
 ISSN = {0072-5285},
 Volume = {46},
 Year = {1978},
 Publisher = {Springer, Cham},
 Language = {English},
 zbMATH = {3599198},
 Zbl = {0385.60001}
}

@Book{AaronBook97,
 Author = {Aaronson, Jon},
 Title = {An introduction to infinite ergodic theory},
 FSeries = {Mathematical Surveys and Monographs},
 Series = {Math. Surv. Monogr.},
 ISSN = {0076-5376},
 Volume = {50},
 ISBN = {0-8218-0494-4},
 Year = {1997},
 Publisher = {Providence, RI: American Mathematical Society},
 Language = {English},
 zbMATH = {1027990},
 Zbl = {0882.28013}
}

@article {TailBerHollan89,
    AUTHOR = {Berbee, H. C. P. and den Hollander, W. Th. F.},
     TITLE = {Tail triviality for sums of stationary random variables},
   JOURNAL = {Ann. Probab.},
  FJOURNAL = {The Annals of Probability},
    VOLUME = {17},
      YEAR = {1989},
    NUMBER = {4},
     PAGES = {1635--1645},
      ISSN = {0091-1798,2168-894X},
   MRCLASS = {60F20},
  MRNUMBER = {1048949},
MRREVIEWER = {Manfred\ Denker},
       URL =
              {http://links.jstor.org/sici?sici=0091-1798(198910)17:4<1635:TTFSOS>2.0.CO;2-Y&origin=MSN},
}

@book {Lev-Brow2,
    AUTHOR = {L\'evy, Paul},
     TITLE = {Processus stochastiques et mouvement brownien},
      NOTE = {Suivi d'une note de M. Lo\`eve,
              Deuxi\`eme \'edition revue et augment\'ee},
 PUBLISHER = {Gauthier-Villars \& Cie, Paris},
      YEAR = {1965},
     PAGES = {vi+438},
   MRCLASS = {60.00 (60.40)},
  MRNUMBER = {190953},
MRREVIEWER = {H.\ P.\ McKean, Jr.},
}

@article {Lev39,
    AUTHOR = {L\'evy, Paul},
     TITLE = {Sur certains processus stochastiques homog\`enes},
   JOURNAL = {Compositio Math.},
  FJOURNAL = {Compositio Mathematica},
    VOLUME = {7},
      YEAR = {1939},
     PAGES = {283--339},
      ISSN = {0010-437X,1570-5846},
   MRCLASS = {60.0X},
  MRNUMBER = {919},
MRREVIEWER = {J.\ L.\ Doob},
       URL = {http://www.numdam.org/item?id=CM_1940__7__283_0},
}

@book {PeYu-book-Brownian,
    AUTHOR = {M\"orters, Peter and Peres, Yuval},
     TITLE = {Brownian motion},
    SERIES = {Cambridge Series in Statistical and Probabilistic Mathematics},
    VOLUME = {30},
      NOTE = {With an appendix by Oded Schramm and Wendelin Werner},
 PUBLISHER = {Cambridge University Press, Cambridge},
      YEAR = {2010},
     PAGES = {xii+403},
      ISBN = {978-0-521-76018-8},
   MRCLASS = {60J65 (28A78 60H05 60J45 60J55 60J67)},
  MRNUMBER = {2604525},
MRREVIEWER = {Ren\'e\ L.\ Schilling},
       DOI = {10.1017/CBO9780511750489},
       URL = {https://doi.org/10.1017/CBO9780511750489},
}

@article{molinek1994asymptotic,
    AUTHOR = {Molinek, D. K.},
     TITLE = {Asymptotic measures for skew products of {B}ernoulli shifts
              with generalized north pole--south pole diffeomorphisms},
   JOURNAL = {Trans. Amer. Math. Soc.},
  FJOURNAL = {Transactions of the American Mathematical Society},
    VOLUME = {345},
      YEAR = {1994},
    NUMBER = {1},
     PAGES = {263--291},
      ISSN = {0002-9947,1088-6850},
   MRCLASS = {28D05 (58F03 58F11 60F05)},
  MRNUMBER = {1254191},
MRREVIEWER = {Manfred\ Denker},
       DOI = {10.2307/2154604},
       URL = {https://doi.org/10.2307/2154604},
}

@article{nakano2017historic,
    AUTHOR = {Nakano, Yushi},
     TITLE = {Historic behaviour for random expanding maps on the circle},
   JOURNAL = {Tokyo J. Math.},
  FJOURNAL = {Tokyo Journal of Mathematics},
    VOLUME = {40},
      YEAR = {2017},
    NUMBER = {1},
     PAGES = {165--184},
      ISSN = {0387-3870},
   MRCLASS = {37C40 (37E10 37H10)},
  MRNUMBER = {3689984},
MRREVIEWER = {Romain\ Aimino},
       DOI = {10.3836/tjm/1502179221},
       URL = {https://doi.org/10.3836/tjm/1502179221},
}

@article{LR17,
	Author = {Labouriau, I. S. and Rodrigues, A. A. P.},
	Doi = {10.1088/1361-6544/aa64e9},
	Fjournal = {Nonlinearity},
	Issn = {0951-7715},
	Journal = {Nonlinearity},
	Mrclass = {34C28 (34C37 37G20)},
	Mrnumber = {3639293},
	Number = {5},
	Pages = {1876--1910},
	Title = {On {T}akens' last problem: tangencies and time averages near heteroclinic networks},
	Url = {https://mathscinet.ams.org/mathscinet-getitem?mr=3639293},
	Volume = {30},
	Year = {2017}}

@incollection {Bonifant,
    AUTHOR = {Bonifant, Araceli and Milnor, John},
     TITLE = {Schwarzian derivatives and cylinder maps},
 BOOKTITLE = {Holomorphic dynamics and renormalization},
    SERIES = {Fields Inst. Commun.},
    VOLUME = {53},
     PAGES = {1--21},
 PUBLISHER = {Amer. Math. Soc., Providence, RI},
      YEAR = {2008},
      ISBN = {978-0-8218-4275-1},
   MRCLASS = {37C70 (37C40 37E30)},
  MRNUMBER = {2477416},
MRREVIEWER = {Mattias\ Jonsson},
}

@incollection {Rue2001,
    AUTHOR = {Ruelle, David},
     TITLE = {Historical behaviour in smooth dynamical systems},
 BOOKTITLE = {Global analysis of dynamical systems},
     PAGES = {63--66},
 PUBLISHER = {Inst. Phys., Bristol},
      YEAR = {2001},
      ISBN = {0-7503-0803-6},
   MRCLASS = {37C40 (37J40)},
  MRNUMBER = {1858471},
}

@article {Takens2008,
    AUTHOR = {Takens, Floris},
     TITLE = {Orbits with historic behaviour, or non-existence of averages},
   JOURNAL = {Nonlinearity},
  FJOURNAL = {Nonlinearity},
    VOLUME = {21},
      YEAR = {2008},
    NUMBER = {3},
     PAGES = {T33--T36},
      ISSN = {0951-7715,1361-6544},
   MRCLASS = {37C40 (37A99)},
  MRNUMBER = {2396607},
MRREVIEWER = {Kazuhiro\ Sakai},
       DOI = {10.1088/0951-7715/21/3/T02},
       URL = {https://doi.org/10.1088/0951-7715/21/3/T02},
}

@Article{CoaLuza23,
 Author = {Coates, Douglas and Luzzatto, Stefano},
 Title = {Persistent non-statistical dynamics in one-dimensional maps},
 FJournal = {Communications in Mathematical Physics},
 Journal = {Commun. Math. Phys.},
 ISSN = {0010-3616},
 Volume = {405},
 Number = {4},
 Pages = {34},
 Note = {Id/No 102},
 Year = {2024},
 Language = {English},
 DOI = {10.1007/s00220-024-04957-0},
 zbMATH = {7831314}
}

@Misc{CoaMelTale24,
 Author = {Coates, Douglas and Melbourne, Ian and Talebi, Amin},
 Title = {Natural measures and statistical properties of non-statistical maps with multiple neutral fixed points},
 Year = {2024},
 HowPublished = {Preprint, {arXiv}:2407.07286 [math.{DS}] (2024)},
 URL = {https://arxiv.org/abs/2407.07286},
 arXiv = {arXiv:2407.07286}
}

@article {ErdKac47,
    AUTHOR = {Erd\"{o}s, P. and Kac, M.},
     TITLE = {On the number of positive sums of independent random
              variables},
   JOURNAL = {Bull. Amer. Math. Soc.},
  FJOURNAL = {Bulletin of the American Mathematical Society},
    VOLUME = {53},
      YEAR = {1947},
     PAGES = {1011--1020},
      ISSN = {0002-9904},
   MRCLASS = {60.0X},
  MRNUMBER = {23011},
MRREVIEWER = {M.\ Lo\`eve},
       DOI = {10.1090/S0002-9904-1947-08928-X},
       URL = {https://doi.org/10.1090/S0002-9904-1947-08928-X},
}

@article {Thaler02,
    AUTHOR = {Thaler, Maximilian},
     TITLE = {A limit theorem for sojourns near indifferent fixed points of
              one-dimensional maps},
   JOURNAL = {Ergodic Theory Dynam. Systems},
  FJOURNAL = {Ergodic Theory and Dynamical Systems},
    VOLUME = {22},
      YEAR = {2002},
    NUMBER = {4},
     PAGES = {1289--1312},
      ISSN = {0143-3857,1469-4417},
   MRCLASS = {37E05 (28D05 37A05 37A50 60F05)},
  MRNUMBER = {1926288},
MRREVIEWER = {Pawe\l \ G\'{o}ra},
       DOI = {10.1017/S0143385702000573},
       URL = {https://doi.org/10.1017/S0143385702000573},
}

@article {Crovisier2020,
    AUTHOR = {Crovisier, Sylvain and Yang, Dawei and Zhang, Jinhua},
     TITLE = {Empirical measures of partially hyperbolic attractors},
   JOURNAL = {Comm. Math. Phys.},
  FJOURNAL = {Communications in Mathematical Physics},
    VOLUME = {375},
      YEAR = {2020},
    NUMBER = {1},
     PAGES = {725--764},
      ISSN = {0010-3616,1432-0916},
   MRCLASS = {37C40 (37C70 37D25 37D35)},
  MRNUMBER = {4082180},
MRREVIEWER = {Zeya\ Mi},
       DOI = {10.1007/s00220-019-03668-1},
       URL = {https://doi.org/10.1007/s00220-019-03668-1},
}

@book {Bow75,
    AUTHOR = {Bowen, Rufus},
     TITLE = {Equilibrium states and the ergodic theory of {A}nosov
              diffeomorphisms},
    SERIES = {Lecture Notes in Mathematics},
    VOLUME = {Vol. 470},
 PUBLISHER = {Springer-Verlag, Berlin-New York},
      YEAR = {1975},
     PAGES = {i+108},
   MRCLASS = {58F10 (28A65)},
  MRNUMBER = {442989},
MRREVIEWER = {L.\ A.\ Bunimovich},
}

@article {Takens-eyebowe94,
	AUTHOR = {Takens, Floris},
	TITLE = {Heteroclinic attractors: time averages and moduli of
	topological conjugacy},
	JOURNAL = {Bol. Soc. Brasil. Mat. (N.S.)},
	FJOURNAL = {Boletim da Sociedade Brasileira de Matem\'{a}tica. Nova
	S\'{e}rie},
	VOLUME = {25},
	YEAR = {1994},
	NUMBER = {1},
	PAGES = {107--120},
	ISSN = {0100-3569},
	MRCLASS = {58F11 (58F21 58F25)},
	MRNUMBER = {1274765},
	MRREVIEWER = {Michael\ Hurley},
	DOI = {10.1007/BF01232938},
	URL = {https://doi.org/10.1007/BF01232938},
}

@article {Andde22,
	AUTHOR = {Andersson, Martin and Guih\'{e}neuf, Pierre-Antoine},
	TITLE = {Historic behaviour vs. physical measures for irrational flows
	with multiple stopping points},
	JOURNAL = {Adv. Math.},
	FJOURNAL = {Advances in Mathematics},
	VOLUME = {409},
	YEAR = {2022},
	PAGES = {Paper No. 108626, 71},
	ISSN = {0001-8708,1090-2082},
	MRCLASS = {37E35 (11K50 37A30 37C40)},
	MRNUMBER = {4473634},
	MRREVIEWER = {Konstantin\ Athanassopoulos},
	DOI = {10.1016/j.aim.2022.108626},
	URL = {https://doi.org/10.1016/j.aim.2022.108626},
}

@article {CarVaran22,
	AUTHOR = {Carvalho, Maria and Varandas, Paulo},
	TITLE = {Genericity of historic behavior for maps and flows},
	JOURNAL = {Nonlinearity},
	FJOURNAL = {Nonlinearity},
	VOLUME = {34},
	YEAR = {2021},
	NUMBER = {10},
	PAGES = {7030--7044},
	ISSN = {0951-7715,1361-6544},
	MRCLASS = {37C10 (37A30 37C40 37D25 37D30)},
	MRNUMBER = {4313703},
	MRREVIEWER = {Jinhua\ Zhang},
	DOI = {10.1088/1361-6544/ac1f77},
	URL = {https://doi.org/10.1088/1361-6544/ac1f77},
}

@article {Thaler80,
    AUTHOR = {Thaler, Maximilian},
     TITLE = {Estimates of the invariant densities of endomorphisms with
              indifferent fixed points},
   JOURNAL = {Israel J. Math.},
  FJOURNAL = {Israel Journal of Mathematics},
    VOLUME = {37},
      YEAR = {1980},
    NUMBER = {4},
     PAGES = {303--314},
      ISSN = {0021-2172},
   MRCLASS = {28D05 (10K10)},
  MRNUMBER = {599464},
MRREVIEWER = {F.\ Schweiger},
       DOI = {10.1007/BF02788928},
       URL = {https://doi.org/10.1007/BF02788928},
}

@article {Thaler83,
    AUTHOR = {Thaler, Maximilian},
     TITLE = {Transformations on {$[0,\,1]$} with infinite invariant
              measures},
   JOURNAL = {Israel J. Math.},
  FJOURNAL = {Israel Journal of Mathematics},
    VOLUME = {46},
      YEAR = {1983},
    NUMBER = {1-2},
     PAGES = {67--96},
      ISSN = {0021-2172},
   MRCLASS = {28D05},
  MRNUMBER = {727023},
MRREVIEWER = {Nathaniel\ Friedman},
       DOI = {10.1007/BF02760623},
       URL = {https://doi.org/10.1007/BF02760623},
}

@article {Stefano-comu,
    AUTHOR = {Coates, Douglas and Luzzatto, Stefano and Muhammad, Mubarak},
     TITLE = {Doubly intermittent full branch maps with critical points and
              singularities},
   JOURNAL = {Comm. Math. Phys.},
  FJOURNAL = {Communications in Mathematical Physics},
    VOLUME = {402},
      YEAR = {2023},
    NUMBER = {2},
     PAGES = {1845--1878},
      ISSN = {0010-3616,1432-0916},
   MRCLASS = {37E05 (37A50)},
  MRNUMBER = {4627333},
       DOI = {10.1007/s00220-023-04766-x},
       URL = {https://doi.org/10.1007/s00220-023-04766-x},
}

@article{BS00,
    AUTHOR = {Barreira, Luis and Schmeling, J\"org},
     TITLE = {Sets of ``non-typical'' points have full topological entropy
              and full {H}ausdorff dimension},
   JOURNAL = {Israel J. Math.},
  FJOURNAL = {Israel Journal of Mathematics},
    VOLUME = {116},
      YEAR = {2000},
     PAGES = {29--70},
      ISSN = {0021-2172,1565-8511},
   MRCLASS = {37C45 (37A25 37D35 37F99)},
  MRNUMBER = {1759398},
MRREVIEWER = {Benoit\ Saussol},
       DOI = {10.1007/BF02773211},
       URL = {https://doi.org/10.1007/BF02773211},
}

@article{FFW01,
    AUTHOR = {Fan, Ai-Hua and Feng, De-Jun and Wu, Jun},
     TITLE = {Recurrence, dimension and entropy},
   JOURNAL = {J. London Math. Soc. (2)},
  FJOURNAL = {Journal of the London Mathematical Society. Second Series},
    VOLUME = {64},
      YEAR = {2001},
    NUMBER = {1},
     PAGES = {229--244},
      ISSN = {0024-6107,1469-7750},
   MRCLASS = {37B40 (28A80 37A35 37C45 37D35)},
  MRNUMBER = {1840781},
MRREVIEWER = {J\'er\^ome\ Buzzi},
       DOI = {10.1017/S0024610701002137},
       URL = {https://doi.org/10.1017/S0024610701002137},
}

@article{BKNRS20,
    AUTHOR = {Barrientos, Pablo G. and Kiriki, Shin and Nakano, Yushi and
              Raibekas, Artem and Soma, Teruhiko},
     TITLE = {Historic behavior in nonhyperbolic homoclinic classes},
   JOURNAL = {Proc. Amer. Math. Soc.},
  FJOURNAL = {Proceedings of the American Mathematical Society},
    VOLUME = {148},
      YEAR = {2020},
    NUMBER = {3},
     PAGES = {1195--1206},
      ISSN = {0002-9939,1088-6826},
   MRCLASS = {37C05 (37C20 37C25 37C29 37C70)},
  MRNUMBER = {4055947},
MRREVIEWER = {C.\ M.\ Carballo},
       DOI = {10.1090/proc/14809},
       URL = {https://doi.org/10.1090/proc/14809},
}

@Article{Ser20,
 Author = {Sera, Toru},
 Title = {Functional limit theorem for occupation time processes of intermittent maps},
 FJournal = {Nonlinearity},
 Journal = {Nonlinearity},
 ISSN = {0951-7715},
 Volume = {33},
 Number = {3},
 Pages = {1183--1217},
 Year = {2020},
 Language = {English},
 DOI = {10.1088/1361-6544/ab5ceb},
 zbMATH = {7165635},
 Zbl = {1431.60036}
}

@misc{liu2021brownian,
  author       = {Christian Liu},
  title        = {Brownian Motion as the Limiting Distribution of Random Walks},
  year         = {2021},
  note         = {University of Chicago REU paper},
  url          = {http://math.uchicago.edu/~may/REU2021/REUPapers/Liu,Christian.pdf},
  institution  = {University of Chicago},
  month        = {August},
}

@Article{Rud88,
 Author = {Rudolph, Daniel J.},
 Title = {Asymptotically {Brownian} skew products give non-loosely {Bernoulli} {K}- automorphisms},
 FJournal = {Inventiones Mathematicae},
 Journal = {Invent. Math.},
 ISSN = {0020-9910},
 Volume = {91},
 Number = {1},
 Pages = {105--128},
 Year = {1988},
 Language = {English},
 DOI = {10.1007/BF01404914},
 URL = {https://eudml.org/doc/143533},
 zbMATH = {4069892},
 Zbl = {0655.58034}
}

@article {SeraYano19,
    AUTHOR = {Sera, Toru and Yano, Kouji},
     TITLE = {Multiray generalization of the arcsine laws for occupation
              times of infinite ergodic transformations},
   JOURNAL = {Trans. Amer. Math. Soc.},
  FJOURNAL = {Transactions of the American Mathematical Society},
    VOLUME = {372},
      YEAR = {2019},
    NUMBER = {5},
     PAGES = {3191--3209},
      ISSN = {0002-9947,1088-6850},
   MRCLASS = {60F05 (28D05 37A40)},
  MRNUMBER = {3988607},
       DOI = {10.1090/tran/7755},
       URL = {https://doi.org/10.1090/tran/7755},
}

@article{PP90,
    AUTHOR = {Parry, William and Pollicott, Mark},
     TITLE = {Zeta functions and the periodic orbit structure of hyperbolic
              dynamics},
   JOURNAL = {Ast\'erisque},
  FJOURNAL = {Ast\'erisque},
      YEAR = {1990},
     PAGES = {268},
      ISSN = {0303-1179,2492-5926},
   MRCLASS = {58F20 (58F11 58F15)},
  MRNUMBER = {1085356},
MRREVIEWER = {Nicola\u i\ T.\ A.\ Haydn},
}

@article{Liv72,
doi = {10.1070/IM1972v006n06ABEH001919},
url = {https://dx.doi.org/10.1070/IM1972v006n06ABEH001919},
year = {1972},
month = {dec},
publisher = {},
volume = {6},
number = {6},
pages = {1278},
author = {A N Livšic},
title = {COHOMOLOGY OF DYNAMICAL SYSTEMS},
journal = {Mathematics of the USSR-Izvestiya}
}

@article{JM00,
  AUTHOR = {Ji, Chuanshu and Molinek, Donna},
     TITLE = {Asymptotic physical measures for nearly {B}rownian skew
              products},
   JOURNAL = {Ergodic Theory Dynam. Systems},
  FJOURNAL = {Ergodic Theory and Dynamical Systems},
    VOLUME = {20},
      YEAR = {2000},
    NUMBER = {2},
     PAGES = {473--481},
      ISSN = {0143-3857,1469-4417},
   MRCLASS = {37A50 (28D15 37D15 62F10)},
  MRNUMBER = {1756980},
MRREVIEWER = {Manfred\ Denker},
       DOI = {10.1017/S0143385700000225},
       URL = {https://doi.org/10.1017/S0143385700000225},
}

@article{BNRR22,
    AUTHOR = {Barrientos, Pablo G. and Nakano, Yushi and Raibekas, Artem and
              Roldan, Mario},
     TITLE = {Topological entropy and {H}ausdorff dimension of irregular
              sets for non-hyperbolic dynamical systems},
   JOURNAL = {Dyn. Syst.},
  FJOURNAL = {Dynamical Systems. An International Journal},
    VOLUME = {37},
      YEAR = {2022},
    NUMBER = {2},
     PAGES = {181--210},
      ISSN = {1468-9367,1468-9375},
   MRCLASS = {37B40 (37C40 37C45)},
  MRNUMBER = {4430566},
MRREVIEWER = {Zhiming\ Li},
       DOI = {10.1080/14689367.2022.2031890},
       URL = {https://doi.org/10.1080/14689367.2022.2031890},
}

@article{KS17,
  title={Takens' last problem and existence of non-trivial wandering domains},
  author={Kiriki, Shin and Soma, Teruhiko},
  journal={Advances in Mathematics},
  volume={306},
  pages={524--588},
  year={2017},
  publisher={Elsevier}
}

@article{Bar22,
  title={Historic wandering domains near cycles},
  author={Barrientos, Pablo G},
  journal={Nonlinearity},
  volume={35},
  number={6},
  pages={3191},
  year={2022},
  publisher={IOP Publishing}
}

@article{BB23,
  title={Emergence of wandering stable components},
  author={Berger, Pierre and Biebler, Sebastien},
  journal={Journal of the American Mathematical Society},
  volume={36},
  number={2},
  pages={397--482},
  year={2023}
}
